\documentclass[10pt,reqno]{amsart}
\usepackage{amsmath,amsfonts,amssymb,amscd,amsthm,amsbsy,bbm, epsf,calc,enumerate,color,graphicx,verbatim,cite}
\usepackage{color}
\usepackage{datetime}

\newcounter{hours}\newcounter{minutes}

\makeatletter
\newtheorem*{rep@theorem}{\rep@title}
\newcommand{\newreptheorem}[2]{%
\newenvironment{rep#1}[1]{%
 \def\rep@title{#2 \ref{##1}}%
 \begin{rep@theorem}}%
 {\end{rep@theorem}}}
\makeatother

\theoremstyle{theorem}
\newtheorem{definition}{Theorem}[section]
\newtheorem{thm}{Theorem}[section]
\newreptheorem{thm}{Theorem}

\newtheorem{lem}[thm]{Lemma}
\newreptheorem{lem}{Lemma}
\newtheorem{cor}[thm]{Corollary}
\newtheorem{prop}[thm]{Proposition}

\theoremstyle{definition}
\newtheorem{ex}[thm]{Example}
\newtheorem*{Claim}{Claim}
\newtheorem{DEF}[definition]{Definition}
\newtheorem{hyp}{Assumption}

\theoremstyle{remark}                  
\newtheorem{rem}[thm]{Remark}

\def\R{{\mathbb R}}

\def\real{{\mathbb R}}

\def\integer{{\mathbb Z}}

\def\e{\varepsilon}

\newcommand{\floor}[1]{\lfloor #1 \rfloor}

\DeclareMathOperator*{\argmin}{\arg\!\min}
\def\liminf{\mathop{\lim\,\inf}\limits}%
\def\limsup{\mathop{\lim\,\sup}\limits}%
\def\argmin{\mathop{\arg\,\min}\limits}%
\numberwithin{equation}{section}

\begin{document}

\title{Dynamic Stability of Equilibrium Capillary Drops}
\author{ William M. Feldman and Inwon C. Kim}
\address{Dept. of Mathematics, UCLA.}
\email{ wfeldman10@math.ucla.edu,  ikim@math.ucla.edu}
 \thanks{Both authors have been partially supported by NSF  DMS-0970072.}
\begin{abstract}
We investigate a model for contact angle motion of quasi-static capillary drops resting on a horizontal plane. We prove global in time existence and long time behavior (convergence to equilibrium) in a class of star-shaped initial data for which we show that topological changes of drops can be ruled out for all times. Our result applies to any drop which is initially star-shaped with respect to a a small ball inside the drop, given that the volume of the drop is sufficiently large. For the analysis, we combine geometric arguments based on the moving-plane type method with energy dissipation methods based on the formal gradient flow structure of the problem.
\end{abstract}

\maketitle
\tableofcontents

\section{Introduction}
We consider the function $u(x,t): \R^N\times [0,\infty) \to[0,\infty)$ solving the following free boundary problem
\begin{equation}\label{eqn: CLMV}\tag{P-V}
\left \{ \begin{array}{lll}
         -\Delta u(x,t)  = \lambda(t) & \text{ in }&  \{u>0\},\\ \\
         \mathcal{V} = F(|Du|)& \text{ on } & \partial\{u>0\},
         \end{array}\right.
\end{equation}
where the Lagrange multiplier $\lambda(t)$ is associated to the volume constraint 
\begin{equation}\label{volume}
\int u(\cdot,t) dx\equiv V.
\end{equation}
(see Figure 1.)
\begin{figure}[t]
\begin{center}
\includegraphics[scale=0.7]{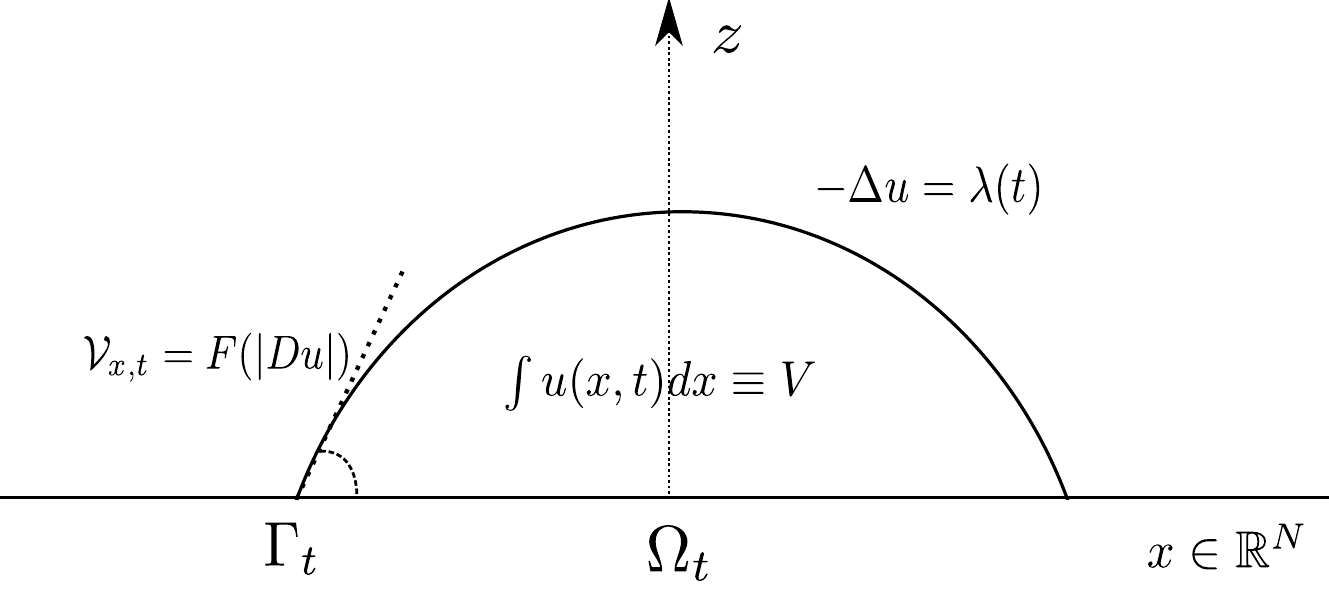}
\end{center}
\caption{Description of the problem}
\label{fig: problem}
\end{figure}

$\mathcal{V}=\mathcal{V}_{x,t}$ given above denotes the (outward) normal velocity of the free boundary $\partial\{u>0\}$ at $(x,t)$.  The prescribed normal velocity $F:\R\to \R$ is a continuous and strictly increasing function with $F(1)=0$, where $1$ denotes the equilibrium contact angle of the profile $u$ with the surface. See Section 2 for a rigorous formulation of weak solutions for \eqref{eqn: CLMV}.
\medskip

In two dimensions, \eqref{eqn: CLMV} is a simplified model to describe a liquid droplet resting on a flat surface. The model is quasi-static, the speed of the contact line is much slower than the capillary relaxation time: see below for more discussion on the derivation of the model (also see e.g. \cite{AD}, \cite{greenspan}). In this context, $u$ denotes the height of the drop. We call $\Omega_t := \{u(\cdot,t)>0\}$ the {\it wetted set} and the free boundary $\Gamma_t := \partial \{u(\cdot,t)>0\}$ the {\it contact line} between the liquid and the flat surface. 
 The function 
$$F: [0,+\infty) \to \mathbb{R}$$ 
represents the dependence of the normal velocity of the contact line $\Gamma_t$ on the contact angle.  The deviation of the contact angle $|\nabla u|$ on $\Gamma_t(u)$ from the equilibrium value $1$ is responsible for the motion of the contact line. The relation between the velocity of the contact line and the contact angle $|\nabla u |$ is not a settled issue, and various velocity laws have been proposed and studied (e.g. \cite{blake}, \cite{greenspan}, \cite{hocking}) in the fluids literature. This motivates us to study our problem \eqref{eqn: CLMV} with the most general possible $F$, at least for the study of geometric properties of solutions.
\medskip

 In section 4 and 5, we will focus on $F(|Du|) = |Du|^2 -1$ for simplicity of presentation, since the associated energy structure is simplest for this choice of $F$. The approach we present nevertheless will apply to the general $F$, as we will discuss in section 4.

\medskip

Our aim is to address the long time behavior of the solution of \eqref{eqn: CLMV}. Note that the equilibrium solution $v(x)$ for \eqref{eqn: CLMV} solves 
\begin{equation}\label{eqn: EQ}\tag{EQ}
\left \{ \begin{array}{lll}
         -\Delta v(x) = \lambda &\text{ in }& \{v>0\}, \\ \\
|Dv|=1 &\text{ on } &\partial\{v>0\}. 
         \end{array}\right.
\end{equation}
Assuming that $v$ is $C^{2,\alpha}$, the classical result of Serrin \cite{Serrin71} yields that $v$ is radial. We would like to show the dynamic stability of the above result, in the context of our model. More precisely, we would like to show that the solutions of \eqref{eqn: CLMV} uniformly converges to the round drop given by \eqref{eqn: EQ}. Such a stability result is, to the best of the authors'  knowledge, new in any model describing evolution of volume-preserving drops. (see \cite{aftalon} and \cite{brandolini} for stability results on stationary problems).
\medskip

In general, the behavior of evolving drops is highly difficult to observe due to the diverse possibility of topological changes the drop may go through, and due to the generic non-unique nature of the evolution (see \cite{AD} for examples). On the other hand, it is reasonable to expect that the drops stay simply connected if its initial shape is star-shaped with respect to some ball inside of the drop.  This is what we confirm in terms of {\it $\rho$-reflection} (see Definition~\ref{rho_reflection}) which measures star-shapedness of a drop with respect to the reflection comparison. Numerical simulations suggest that initially convex drops may develop an wedge while it contracts (see e.g. \cite{glasner}), and thus our definition, which allows wedges on the free boundary, seems to be appropriate to address the global-time behavior of the drop. 
\medskip

\begin{thm}[Theorem 5.1]\label{main}
Suppose $u_0$ satisfies $\rho$-reflection and $\int u_0 = V$, with $0<\rho \leq \frac{1}{10}V^{\tfrac{1}{N+1}}$. Then the following holds:
\begin{itemize}
\item[(a)] there exists an ``energy" solution (see the definition in section 5)  $u(x,t)$ of \eqref{eqn: CLMV} with initial data $u_0$ which stays star-shaped for all times and is H\"{o}lder continuous in space and time variable.
\item[(b)]  Any energy solution of \eqref{eqn: CLMV} with initial data $u_0$ converges uniformly modulo translation to the unique equilibrium solution $v$ of \eqref{eqn: EQ} with volume $V$. Moreover the set $\{u(\cdot,t)>0\}$ converges, modulo translation, to a ball  in Hausdorff topology.
\end{itemize}
\end{thm}
 \medskip

   \begin{rem}\label{boost}
It should be pointed out that the result is not a perturbative one. In fact for any initial base of the drop $\{u_0>0\}$ which is star-shaped with respect to some ball inside it (e.g. a star or a triangle), our theorem implies given that the drop has sufficiently large volume. Indeed the proof is based on geometric, moving-plane method-type arguments (see section 3.6). As a consequence we obtain explicit estimates on the size of the parameter $\rho$.     \end{rem}
 
 \begin{rem}
 It may be possible to obtain the rate of the convergence by further investigation of the formal energy dissipation inequality carried out in the beginning of section 4. Unfortunately in our setting $u$ is not regular enough for the calculation to go through.
 \end{rem}

 \begin{rem}
   Corresponding results were proved for convex solutions of the volume preserving, anisotropic mean curvature flow by Belletini et. al. \cite{BCCN09}, but their approach strongly depends on the level set formulation of the problem as well as the convexity preserving and regularizing feature of the mean curvature flow. 
\end{rem}

 \medskip

  It is an open question whether  there exist other geometric properties besides $\rho$-reflection that are preserved throughout the evolution \eqref{eqn: CLMV}. In fact, finding such a geometric property is one of the main novel features in our result. Let us point out that, in particular, it is unknown whether the convexity of the drop is preserved in the system \eqref{eqn: CLMV}.

\medskip

\subsection{The Model}\label{sec: The Model}
The energy of a static droplet which occupies a subset of $\mathbb{R}^{N+1}$ resting on the plane $\{x_{N+1}=0\}$ is given by
\begin{equation}\label{eq: full energy}
J(E) = P(E\cap \{x_{N+1}>0\})+\int_{E\cap \{x_{N+1}=0\}} \sigma d\mathcal{H}^{N}+ \int_{E} \rho g x_{N+1}dx.
\end{equation}
Here $P(F)$ is the perimeter of a subset $F$ of $\mathbb{R}^N$ defined at least for bounded domains with Lipschitz boundaries.  Stationary droplets correspond to minimizers of the energy $J$ when the volume $|E| = V$ is fixed. The first term is the energy due to surface tension of the free surface of the droplet.  The second term is the energy due to the adhesion between the droplet and the surface it rests on, $\sigma$ is called the relative adhesion coefficient and in general it can be spatially dependent.  In this paper however we will assume $\sigma$ is constant.  The third is a gravitational potential energy, $\rho$ is the mass per unit volume of water and $g$ the acceleration due to gravity.  We will neglect the effects of gravity and assume that 
$$ g = 0$$
one expects this to be a reasonable approximation when the total volume is small.

We will make the further assumption that our droplets are the region under the graph of a function $u:\mathbb{R}^N \to [0,+\infty)$,
$$ E = \{ (x',x_{N+1}) \in \mathbb{R}^{N}\times [0,+\infty): x_{N+1}<u(x')\}.$$
Now, thinking of the droplet in terms of the function $u$, its height profile, the energy simplifies to become,
\begin{equation}\label{eq: function energy}
J(u) = \int _{u>0} \sqrt{1+|Du|^2} dx+\sigma |\{u>0\}|.
\end{equation}
The Euler-Lagrange equation corresponding to fixed volume energy minimizers involves the mean curvature of the free surface.  Finally, to simplify our analysis, we linearize \eqref{eq: function energy} and taking $\sigma = -1/2$ for concreteness we get
\begin{equation}
\mathcal{J}(u) = \int _{\mathbb{R}^N} |Du|^2 dx+|\{u>0\}|.
\end{equation}
Formally our problem \eqref{eqn: CLMV} can be written as the gradient flow associated with the energy $\mathcal{J}(u)$ (see the heuristic argument in section 4). Based on this structure a minimizing movement scheme in the spirit of \cite{Mielke} was carried out for the general energy $J(E)$ in \cite{AD} to derive a weak solution in the continuum limit, where the solution is given as an evolution of  entire drop (sets of finite perimeters in $\R^{N+1}$). In our paper we show that, in the graph setting with the energy $\mathcal{J}(u)$, for a large class of initial data (see Remark~\ref{boost})  we are able to give a point-wise description on the movement of the contact line $\partial\{u>0\}$ using the notion of viscosity solutions.

\subsection{ Main challenges and strategy}

\medskip

The major difficulty of this problem is the lack of strong compactness (either in the associated energy given above or in the problem itself)  which hinders most approximation arguments without geometric restrictions. In particular the standard comparison principle fails due to $\lambda(t)$, or, more fundamentally, the volume preserving nature of the problem.  We emphasize that one cannot simplify the problem by replacing $\lambda$ fixed in the problem, especially when one is interested in the long-time behavior of the solutions: in fact, one can show that, if we fix $\lambda$ then most drops either vanish in finite time or grow to infinity. In comparison to other volume-preserving problems such as the volume-preserving mean curvature flow  or the Hele-Shaw flow with surface tension, the evolution of our problem \eqref{eqn: CLMV}  is driven by the contact angle, not by higher-order regularizing terms such as the mean curvature. This aspect of the problem challenges, for example,  the diffusive-interface approach which are present in many other problems.

\medskip

To get around aforemented difficulties and establish the existence of solutions for \eqref{eqn: CLMV}, we utilize the energy structure of the problem. Formally, 
solutions of the problem \eqref{eqn: CLMV} are gradient flows of an energy associated to the capillary drop problem.  This formulation of the problem is best understood from the perspective of the wetted set. The energy of a capillary drop of volume $V$ resting above a domain $\Omega \subset \mathbb{R}^N$ is taken to be
\begin{equation}\label{eq: energy}
\mathcal{J}(\Omega) := \int_\Omega |Du|^2 +|\Omega|
\end{equation}
where, for now, $u$ is the height profile above the wetted set $\Omega$ and is defined by
\begin{equation}
u := \argmin \{u \in H^1_0(\Omega): \int_{\mathbb{R}^N} u = V \}.
\end{equation}
Then one models the motion of the capillary droplet as a gradient flow of the functional $\mathcal{J}$ in an appropriate space of subsets of $\mathbb{R}^N$. 

\medskip
 A rigorous justification of this approach is carried out using a regularized discrete gradient flow scheme in the space of Cacciopoli sets (sets of finite perimeter) in \cite{GrunewaldKim11} (also see \cite{AD} for more general approach).  Since the analysis has been performed in the general setting which allows pinching and merging of droplet components, the resulting continuum solution is rather weak and is unstable under the variation of initial data. In this paper we show that much stronger results can be obtained when the initial data satisfies the $\rho$-reflection (see Theorem~\ref{main}). Our analysis will be based on the aforementioned energy structure of the problem with geometric arguments (reflection maximum principle)  as well as a modified viscosity solution theory where one views $\lambda(t)$ as a prescribed parameter.

 \medskip

 Let us explain the geometric argument in more detail. For {\it a priori} given  positive and continuous $\lambda(t)$, let us consider the problem
\begin{equation}\label{eqn: CLML}\tag{P-$\lambda$}
\left\{ \begin{array}{lll}
-\Delta u(x,t) = \lambda(t) &\text{ in }& \{u>0\}, \\ \\
 u_t = F(|Du|)|Du| & \text{ on } & \partial\{u>0\}. 
\end{array}\right.
\end{equation}
For a given function $\lambda(t)$ as above and under some assumption on the niceness of the initial data $\Omega_0$ and boundedness and ellipticity assumptions on $F$, a comparison principle holds for \eqref{eqn: CLML} and, by a standard application of Perron's method, there exist global-in-time viscosity solutions. 

\medskip

\medskip

We show, under an additional assumption on the star-shapedness of the initial data $\Omega_0$ ($\rho$-reflection), the existence of such a function $\lambda(t)$ defined for all times such that a viscosity solutions $u$ of \eqref{eqn: CLML} satisfies
$$ \int u(x,t) \ dx = \int u(x,0) \ dx =: V$$
for all $t>0$. The proof is  carried out by showing that, when $\Omega_0$ is strongly star-shaped,  then the ``energy solution", constructed with the discrete time scheme associated with gradient flow structure mentioned above, stays star-shaped and coincides with the viscosity solution of \eqref{eqn: CLML} with the corresponding $\lambda$. We point out that it is not a priori clear whether the discrete-time energy solutions preserve the star-shaped condition, so we incorporate the restriction directly into the approximate scheme (See section 4) to obtain the continuum limit. The introduction of geometric restriction to the gradient flow scheme seems to be new and of independent interest. 

\medskip

\subsection{Outline of the paper} 
In section 2 we recall known results about the solutions of $(EQ)$.  It turns out that viscosity solutions of $(EQ)$ with connected and Lipschitz positive phase are radial and minimize the associated energy.

\medskip

 In section 3 we investigate the geometric properties of the viscosity solutions of \eqref{eqn: CLML}. In particular we show that $\rho$-reflection property is preserved over the time (see Corollary ~\ref{reflectiontime}), based on the reflection maximum principle. We point out that the reflection maximum principle holds only when the solution of \eqref{eqn: CLML} is stable under perturbations, which is the case when the solutions are star-shaped (see section 3.5 as well as the Appendix). 

\medskip

In section 4 we discuss the energy structure of our original problem \eqref{eqn: CLMV}. Motivated by the formal gradient flow structure, we construct a solution of \eqref{eqn: CLMV} as a continuum limit of  a discrete-time ``minimizing movement" scheme following the approach of \cite{Chambolle04}, \cite{Mielke} and \cite{AD}. By putting the $\rho$-reflection property as a constraint for the minimizing movements, we obtain uniform convergence of solutions in the continuum limit. An important result proved here is Proposition 4.7, which states that the limiting ``energy" solution is indeed a viscosity solution of \eqref{eqn: CLML} with volume-preserving property.

\medskip

In section 5 we make use of the energy structure of the problem to show that any energy solution obtained in section 4 uniformly converges to the radial solution as $t\to\infty$, modulo translation (Theorem~\ref{thm: convergence thm}) We point out that our result does not imply that our solutions converge to a unique radial solution centered at a given point, since the drops may slowly move around a range of round profiles and may not converge to a single one.  The physical uniqueness of the limiting profile and the characterization of its center remain open at the moment (Though see Proposition~\ref{conditional} for the discussion on the uniqueness).

\section{The Equilibrium Problem}\label{sec: The Equilibrium Problem}
Now consider the volume constrained minimization problem 
\begin{equation}\label{equilibrium problem min form}
 \inf \{J(u): u \in H^1(\mathbb{R}^N) \text{ and } \int_{\mathbb{R}^N} u = V\}.
 \end{equation}
One can show immediately using symmetric decreasing rearrangements that any infimizer of 
$$ \inf \{J(u): u \in H^1(\mathbb{R}^N) \text{ and } \int_{\mathbb{R}^N} u = V \text{ and } u \text{ radial }\} $$
 is also an infimizer of \eqref{equilibrium problem min form}.  That the only minimizers are radial is more delicate.  One can show this using a theorem from \cite{BrothersZiemer88}.  We will not discuss the specifics here since this fact will not be needed for our arguments.
 
 \medskip
 
To find the unique (up to translation) minimizers of \eqref{equilibrium problem min form} among radial functions first fix a particular radius $r$ for the radially symmetric support set of the droplet 
 $$\{ u>0\}=B_r(0).$$
 Now the minimization becomes 
 $$ \inf \{ J(u): u \in H^1_0(B_r(0)) \text{ and } \int_{\mathbb{R}^N} u = V \text{ and } u \text{ radial }\} $$
 the unique infimizer is just the solution of the Dirichlet problem for the Laplace operator with lagrange multiplier $\lambda=\lambda(r) \in (0,+\infty)$ chosen so that the height profile
$$ u_r(x) = \max\{\frac{\lambda}{2N}(r^2 -|x|^2),0\}$$
has the correct volume. Then explicitly calculating $\lambda(r)$ and minimizing $J$ over $r>0$ one can easily check that for
$$r_* = (\frac{N}{N+2})^{1/(N+1)}V^{1/(N+1)} $$
$u_{r_*}$ is the strict minimizer of $J$ among radial $H^1$ functions and therefore -- by the rearrangement argument that was mentioned above -- is also a minimizer among all $H^1$ functions.

Alternatively, one can consider the Euler-Lagrange equation for \eqref{equilibrium problem min form}  which is given by,
\begin{equation}\label{eqn: equilibrium problem euler lagrange}\tag{$EQ$}
\left\{ \begin{array}{lll}
-\Delta u = \lambda & \text{ in } &  \{u>0\}, \\ \\
   \lambda>0 & \text{ s.t.} & \int u = V, \\ \\
|Du| = 1 & \text{ on } & \partial\{u>0\}. 
\end{array}\right.
\end{equation}
Then there is a classical theorem of Serrin \cite{Serrin71} regarding the uniqueness of solutions of $(EQ)$.
\begin{thm}\label{thm: serrin}\textup{(Serrin)}
Let $u: \mathbb{R}^N \to [0,+\infty)$ be compactly supported such that $\partial \{u >0\}$ a $C^2$ hypersurface in $\mathbb{R}^N$ and $u$ is a classical solution of (EQ).  Then $\{u>0\} = B_{r^*}(x_0)$ for some $x_0 \in \mathbb{R}^N$ and 
\begin{equation}\label{eqn: equilibrium radius}
r^* = (\frac{N}{N+2})^{1/(N+1)}V^{1/(N+1)}.
\end{equation}
\end{thm}
Serrin's proof of this theorem used the method of moving planes and a variant of the Hopf Lemma for domains with corners.  The best possible spatial regularity we are able to show for the evolving contact line in \eqref{eqn: CLML} is Lipschitz.  In order to apply Serrin's result to show the convergence to equilibrium, we need the same result to hold for viscosity solutions of $(EQ)$ with positivity sets which are Lipschitz domains.  It is not clear whether it is possible to use a variant of the moving planes method to show this.  Instead we use some regularity results for free boundary problems.  First a theorem of De Silva from \cite{DeSilva09} shows that viscosity solutions of $(EQ)$ with Lipschitz free boundaries are classical solutions.

\begin{thm}\label{thm: de silva}\textup{(De Silva) (Lipschitz implies $C^{1,\alpha}$)} Let $\Omega \subseteq \mathbb{R}^N$ a domain and $u:\mathbb{R}^N\to [0,+\infty)$ be a viscosity solution of the free boundary problem in $\Omega$,
\begin{equation}\label{problem}
\left\{\begin{array}{lll}
-\sum_{i,j=1}^N a_{ij} \partial_{ij}u = f & \text{ in } & \{x \in \Omega : u(x)>0\} \\ \\
|Du| = g & \text{ on } & \Gamma(u) := \partial\{u>0\}\cap\Omega
\end{array}\right.
\end{equation}
where $a_{ij}$, $g \in C^{0,\beta}(\Omega)$ for some $0 < \beta \leq 1$ and $f \in C(\Omega) \cap L^{\infty}(\Omega)$ and $g \geq 0$. Moreover, we have the ellipticity condition: there exists $0<\lambda<\Lambda$ such that for all $\xi \in \partial B_1(0)\subset\mathbb{R}^N$ and all $x \in \Omega$,
$$ \lambda \leq \sum a_{ij}(x)\xi_i\xi_j \leq \Lambda. $$
Then if $x_0 \in \Gamma(u)$, $g(x_0)>0$ and $\Gamma(u)$ is a Lipschitz graph in a neighborhood of $x_0$ then $\Gamma(u)$ is $C^{1,\alpha}$ in a smaller neighborhood of $x_0$ for $\alpha>0$ depending only on $N$, $\beta$, and the ellipticity constants $\lambda$ and $\Lambda$.
\end{thm}
Using above result, higher regularity of $u$ can be derived from the Hodograph method \cite{KN}  when the coefficients are smooth and $\Gamma(u)$ is locally Lipschitz. See Appendix C for more details .

\begin{cor}\label{smoothness}
Let $u$ solve \eqref{problem} with $(a_{ij})$ as identity matrix and $f=g=1$. In addition suppose that $\Gamma(u)$ is locally Lipschitz and is bounded. Then $u$ and $\Gamma(u)$ is $C^{\infty}$. 
\end{cor}

\section{Viscosity Solutions}\label{sec: Viscosity Solutions}
\subsection{Basic Definitions and Assumptions}\label{sec: Basic Definitions and Assumptions}
We will recall the basic viscosity solution theory for \eqref{eqn: CLML}.  A more detailed exposition for free boundary problems of a similar form can be found in \cite{Kim03}.  First we restate the problem under consideration,
\begin{equation}\tag{\ref{eqn: CLML}}
\left\{
\begin{array}{lll}
-\Delta u(x,t) = \lambda(t) & \text{ in } & \Omega_t, \\ \\
 u_t = F(|Du|)|Du| & \text{ on } & \Gamma_t. 
\end{array}\right.
\end{equation}
$\lambda: [0,+\infty) \to [0,+\infty)$ will be bounded and continuous.  We will work in a space-time parabolic domain $Q = U\times(a,b]$  where $0\leq a <b\leq +\infty$ and $U$ is a domain (possibly unbounded) in $\mathbb{R}^N$, with parabolic boundary,
$$\partial_p Q := \overline{U}\times \{t=0\} \cup \partial U \times [a,b].$$  
The positive phase of a height profile $u: Q \to [0,+\infty)$ is defined as
\begin{equation}
 \Omega(u) := \{u>0\} \quad \text{ and } \quad \Gamma(u) := \partial \Omega(u)
 \end{equation}
 and at particular times,
 $$ \Omega_t(u) = \Omega(u) \cap \{ (y,s) \in Q: s=t\} \quad \text{ and } \quad \Gamma_t(u) = \partial\Omega_t(u). $$
The dependence on $u$ will be omitted when it is unambiguous which droplet profile we are referring to. The free boundary velocity $F$ will satisfy the boundedness and ellipticity assumption,
\begin{hyp}\label{hyp: F incr cond}
$F:[0,+\infty)\to\real$ is strictly monotone increasing and continuous, and $F(1)=0$.
\end{hyp}
\begin{hyp}\label{hyp: F cond}
There exists $c>0$ such that for $\epsilon$ sufficiently small,
\begin{equation}\label{eqn: F condition}
(1+\epsilon)F((1+\epsilon)^{-1}s)+c\epsilon \geq F(s).
\end{equation}
\end{hyp}
Notice if Assumption 2 holds then necessarily $F$ has sublinear growth at $\infty$.  For any $s>0$ using \eqref{eqn: F condition} $\floor{s/\epsilon}$ times with $\epsilon\to 0$ yields:
$$F(s) \leq cs+\max\{(1+s)F(0),0\}.$$

The monotonicity assumption on $F$ implies, at least formally, that the problem \eqref{eqn: CLML} has a comparison principle.  This is the underlying reason why viscosity solutions are the natural definition of weak solution for this PDE.  The second assumption is important for proving the strong comparison results and thereby the regularity of the viscosity solutions of \eqref{eqn: CLML}.  It is not clear to the authors whether this assumption is only technical.  Some examples which satisfy the Assumptions \ref{hyp: F incr cond} and \ref{hyp: F cond} are
\begin{ex}
If $F(s) := s^p-1$ for $p\leq 1$ then by a simple calculation \eqref{eqn: F condition} is satisfied for all $\epsilon>0$ with $c=1$.  
\end{ex}
The following example in the case $p=2$ will be important to us later,
\begin{ex}
Let $M>0$ and $p>1$, suppose $F(s) := \max\{s^p-1,M\}$ then \eqref{eqn: F condition} is satisfied for $\epsilon \leq 1+\tfrac{1}{2}(p-1)^{-1}$ and $c \geq 2(M+2)$.  The calculation is given below,
\begin{align*}
\max\{(1+\epsilon)[(1+\epsilon)^{-p}s^p-1)],M\}+c\epsilon  &\geq   \max\{s^p-1-(p-1)\epsilon s^p-\epsilon,M\} +c\epsilon \\
& \geq  \max\left\{s^p-1-2(M+2)\epsilon,M\right\} +c\epsilon \\
& \geq \max\left\{s^p-1+\left(c-2(M+2)\right)\epsilon,M\right\} \\
&\geq \max\{s^p-1,M\}.
\end{align*}
Above in the second line we have used that if 
$$ s^p-1-(p-1)\epsilon s^p-\epsilon \leq M$$
then
$$ s^p \leq 2(M+2). $$
\end{ex}

 Now we turn to defining a notion of solution for \eqref{eqn: CLML}.  For $E \subset \real^N\times[0,+\infty)$ we use the notation $C^{2,1}(E)$ for functions $f$ with two continuous derivatives in the spatial variables and one continuous derivative in time. First we define a classical solution of the free boundary problem.
 \begin{DEF} (Classical Solutions)
 A profile $u: Q \to [0,+\infty)$ is a classical subsolution (supersolution) for the free boundary problem \eqref{eqn: CLML} on $Q$ if $ u \in C^{2,1}(\overline{\Omega(u)}\cap Q)$ and \eqref{eqn: CLML} is satisfied in the pointwise classical sense with equality signs replaced by $\leq$ ($\geq$ in the supersolution case).  Classical solutions are both super and subsolutions.  
 \end{DEF}
  For general initial data the contact line motion problem will not have a classical solution.  To get existence of a weak solution we define the notion of viscosity solution of \eqref{eqn: CLML}.   First though we define the following standard notion.
  \begin{DEF}\label{strictlyseparated} (Strictly separated)
Let $v$, $w$ be defined on a set $D\subseteq\mathbb{R}^N$ then we say $v$ and $w$ are strictly separated on $D$ and write $ v \prec w$ if $\overline{\Omega(v)}\cap\overline{D}$ is compact and $v<w$ on $\overline{\Omega(v)}\cap D$.
\end{DEF}

Next we define subsolutions and supersolutions, then a viscosity solution is defined to satisfy both the sub and supersolution properties.  Informally, $u$ is a subsolution of \eqref{eqn: CLML} if $u$ cannot be crossed from above by any strict classical supersolution.  More precisely:

\begin{DEF}\label{def: subsolution} (Subsolution) A non-negative upper semi-continuous function $u:Q \to \mathbb{R}_+$ is a subsolution of \eqref{eqn: CLML} if, for any parabolic neighborhood $Q' \subseteq Q$, and any strict classical supersolution $\phi$ with $u \prec \phi$ on $\partial _p Q'$, then $u \leq (\phi)_+$ in $Q'$.
\end{DEF}
\begin{DEF}\label{def:supersolution} (Supersolution) A non-negative lower semi-continuous function $u:Q \to \mathbb{R}_+$ is a supersolution of \eqref{eqn: CLML} if, for any parabolic neighborhood $Q' \subseteq Q$, and any strict classical supersolution $\phi$ with $u \prec \phi$ on $\partial _p Q'$, then $ \phi \leq u$ in $Q'$.
\end{DEF}
\begin{DEF}\label{viscositysolution}
(Solution) A viscosity solution of \eqref{eqn: CLML} is a non-negative continuous function $u$ on $Q$ which is both a supersolution and a subsolution.
\end{DEF}

Naturally, one can assign boundary data on the parabolic boundary.  We will usually have $Q = \mathbb{R}^N \times (0,T]$ and in that case this will reduce to assigning initial data.
\begin{DEF} \label{barrier}(Subsolution with boundary data $g$)
Let $g:\partial_p Q \to [0,+\infty)$ bounded then $u$ is a subsolution of \eqref{eqn: CLML} on $Q$ with boundary data $g$ if $u$ is a subsolution of \eqref{eqn: CLML} on $Q$ as defined above and
$$ \limsup_{(y,s)\to (x,t)\in\partial_p Q} u(y,s) \leq g(x,t)$$
\end{DEF}
Supersolutions and solutions are then defined analogously.  

\medskip

  For the rest of the paper, in the case $Q=\real^N\times (0,T]$, when we assign initial data we will specify the positivity set as some $\Omega_0 \subset \real^N$ an open domain and then $u_0$ the initial data will be taken to solve
$$ -\Delta u_0 (x) = \lambda(0) \quad \text{ and} \quad u_0(x)=0 \text{ for } x\in\Gamma_0 = \partial\Omega_0. $$
 As mentioned in the introduction, it is often more natural to think of the problem as an evolution on the positivity sets.  
 \begin{rem}\label{rem: continuity of solutions}
 In general, the solution of \eqref{eqn: CLML} is not expected to be continuous.  First of all, the solutions can vanish or blow up in finite time.  The following example in the case $N=1$ demonstrates these behaviors.
 \begin{ex}
 Consider in $N=1$ the problem \eqref{eqn: CLML} with $ \lambda(t) = \lambda$ for some $\lambda>0$
 and initial data
 \begin{equation*}
 \begin{array}{ll}
 u_0(x) = \tfrac{\lambda}{2N}(1-x^2)_+, & \Omega_0 = [-1,1].
 \end{array}
 \end{equation*}
  The solution of \eqref{eqn: CLML} with the above initial data and inhomogeneity will take the form
 $$ u(x,t) = \tfrac{\lambda}{2N}(r(t)-x^2)_+ $$
 where $r(t)$ solves the ordinary differential equation
 $$ \dot{r}(t) = F\left(\tfrac{\lambda}{N}r(t)\right) $$
 with initial data $r(0) = 1$.  From basic ODE theory there are choices of $F$ for which some solutions of the above equation blow up in finite time.  For example if 
 \begin{equation*}
 \begin{array}{lll}
 F(s) = s^2-1 & \text{ and } & \lambda \geq 2N,
 \end{array}
 \end{equation*}
  then first of all $\dot{r}(t)>0$ as long as the solution exists and moreover
 $$\dot{r}(t) \geq 4r(t)^2-1 \geq 3r(t)^2 $$
 an equation for which finite time blow up is well known.  (This behavior is ruled out by Assumption 2 on $F$ which implies sublinear growth at $\infty$.)  Alternatively, even when $F$ does not have superlinear growth solutions can disappear in finite time. Consider now any $F$ such that 
 $$F(0)<0$$
 then let $-F(0)>\delta>0$ and choose $\lambda$ sufficiently small that 
 $$F\left(\tfrac{\lambda}{N}\right)<-\delta<0.$$
 Then we have the following differential inequality for $r(t)$,
 $$\dot{r}(t) \leq -\delta$$
 so the solution $u(x,t)$ must become extinct before time $1/\delta$.
 \end{ex}
 Even without blow up the solution may be discontinuous for certain initial data:  one can construct an example in one dimension where two adjacent drops merge after a short time, causing a discontinuity. To address the long time behavior of our solutions without the complexity of multiple components, we will make restrictions on the initial data such that our solutions actually are continuous.
 \end{rem}
 
  The following lemma clarifying the connection between classical and viscosity solutions is standard in the viscosity solutions theory. For example see \cite{BrandleVazquez05, Kim03}.
 
\begin{lem}
Suppose that $u:Q \to [0,+\infty)$ is a classical subsolution (supersolution) of \eqref{eqn: CLML} in $Q$ then it is also a viscosity subsolution (supersolution) in $Q$.  Conversely suppose that $u$ is a viscosity subsolution (supersolution) of \eqref{eqn: CLML} in $Q$ and moreover $u$ is sufficiently regular, in particular
$$u \in C^{2,1}(\overline{\Omega(u)}\cap Q),$$
then $u$ is a classical subsolution (supersolution) in $Q$.
\end{lem}

An important basic property of viscosity solutions theory is the stability of viscosity solutions under uniform convergence.  We state this in the following lemma.  For example see \cite{CrandallIshiiLions92}.
\begin{lem}\label{lem: uniform stability}
Let $u_n:\real^N\times[0,T]$ be a sequence of viscosity solutions of the problems
\begin{equation*}
\left\{
\begin{array}{lll}
-\Delta u_n(x,t) = \lambda_n(t) & \text{ in } & \Omega_t(u_n), \\ \\
 \partial_tu_n = F_n(|Du_n|)|Du_n| & \text{ on } & \Gamma_t(u_n). 
\end{array}\right.
\end{equation*}
where $F_n$ are all monotone increasing and continuous.  Suppose that $u_n \to u$, $F_n \to F$ and $\lambda_n \to \lambda$ uniformly on compact sets, then $u$ is a viscosity solution of 
\begin{equation*}
\left\{
\begin{array}{lll}
-\Delta u(x,t) = \lambda(t) & \text{ in } & \Omega_t(u), \\ \\
 \partial_tu = F(|Du|)|Du| & \text{ on } & \Gamma_t(u). 
\end{array}\right.
\end{equation*}
\end{lem}

As mentioned in the introduction, it is useful sometimes to formulate the problem \eqref{eqn: CLML} as an evolution on domains in $\real^N$.  We quickly define what it means for an evolution
$$\Omega_t : [0,+\infty) \to \{ \text{ bounded subsets of } \real^N\} $$
to solve \eqref{eqn: CLML} in the viscosity sense.
\begin{DEF}
The evolution $\Omega_t$ is a supersolution (subsolution) of \eqref{eqn: CLML} if and only if, 
\begin{equation}\label{eqn: drp above}
u[\Omega_t] := 
\left \{ \begin{array}{lll}
         -\Delta u(x,t)  = \lambda(t) & \text{ in }&  \Omega_t\\ \\
         u=0 & \text{ on } & \Gamma_t.
         \end{array}\right.
\end{equation}
is a supersolution (subsolution) in the sense of Definition \ref{def:supersolution} (Definition \ref{def: subsolution}).
\end{DEF}

\subsection{Sup and Inf Convolutions}
An important property of subsolutions (supersolutions) is the closure under sup (inf) convolutions. These convolutions will be used to perturb the originial solutions and to overcome the lack of scaling invariance in the problem \eqref{eqn: CLML}.

\begin{lem}\label{lem: spaceconvs}
\textup{(Convolutions in space) }
\begin{enumerate}[(a)]
\item Let $u$ be a subsolution of \eqref{eqn: CLML} on $\mathbb{R}^N\times[0,+\infty)$, and $r>0$, $c \geq 0$ and define the sup-convolution of $u$ 
\begin{equation}\label{eqn: spacesupconv}
 \widetilde{u}(x,t) := \sup_{y \in B_{r-ct}(x)}u(y,t). 
 \end{equation}
 Then $\widetilde{u}$ is a subsolution with free boundary speed $F(|D\widetilde{u}|)-c$ as long as $r-ct>0$.
 \item Let $u$ be a supersolution of \eqref{eqn: CLML} on $\mathbb{R}^N\times[0,+\infty)$, and $r,c>0$ and define the inf-convolution of $u$ 
\begin{equation}\label{eqn: spaceinfconv}
 \widetilde{u}(x,t) := \inf_{y \in B_{r-ct}(x)}u(y,t). 
 \end{equation}
 Then $\widetilde{u}$ is a supersolution with free boundary speed $F(|D\widetilde{u}|)+c$ as long as $r-ct>0$.
\end{enumerate}
\end{lem}
Before we prove this let us note the continuity properties of the sup and inf convolutions.
\begin{lem}\label{lem: sup inf continuity} \textup{(Continuity of sup and inf convolutions)}
Let $f : \mathbb{R}^N\to \mathbb{R}$ and let $S : \mathbb{R}^N \to \mathcal{B}(\mathbb{R}^N)$ such that $S(x)$ is compact for all $x \in \mathbb{R}^N$ and $S$ is continuous under the Hausdorff distance $d_H$. Define the sup and inf convolutions 
$$ f^{S}(x) = \sup_{h\in S(x)}f(y+h) \quad \text{ and } \quad f_{S}(x)=\inf_{h \in S(x)}f(y+h). $$
Then we have the following
\begin{enumerate}[(i)]
\item If $f$ is upper semi-continuous then $f^S$ is as well.
\item If $f$ is lower semi-continuous the $f_S$ is as well.
\item If $f$ is continuous then both $f^S$ and $f_S$ are as well. 
\end{enumerate}
\end{lem}
\begin{proof}
We will prove (i), (ii) is similar and (iii) follows from (i) and (ii) by noting that suprema of lower semi-continuous functions are lower semicontinuous and infima of upper semicontinuous functions are upper semi-continuous. Let $\alpha \in \mathbb{R}$ we will show that the sub level sets $\{f^S < \alpha\}$ are open.  Let $x_0 \in \{f^S < \alpha\}$ then $x_0+S(x_0) \subset \{f <\alpha\}$.  Since $x_0+S(x_0)$ is compact, $\{f\geq\alpha\}$ is closed, and they are disjoint, they must be a positive distance apart. That is, there exists $\epsilon>0$ such that 
$$ \bigcup_{h \in S(x_0)} B_\epsilon(x_0+h) \subseteq \{f<\alpha\}. $$
By continuity of $S$ there exists $\delta<\epsilon/2$ such that $|x-x_0|<\delta$ means that $d_H(S(x),S(x_0))<\epsilon/2$. Let $x \in B_{\epsilon/2}(x_0)$ and since upper semi-continuous functions achieve their maximum on compact sets there exists $h \in S(x)$ such that $f^S(x) = f(x+h)$.  Then there exists $h_0$ in $S(x_0)$ such that $|h-h_0| <\epsilon/2$ so that
$$ |(x+h)-(x_0+h_0)| <\epsilon\quad  \text{ and hence }\quad x+h \in B_{\epsilon}(x_0+h_0) \subseteq \{f<\alpha\}. $$ 
So $f^S(x) <\alpha$ and therefore $ B_{\epsilon/2}(x_0) \subseteq \{f^S<\alpha\}$.

\end{proof}
{\it The proof of Lemma \ref{lem: spaceconvs}:} We will only prove the inf convolution case (b), since the proof of (a) is essentially the same.  As in \eqref{eqn: spaceinfconv} for $r,c>0$ define the inf convolution of a supersolution $u$,
\begin{equation*}
 \widetilde{u}(x,t) := \inf_{y \in B_{r-ct}(x)}u(y,t). 
 \end{equation*}
 Suppose that for some parabolic domain $Q \subseteq \real^N \times [0,r/c)$ there exists a strict classical subsolution $\phi\in C^{2,1}(Q \cap \overline{\Omega(\phi)})$ with the free boundary speed $F(|D\phi|) + c$, which satisfies
 $$ \phi \prec \widetilde{u} \ \hbox{ on } \ \partial_pQ \ \hbox{ and } \  \phi(x_0,t_0)>\widetilde{u}(x_0,t_0) \ \hbox{ for some } (x_0,t_0) \in Q.$$
Let $x_1$ be the point where $\widetilde{u}(x_1,t_0)= u(x_0,t_0)$. Now we define the translated (and convoluted) test function
$$ \psi (x,t) := \sup_{|x-y| \leq c(t_0-t), (y,t)\in Q} \phi (y+(x_0-x_1),t) \hbox{ for } t\leq t_0,$$
and  define $\psi_2$ by  the solution of the Dirichlet problem
$$
-\Delta \psi_2(\cdot,t) = \lambda(t) \hbox{ in } \{\psi>0\}\cap Q, \quad \psi_2=\psi \hbox{ on the lateral boundary of }\{\psi>0\}\cap Q.
$$
Note that the free boundary speed of $\psi_2$ has been decreased by $c$ over the free boundary speed of $\phi$ in $Q$.  In other words, $\psi_2$ is a classical subsolution with free boundary speed $F(|D\psi|)$. 

Due to the definition of $\widetilde{u}$, we have $\psi_2 \leq u$ on the parabolic neighborhood of $Q +(x_1-x_0)$ and $\psi_2>u$ at $(x_1,t_0)$, which yields a contradiction. 


\hfill$\Box$

\subsection{Comparison}\label{sec: Comparison}
We state the general strictly separated comparison principle for viscosity solutions of \eqref{eqn:CLML},
\begin{thm}\label{thm: strictly separated comparison} \textup{(Comparison for strictly separated data)}
Suppose $u$ is a supersolution and $v$ a subsolution of \eqref{eqn: CLML} on $Q$. Suppose $u$ and $v$ are strictly separated (defined below) on the parabolic boundary of $Q$,
$$v \prec u \quad \text{ on } \quad \partial_p Q $$
then $v < u$ in $\overline{\Omega(v)} \cap Q$, in particular $\Gamma(v)$ cannot touch $\Gamma(u)$ from the interior on any compact subset of $Q$.
\end{thm}
\begin{proof}
The proof can be found in various places in the literature \cite{Kim03,BrandleVazquez05,CaffarelliVazquez99, KimPozar12}. We will go into more detail later in Section \ref{sec: Comparison}.
\end{proof}

 As mentioned in Remark \ref{rem: continuity of solutions}, we will need to make a geometric restriction on our initial data to expect the existence of a viscosity solution which is stable under a family of perturbations.  First, we recall the definition of a set star-shaped with respect to a point:
\begin{DEF}
A domain $\Omega \subseteq \real^N$ is called star-shaped with respect to a point $x$ if for every $y \in \partial \Omega$ the line segment between $x$ and $y$ is contained in $\overline{\Omega}$.
\end{DEF}
Then the assumption on our initial data $\Omega_0$ will be called \textit{strong star-shapedness}: 
\begin{DEF}\label{hyp: starshapedness}
We call a domain $\Omega$ \textit{strongly star-shaped} if there exists $r>0$ such that $\Omega$ is star-shaped with respect to every point of $B_r(0)$.  For each $R>0$ we define the class of uniformly bounded strongly star-shaped sets:
\begin{equation}
\begin{array}{ll}
\mathcal{S}_{r,R}:= \{\Omega \subset B_{R}(0): \Omega \text{ star-shaped with respect to } B_r(0)\}, &
\mathcal{S}_{r} := \bigcup_{R>0} \mathcal{S}_{r,R},
\end{array}
\end{equation}
and the class of strongly star-shaped sets
 \begin{equation}
\mathcal{S}_0 := \bigcup_{r>0} \mathcal{S}_{r}.
\end{equation}
\end{DEF}
That the ball $B_r(0)$ is centered at the origin is only for convenience, the problem is translation invariant. We will often refer to the strongly star-shaped property by saying that a set $\Omega_0 \in \mathcal{S}_{r,R}$ in order to clarify the role of $r$.  We note a basic property of sets in $\mathcal{S}_{r,R}$ -- they are Lipschitz domains with the Lipschitz constant depending only on $r,R$.  First, we define a notation for cones with apex at the origin. For $x \in \real^N$ and $\theta \in (0,\pi)$ the cone in the direction $x$ with opening angle $\theta$ is called,
\begin{equation}\label{eqn: cone def}
C(x,\theta):= \left\{y: \langle x , y \rangle \geq (\cos\theta)|x||y|\right\}.
\end{equation}
We show that $\Omega \in \mathcal{S}_{r}$ have interior and exterior cones at every boundary point:
\begin{lem}\label{lem: sstolip}
The following are equivalent for a bounded domain $\Omega \supset B_r(0)$:
\begin{enumerate}[(i)]
\item The domain $\Omega \in \mathcal{S}_{r}$. 
\item There is an $\epsilon_0 \in (0,\infty]$ such that for all $x \in  \Omega^C$ there is an exterior cone to $\Omega$,
\begin{equation}
\begin{array}{lll}
 x+C\left(x,\theta_x\right)\cap B_{\epsilon_0}(x) \subset \Omega^C & \text{ where } & \sin\theta_x = \frac{r}{|x|}.
 \end{array}
 \end{equation}
\item For all $x \in \partial \Omega$ there is an interior 'cone' to $\Omega$,
\begin{equation}
\begin{array}{lll}
 \left(x+C\left(-x,\theta_x\right)\right) \cap C\left(x,\tfrac{\pi}{2}-\theta_x\right) \cup B_r(0) \subset \Omega & \text{ where } & \sin\theta_x = \frac{r}{|x|}.
 \end{array}
 \end{equation}
 \item There exists $\epsilon_0>0$ so that
  \begin{equation}\label{eqn: inf conv containment}
  \Omega \subset\subset \bigcap_{|z| \leq a\epsilon} \left[(1+\epsilon)\Omega+z\right] \  \hbox{ for all } \epsilon_0>\epsilon>0,
 \end{equation}
 for every $0<a<r$.
\end{enumerate}
\end{lem}
\begin{proof}
The proof for parts (i)-(iii) is essentially in the picture, see Figure~\ref{fig: strong star shaped}.  For part (iv) let us first note that (see figure \ref{fig: strong star shaped}),
 $$ x+C(x,\theta_x)\cap B_{\epsilon_0r}(x)  \subset \bigcup_{\epsilon_0>\epsilon>0}\bigcup_{0<a<r} B_{a\epsilon}((1+\epsilon)x).$$
 Therefore by (ii) $\Omega \in \mathcal{S}_r$ if and only if for all $x \in  \Omega^C$ all $\epsilon_0>\epsilon>0$ and all $0<a<r$,
 $$ B_{a\epsilon}((1+\epsilon)x) \subset \subset \Omega^C,$$
 or equivalently, for every $\epsilon_0>\epsilon>0$ and all $0<a<r$,
 $$ \{ x: d(x,(1+\epsilon)\Omega^C) \leq a\epsilon \} \subset\subset \Omega^C. $$
 Again, equivalently, for every $\epsilon_0>\epsilon>0$ and all $0<a<r$,
 \begin{align*}
 \Omega &\subset \subset \{ x: d(x,(1+\epsilon)\Omega^C) \leq a\epsilon \}^C  = \real^N\setminus\bigcup_{|z| \leq a \epsilon} [(1+\epsilon)\Omega^C+z]  = \bigcap_{|z| \leq a\epsilon} \left[(1+\epsilon)\Omega+z\right]. 
 \end{align*}
 
 \begin{figure}[t]
\begin{center}
\includegraphics[scale=0.5]{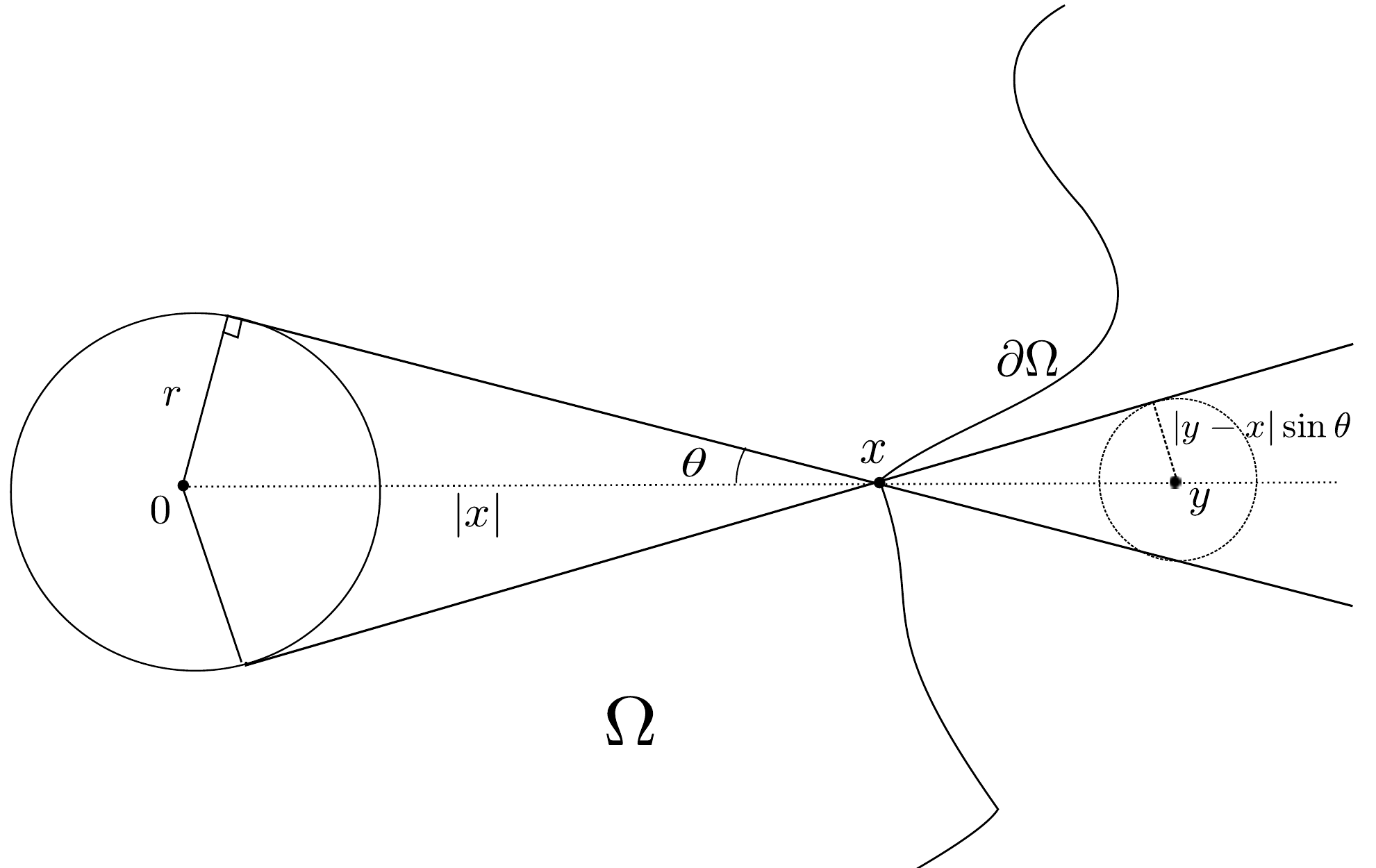}
\end{center}
\caption{Interior and exterior cones for a strongly star-shaped domain.}
\label{fig: strong star shaped}
 \end{figure}
\end{proof}

We prove a comparison principle for reflections through hyperplanes.  This will be very useful to us later because the reflection ordering is a property which does not depend on the radius from the strong star-shapedness property.  Let $H$ be a hyperplane in $\mathbb{R}^N$ with $y \in H$ define the reflection through $H$ by
$$ \phi_H(x) = x-2\langle x-y, \nu_H\rangle\nu_H.$$
The symmetry of the problem \eqref{eqn: CLML} with respect to reflections at least formally implies that if a solution $u(\cdot,t)$ and its reflection $u(\phi_H(\cdot),t)$ through $H$ are initially ordered in the half spaces of $H$ then they will remain so. 
\begin{prop}\label{prop: reflection comparison}
\textup{(Reflection Comparison)} Suppose $u:\mathbb{R}^N\times[0,T]\to[0,+\infty)$ is a solution of \eqref{eqn: CLML} such that $\Omega_t(u)\in\mathcal{S}_{r}$ for some $r>0$ and for all $t \in [0,T]$.  Let $H$ be a hyperplane in $\mathbb{R}^N$ with $H\cap B_r(0) = \emptyset$ and define $H_+$ and $H_-$ the half spaces of $H$ such that $B_r(0) \subset H_+$. Suppose that $\phi_H(\Omega_0)\cap H_+ \subseteq\Omega_0 \cap H_+ $ i.e.
$$ u(\phi_H(x),0) \leq u(x,0) \text{ for } x \in H_+.$$
Then we have the ordering:
$$ u(\phi_H(x),t) \leq u(x,t) \text{ for } (x,t) \in H_+\times[0,T].$$
\end{prop}
A similar argument will give the more standard strong comparison result,
\begin{lem}\label{lem: strong comparison} \textup{(Strong comparison)}
Suppose $u$ and $v$ are respectively a sub- and supersolution of \eqref{eqn: CLML} on $\mathbb{R}^N\times [0,+\infty)$ which are initially ordered,
$$\Omega_0(u)\subseteq\Omega_0(v).$$
 Suppose that that there exists $T\geq0$ such that for each $t \in [0,T]$ either $\Omega_t(u)$ or $\Omega_t(v)$ is in $\mathcal{S}_{r,R}$ for some $r,R>0$ independent of $t$.  Then let $c$ from Assumption~\ref{hyp: F cond} and define $t_0 := r^2/(Rc)$ then
$$
\Omega_t(u) \subseteq \Omega_t(v) \hbox{ for } t\leq T+t_0.
$$
\end{lem}
\begin{proof}
The proof is a slightly easier version of the proof of Proposition \ref{prop: reflection comparison} so we omit it.
\end{proof}

 In order to get a strong comparison type result we want to slightly perturb the supersolution (or the subsolution) so that strict separation holds initially and we can use Theorem \ref{thm: strictly separated comparison}.  To achieve this we use the inf and sup convolutions introduced above.  This will be where the 'technical' assumption on $F$ \ref{hyp: F cond} comes in.

\textit{Proof of Proposition \ref{prop: reflection comparison}:} Let $\epsilon_0$ and $c$ from Assumption \ref{hyp: F cond}, $0<a<r$, and define for $0<\epsilon \leq \epsilon_0$,
\begin{align*}
v_\epsilon(x,t) &= (1+\epsilon)^2u((1+\epsilon)^{-1}x,t);\hbox { and } \\
 u_\epsilon(x,t) = &\inf_{y \in B_{a\epsilon-c\epsilon t}(x)} v_\epsilon(y,t) \ \  \text{ for } \ \ 0\leq t < \frac{a}{c}. 
\end{align*}
Then, from Lemma \ref{lem: sstolip} part (iv) and the strong maximum principle for harmonic functions, we have
$$
u(x,t) < u_\epsilon(x,t)\ \text{ in }\ \mathbb{R}^N \times [0,\frac{a}{c})\hbox{ for } \epsilon>0. 
$$ 
We check the super-solution property of $u_\epsilon$.  Note that $v_\epsilon(x,t)$ is a supersolution of,
\begin{equation*}
\left \{ \begin{array}{lll}
         -\Delta v_\epsilon(x,t) \geq \lambda(t) & \text{ in } & \Omega_t(v_\epsilon)\\ \\
          \partial_tv_\epsilon \geq (1+\epsilon)F((1+\epsilon)^{-1}|Dv_\epsilon|)|Dv_\epsilon| & \text{ on } & \Gamma_t(v_\epsilon).
         \end{array}\right.
\end{equation*}
Then, from Lemma \ref{lem: spaceconvs}, $u_\epsilon(\cdot,t)$ is a supersolution of 
\begin{equation*}
\left \{ \begin{array}{lll}
         -\Delta u_\epsilon(x,t) \geq \lambda(t) & \text{ in } & \Omega_t(u_\epsilon)\\ \\
          \partial_tu_\epsilon \geq [(1+\epsilon)F((1+\epsilon)^{-1}|Du_\epsilon|)+c\epsilon]|Du_\epsilon| & \text{ on } & \Gamma_t(v_\epsilon).
         \end{array}\right.
\end{equation*}
Using Assumption \ref{hyp: F cond} we conclude that $u_\epsilon$ is a super-solution of \eqref{eqn: CLML}.

Since $u_\epsilon$ is a  super-solution of \eqref{eqn: CLML}, by the strictly separated comparison principle Theorem~\ref{thm: strictly separated comparison},
$$ u(\phi_H(x),t) \leq u_\epsilon(x,t) \text{ in }\ H_+ \times [0,\frac{a}{c}). $$
Since $u_\epsilon \searrow u$ as $\epsilon\to 0$,  it follows that
$$ u(\phi_H(x),t) \leq u(x,t) \text{ in }\ H_+ \times [0,\frac{a}{c}).$$
 We now iterate $\floor{T_0c/a}+1$ times to conclude.

\hfill$\Box$

\subsection{Short Time Existence}
As usual in the viscosity solutions theory, the comparison theorem is the key ingredient needed to use Perron's method to show existence.  In the existence proof we will use the following elementary facts about solutions of Poisson's equation at the boundary.

\begin{lem}\label{lem: holder in lip domain} \textup{(Boundary H\"{o}lder Estimates in Lipschitz Domains)} Let $\Omega \subset \real^N$ be a domain and $x_0 \in \partial\Omega$ and suppose that there exists $\theta \in (0,\pi/2)$ and $\nu \in S^{N-1}$ such that,
$$ x_0+\{x: \langle x,\nu \rangle>\cos\theta|x| \}\subset\real^N\setminus\Omega.  $$
In other words, $\Omega$ has an exterior cone of opening angle $\theta$ at $x_0$.  Suppose $\Lambda>\lambda>0$ and $u: \Omega \to [0,+\infty)$ satisfies 
  \begin{equation*}
  \begin{array}{lll}
        -\Delta u (x)\leq \lambda & \text{ for} & x \in \Omega \\ \\
        u(x) = 0 & \text{ on } & \partial\Omega
        \end{array}
     \end{equation*}
        Then there exists $\alpha = \alpha(\theta,N) \in (0,1)$ and $C=C(\Lambda,N)>0$ such that for all $h>0$,
        $$\sup_{|z|\leq h} u(x_0+z) \leq Ch^\alpha .$$
\end{lem} 

 \begin{lem}\label{lem: bdry gradient est} \textup{(Boundary Gradient Estimate in $C^{1,1}$ Domains)}
 Let $\Omega \subset \real^N$ be a domain $x_0 \in \partial\Omega$ and suppose that there exists $r>0$, $y \in \real\setminus\overline{\Omega}$ such that
$$ B_r(y) \cap \overline{\Omega } = \{x\}.$$
In other words, $\Omega$ has an exterior ball at $x_0$ of radius $r$.  Suppose $\Lambda>\lambda>0$ and $u: \Omega \to [0,+\infty)$ satisfies $u(x_0) = 0$ and
  \begin{equation*}
        -\Delta u \leq \lambda \quad \text{ for} \quad x \in \Omega .
\end{equation*}
Then 
$$ \lim_{s \to 0}\sup_{z \in \Omega \cap B_s(x)} \frac{u(z)}{s} \leq C(\Lambda,N)\frac{\sup_{B_{2r}(x_0)} u}{r}. $$
 \end{lem}
 
 We move on to the 'short time' existence theorem.  Actually this shows the existence of a global in time discontinuous viscosity solution without any need for Assumption \ref{hyp: F cond}.  The uniqueness and continuity however rely on the strong comparison result which requires Assumption \ref{hyp: F cond} and only holds for a short time.  This is to be expected as noted in Remark \ref{rem: continuity of solutions}.  In particular the limiting factor will be the strong comparison principle Lemma \ref{lem: strong comparison} which only holds for a short time barring any a priori knowledge about strong star-shapedness. 
\begin{thm}\label{thm: short time existence} \textup{(Short time existence and uniqueness)}
Let $r,R>0$ and $\Omega_0 \in \mathcal{S}_{r,R}$ and let $u_0$ solve
$$-\Delta u_0 (x) = \lambda(0) \quad \text{ and} \quad u_0(x)=0 \text{ for } x\in\Gamma_0.$$
Then there is a $t_0>0$ depending only on $r$ such that there exists a (unique) continuous viscosity solution $u$ of \eqref{eqn: CLML} on $Q = \mathbb{R}^N\times(0,t_0]$ with initial data $u_0$.
\end{thm}

\begin{proof}
This proof can be found in \cite{GlasnerKim09}.  We construct a sub and supersolution which take the initial data continuously then apply Perron's method. 

1. First we construct a subsolution. Let $\gamma>0$ to be chosen later and define,
\begin{equation}\label{eqn: perron subsoln}
U(x,t) := \left \{ \begin{array}{lll}
        (1-\gamma t)^2u_0(\frac{x}{1-\gamma t}) & \text{ for} & t<\gamma^{-1} \\ \\
        0 & \text{ for} & t>\gamma^{-1}
         \end{array}\right.
\end{equation}
Then, calculating formally, if $x_0 \in \Gamma_0$ then $(1-\gamma t)x_0$ is in $\Gamma_t$ and
$$ \frac{U_t}{|DU|}( (1-\gamma t)x_0) = \gamma x_0 \cdot \frac{Du_0}{|Du_0|}\leq -\gamma |x_0| \sqrt{1-\left(\frac{r}{|x_0|}\right)^2} \leq \min F $$
if $\gamma = \gamma(r,R)$ is chosen small.  In particular,
$$ \gamma < \min F\left(R^2-r^2\right)^{-1/2}$$
will work.  This calculation shows that $U$ is a subsolution of \eqref{eqn: CLML} when $\partial\Omega_0$ is smooth enough that $Du_0$ is defined on $\partial\Omega_0$.  The calculation can be transferred to the test functions to show that $U$ is a subsolution in the general case.  

2. Next we construct a supersolution which takes the initial data continuously.  Let $\alpha=\alpha(r/2,R)\in(0,1)$ from Lemma~\ref{lem: holder in lip domain} combined with Lemma~\ref{lem: sstolip}.  Define, for any $h>0$,
\begin{equation}\label{eqn: perron supersoln}
 V_h(x,t) = u\left[ \bigcap _{|z| \leq \rho(t)}(1+h)\Omega_0+z\right]
 \end{equation}
 where 
$$\rho(t) = a\left(h-\dfrac{1}{a}\max\{F\left(Ch^{\alpha-1}\right),1\}t\right) \ \text{ for } \ 0 \leq t \leq \max\{F\left(Ch^{\alpha-1}\right),1\}^{-1}
$$
  Here $0<a<r$ so that $u_0 \prec V_h$ as in Lemma \ref{lem: sstolip}.  It follows from Lemma \ref{lem: spaceconvs} that $V_h$ satisfies
 $$ \frac{\partial_tV_h}{|DV_h|}(x,t) \geq \max\{F\left(Ch^{\alpha-1}\right),1\} \quad \text{ on } \quad \Gamma_t(V_h) $$
 in the viscosity sense. Let $x \in \Gamma_t(V_h)$, then there exists $y \in \partial B_{\rho(t)}(x) \cap (1+h)\Gamma_0$ and $\Omega_t(V_h)$ has the exterior ball at $x$
 $$ B_{\rho(t)}(y) \subset \real^N\setminus\Omega_t(V_h).$$
Then, for $h$ sufficiently small, $\Omega_t(V_h)$ is an intersection of domains star-shaped with respect to $B_{r/2}(0)$ and we apply Lemma \ref{lem: bdry gradient est} and then Lemma \ref{lem: holder in lip domain} to get,
$$ |DV_h|(x,t) \leq Ch^{-1}\sup_{|z|\leq 2h}V_h(x+z,t) \leq  C h^{\alpha-1} $$
and it follows from the monotonicity of $F$ that $V_h$ is a supersolution.
 
 3. Now we apply Perron's method.  Let us define
 $$ u(x,t) := \sup \{ v: \real^N\times [0,+\infty) \to [0,+\infty): v \text{ is a subsolution of \eqref{eqn: CLML} with initial data } u_0 \}. $$ 
 From its definition and Theorem \ref{thm: strictly separated comparison} we get the ordering for all $h>0$ 
 \begin{equation}\label{initial_order}
 \begin{array}{lll}
  U(x,t) \leq u(x,t) \leq V_h(x,t) & \text{for} & 0 \leq t \leq \max\{F\left(Ch^{\alpha-1}\right),1\}^{-1}. 
  \end{array}
  \end{equation}
  Let us now define the upper and lower semi-continuous envelopes of $u$, respectively
 $$ u^*(x,t) = \limsup_{(y,s) \to (x,t)} u(y,s) \quad \text{ and } \quad u_*(x,t) = \liminf_{(y,s) \to (x,t)} u(y,s).$$
 It is standard in viscosity solution theory (see \cite{user}) to show that $u^*$ is a subsolution and $u_*$ is a supersolution.  Due to \eqref{initial_order} one can check that $u(\cdot,t) \to u_0(\cdot)$ uniformly as $t \to 0$, i.e. $u^*=u_*$ at $t=0$.  Thus, letting $t_0=t_0(r,R)>0$ as in  Lemma \ref{lem: strong comparison}, we have $u^* \leq u_*$ on $[0,t_0]$. Combining this with the fact $u_*\leq u^*$, we obtain the existence of a unique continuous viscosity solution of \eqref{eqn: CLML} on $[0,t_0]$. 
 
\end{proof}
We point out that the above proof yields a modulus of continuity in time of the contact line at $t=0$.  For solutions which are uniformly strongly star-shaped on a time interval $[0,T]$ this extends to give a uniform modulus of continuity in time for the contact line on $[0,T]$.  We make this explicit in the case when $F$ has polynomial growth below.  For the rest of this section we will make the following assumption:
\begin{hyp}\label{hyp: polynomial growth}
There exists $p>0$ such that:
\begin{equation}\label{eqn: polynomial growth}
\limsup_{s\to\infty}\frac{F(s)}{s^p}<+\infty 
\end{equation}
\end{hyp}
We prove a modulus continuity for strongly star-shaped solutions of \eqref{eqn: CLML} depending only on $r$, $R$, $p$, $\sup \lambda$ and $\sup F(s)/s^p$:
 \begin{cor}\label{cor: equicontinuity}
  Suppose $u:\mathbb{R}^N\times[0,T]\to [0,+\infty)$ is a viscosity solution of \eqref{eqn: CLML} with $0<\lambda(t)\leq \Lambda$ bounded and $F$ satisfying Assumptions \ref{hyp: F incr cond}, \ref{hyp: F cond} and \ref{hyp: polynomial growth}.  Moreover suppose that $\Omega_t(v)\in \mathcal{S}_{r,R}$ for some $r,R>0$ and all $t\in[0,T]$. Let $\alpha$ from Lemma \ref{lem: bdry gradient est} and 
 $$\beta:=\tfrac{1}{1+p(1-\alpha)},$$ 
 then $\Gamma_t(u)$ is $C^\beta$ in time in the Hausdorff metric, i.e.,
  $$ d_H(\Gamma_{t_1}(u),\Gamma_{t_2}(u))\leq C_2|t_1-t_2|^{\beta} \hbox { for all } t_1, t_2 \in [0,T].
  $$
  Here $C_2$ depends only on $\sup F(s)/s^p$, $p$, $\Lambda$ and $N$.
 \end{cor}
     \begin{rem}
   An analogous equicontinuity result is true for general $F$ satisfying Assumptions~\ref{hyp: F incr cond} and \ref{hyp: F cond}.   In this case the H\"{o}lder modulus of continuity will have to be replaced by a more general modulus of continuity depending on the growth of $F$ at infinity.  We restrict to the polynomial growth case for simplicity of presentation.
   \end{rem}
  
 \begin{proof}
Let $T>t_2>t_1>0$ and $\gamma$ from the proof of Theorem \ref{thm: short time existence}.  Using the barriers from the previous theorem,
\begin{equation}
u(x,t_2) \geq U(x,t_2) := \left \{ \begin{array}{lll}
        (1-\gamma t)^2u(\frac{x}{1-\gamma t},t_1) & \text{ for} & t-t_1<\gamma^{-1} \\ \\
        0 & \text{ for} & t-t_1>\gamma^{-1}.
         \end{array}\right.
\end{equation}
Also for
$h=C(t_2-t_1)^{\beta}$ we have that
\begin{equation}
u(x,t_2) \leq V_h(x,t_2) \leq (1+h)^2 u\left(\frac{y}{1+h},t_1\right).
\end{equation}
Then the following containments hold:
$$ [(1-\gamma(t_2-t_1))\vee0]\Omega_{t_1}(u)\subseteq\Omega_{t_2}(u)\subseteq (1+C(t_2-t_1)^{\beta})\Omega_{t_1}(u). $$
So if $x \in \Gamma_{t_1}(u)$ and $(t_2-t_1)$ small, 
  $$ \{y: |y-x| \leq CR(t_2-t_1)^{\beta}\}\cap \Gamma_{t_2}(u) \neq \emptyset $$
  and 
   \begin{equation}\label{eqn: bdry holder in time}
    d_H(\Gamma_{t_1}(u),\Gamma_{t_2}(u))\leq C|t_1-t_2|^{\beta}. 
    \end{equation}
  
   \end{proof}

   \begin{rem}
   See Lemma \ref {lem: hausdorff estimates} for some implications of \eqref{eqn: bdry holder in time}.
   \end{rem}
 
We will be able to use this Corollary to show existence of a continuous viscosity solution for \eqref{eqn: CLML} when $F$ has polynomial growth.  The idea is to take the unique solution $u_M$ for the problem with the free boundary velocity $\max\{F,M\}$ and let $M \to \infty$.  Corollary \ref{cor: equicontinuity} will give equicontinuity to take a convergent subsequence of the $u_M$ as long as the $\Omega_t(u_M)\in\mathcal{S}_{r,R}$ with $r,R$ uniform in $M$ and $t$.  In order to show this we will need some kind of preservation of the strongly star-shaped property which does not depend on $M$.  

\begin{lem}\label{lem: sending M to infty}
Let $T>0$ and $u_n:\mathbb{R}^N\times[0,T]\to [0,+\infty)$ be viscosity solutions of:
\begin{equation}\label{eqn: approx problem}
\left\{
\begin{array}{lll}
-\Delta u_n(x,t) = \lambda_n(t) & \text{ in } & \Omega_t(u_n), \\ \\
 \partial_tu_n = F_n(|Du_n|)|Du_n| & \text{ on } & \Gamma_t(u_n). \\ \\
 \Omega_0(u_n) = \Omega_0.
\end{array}\right.
\end{equation}
Here $\lambda_n \to \lambda>0$ uniformly so that in particular there exists $\Lambda>0$ so that:
$$\sup_n \sup_t \lambda_n(t) \leq \Lambda<+\infty. $$
We suppose that $F_n \leq F$ for all $n$ and $F_n \to F$ uniformly on compact sets.  The limiting free boundary speed $F$ has polynomial growth of order $p>1$,
$$ 0<\limsup_{s\to \infty}  \frac{F(s)}{s^p}<+\infty. $$
The $F_n$ are assumed to all satisfy Assumptions \ref{hyp: F incr cond}, \ref{hyp: F cond} and \ref{hyp: polynomial growth}, with Assumption~\ref{hyp: polynomial growth} satisfied uniformly :
\begin{equation}
\begin{array}{lll}
 \sup_n \sup_s \frac{F_n(s)}{s^p}<+\infty.
 \end{array}
 \end{equation}
 Moreover, suppose that there exist $r,R>0$ such that for all $n$ and all $t \in [0,T]$ we have $\Omega_t(u_n) \in \mathcal{S}_{r,R}$.  However note that we do NOT assume that the constants $c_n$ from Assumption \ref{hyp: F cond} are uniformly bounded.  Then the $u_n$ converge uniformly on $\real^N \times [0,T]$ to a viscosity solution of:  
 \begin{equation}\label{eqn: limit pde}
\left\{
\begin{array}{lll}
-\Delta u(x,t) = \lambda(t) & \text{ in } & \Omega_t(u), \\ \\
 \partial_tu = F(|Du|)|Du| & \text{ on } & \Gamma_t(u). \\ \\
 \Omega_0(u) = \Omega_0.
\end{array}\right.
\end{equation}
If additionally $\lambda_n(t) \leq \lambda(t)$ for all $n$ then $u$ is the minimal viscosity solution \eqref{eqn: limit pde}.  In particular, in this case, the limit does not depend on the approximating sequence $F_n$.
\end{lem}
\begin{proof}
1. First we show that $\Omega_t(u_n)$ is an equicontinuous sequence of paths in $\mathcal{S}_{r,R}$.  From Lemma \ref{lem: hausdorff estimates} and Corollary \ref{cor: equicontinuity} we derive for $t,s>0$:
$$\sup_nd_H(\Omega_t(u_n),\Omega_s(u_n))\leq\sup_nd_H(\Gamma_t(u_n),\Gamma_s(u_n)) \leq C|t-s|^\beta. $$
Then from the compactness lemma \ref{lem: compactness} for paths in $\mathcal{S}_{r,R}$ up to a subsequence $\Omega_t(u_n)$ converge uniformly to a continuous path $\Omega_t$ in $\mathcal{S}_{r,R}$ with free boundary $\Gamma_t := \partial\Omega_t$.  

\medskip

Let $u(x,t)$ be the solution of 
\begin{equation}
\begin{array}{lll}
 -\Delta u(x,t) = \lambda(t) & \text{ for } & x \in \Omega_t \\ \\
 u(x,t) = 0 & \text{ for } & x \in \Gamma_t.
 \end{array}
 \end{equation}
 We claim that $u_n$ converges uniformly to $u$ in $\mathbb{R}^N\times[0,T]$, we will demonstrate this by showing the following estimate:
 $$ |u_n-u|(x,t) \lesssim_{r,R}  d_H(\Omega_t(u_n),\Omega_t(u))^\alpha+|\lambda_n(t)-\lambda(t)|. $$
   Let $x \in \Omega_t(u) \Delta \Omega_t(u_n)$, then:
 $$d(x,\partial(\Omega_t (u)\cup \Omega_t(u_n))) \leq d_H(\Gamma_t(u_n),\Gamma_t(u)) \lesssim_{r,R} d_H(\Omega_t(u_n),\Omega_t(u)). $$
 Since $u_n-u=0$ on $\partial(\Omega_t(u) \cup \Omega_t(u_n))$ we get from Lemma \ref{lem: holder in lip domain},
 $$|u_n-u|(x,t) \leq C d(x,\partial(\Omega_t(u) \cup \Omega_t(u_n)))^\alpha \lesssim_{r,R}d_H(\Omega_t(u_n),\Omega_t(u))^\alpha. $$
 Meanwhile for $x \in \Omega_t(u) \cap \Omega_t(u_n)$ we combine the above inequality which holds on the boundary of $\Omega_t(u) \cap \Omega_t(u_n)$ with the fact that $\text{diam}(\Omega_t(u) \cap \Omega_t(u_n)) \lesssim_R 1$ to get,
 $$ |u_n-u|(x,t) \lesssim d_H(\Omega_t(u_n),\Omega_t(u))^\alpha+ |\lambda_n(t)-\lambda(t)|.  $$
 We apply the stability of viscosity solutions under uniform convergence, Lemma \ref{lem: uniform stability}, to see that $u$ is a viscosity solution of the PDE \eqref{eqn: limit pde} as claimed. 
 
 \medskip
 
 2.   Now we show that if $\lambda_n \nearrow \lambda$ then $u$ must be the smallest viscosity solution of \eqref{eqn: limit pde}.  Let $v$ be another viscosity solution of \eqref{eqn: limit pde}.  Note that due to the orderings $F_n \leq F$ and $\lambda_n \leq \lambda$ we have that $v$ is a supersolution of each of the approximating problems \eqref{eqn: approx problem}.  Using the strong comparison principle Lemma~\ref{lem: strong comparison} which holds for each problem \eqref{eqn: approx problem} we get
 $$ u_n \leq v \hbox{ on } \real^N\times[0,T]. $$

\end{proof}

\subsection{Preservation of the Strongly Star-shaped Property}\label{sec: Preservation of the Strongly Star-shaped Property}

Here we describe some of the properties of the viscosity solutions of \eqref{eqn: CLML} and \eqref{eqn: CLMV}.  First, we will show that, if the droplet initially is in $\mathcal{S}_0$, then this property persists for a short time with the radius of the strong star-shaped property going to zero in some finite amount of time. This short-time regularity will not be sufficient to prove any kind of long-time behavior. However, in Proposition \ref{prop: reflection comparison} we have shown that a reflection comparison principle holds as long as the positivity set is in $\mathcal{S}_{r,R}$ for any $r,R>0$. The reflection comparison principle in turn will allow us to show, in some situations, that actually there was no loss of star-shapedness.  In particular for \eqref{eqn: CLMV} we will show that initial data which is in $\mathcal{S}_{r}$ for a sufficiently large $r$ along with a condition on the smallness of 
$$\sup_{x\in \Omega_0} |x|-\inf_{x\in \Omega_0} |x|$$
 will maintain some regularity globally in time.  In fact, it will be strongly star-shaped for all time with a possibly smaller radius.  
 
 \medskip
 
 We show that strong star-shapedness cannot disappear immediately.  The below Lemma is essentially contained in \cite{GlasnerKim09}, we use a different proof. 

 \begin{lem}\label{lem: GK short time ss} 
\textup{(Short time strong star-shapedness)} Let $u:\real^N\times[0,+\infty)\to[0,+\infty)$ be the solution of \eqref{eqn: CLML} with initial positive phase $\Omega_0\in \mathcal{S}_{r}$ for some $r>0$. Then
$$ \Omega_t \in \mathcal{S}_{r-ct} \  \hbox{ for }  \ 0 \leq t < \frac{r}{c} $$
 where $c$ is from Assumption \ref{hyp: F cond}.
\end{lem}

 \begin{proof}  
 Let $\epsilon_0>\epsilon>0$ and $0<a<r$, then from Lemma \ref{lem: sstolip} and Assumption \ref{hyp: F cond} we know that
 \begin{equation*}
 \begin{array}{lll}
  u_\epsilon(x,t) := \inf_{|z|\leq a\epsilon-c\epsilon t} (1+\epsilon)^2u((1+\epsilon)^{-1}(x+z),t) & \text{ for } & t \leq \frac{a}{c}
  \end{array}
  \end{equation*}
  is a supersolution of \eqref{eqn: CLML} which has $u(\cdot,0) \prec u_{\epsilon}(\cdot,0)$.  Therefore from the strict comparison result Theorem~\ref{thm: strictly separated comparison} $u \leq u_\epsilon$ for $t < a/c$ and in particular,
  $$ \Omega_t(u) \subset \Omega_t(u_\epsilon) = \bigcap_{|z|\leq(a-ct)\epsilon} ((1+\epsilon)\Omega_t(u)+z) \hbox{ for every } \epsilon_0>\epsilon>0. $$
  Since $\epsilon$ and $0<a<r$ were arbitrary the converse direction of Lemma \ref{lem: sstolip} part (iv) implies that $\Omega_t(u)$ has the claimed star-shapedness.
 
 \end{proof}

 We would like to show $\Omega_t(u) \in \mathcal{S}_{r}$ as long as $B_r(0)$ is contained in the positive phase $\Omega_t(u)$.  This kind of preservation of the star-shapedness property would be very useful because we could get regularity results by simply showing that the positivity set must always contain some small ball around the origin.  However, the arguments of Lemma \ref{lem: GK short time ss} are insufficient to prove such a result.  This is where the reflection comparison principle Proposition \ref{prop: reflection comparison} is useful.  The key fact is that reflection comparison holds as long as $\Omega_t(u) \in \mathcal{S}_{r}$ for any $r>0$. 
 
 \medskip
 
 We begin by defining a slightly stronger property than strong star-shapedness which is defined in terms of reflections.  In order to do this we will need some notations.  Let $\Omega$ be an open bounded domain in $\mathbb{R}^N$.  For $\nu \in S^{N-1}$, let $H=H(\nu)$ be the hyperplane through the origin orthogonal to $\nu$ with the half spaces it defines:
 $$H_+=H_+(\nu) = \{x: x\cdot \nu > 0\}, \quad H_- =H_-(\nu)= \{x: x\cdot \nu < 0\} .
 $$ 
    Now for $s>0$ define the translates,
$$ H(s) = H+s\nu, \quad H_+(s) = H_++s\nu, \quad H_-(s) = H_-+s\nu.$$
For $s$ sufficiently large
$$\Omega \subset H_-(s)  $$
and so trivially 
\begin{equation}\label{eqn: sliding plane}
\phi_{H(s)}(\Omega\cap H_+(s)) \subset \Omega\cap H_-(s).  
\end{equation}
See Figure~\ref{fig: reflections} for a depiction of the situation.  Now let us slide the hyperplane inwards towards the origin by decreasing $s$ until \eqref{eqn: sliding plane} no longer holds.  We will call $s_{\text{min}}(\nu,\Omega)$ the closest to the origin that we can move this plane with \eqref{eqn: sliding plane} always holding. More precisely, we define
\begin{equation}
s_{\text{min}}(\nu,\Omega) := \inf \{ s>0: \phi_{H(t)}(\Omega\cap H_+(t)) \subset \Omega\cap H_-(t) \text{ for all } t>s\}.
\end{equation}
In the following we omit the dependence of $s_{\text{min}}$ on $\Omega$ when it will not present any confusion.  If $s_{\text{min}}(\nu) = 0$ for every direction $\nu$ then $\Omega$ must be a ball.  The basic approach of Serrin in \cite{Serrin71} to showing the symmetry of solutions of \eqref{eqn: EQ} is to show that $s_{\text{min}}(\nu) = 0$ for all $\nu$.  In our case it is useful to use this same idea of symmetry but in a weaker form.

\begin{figure}[t]
\begin{center}
\includegraphics[scale=0.5]{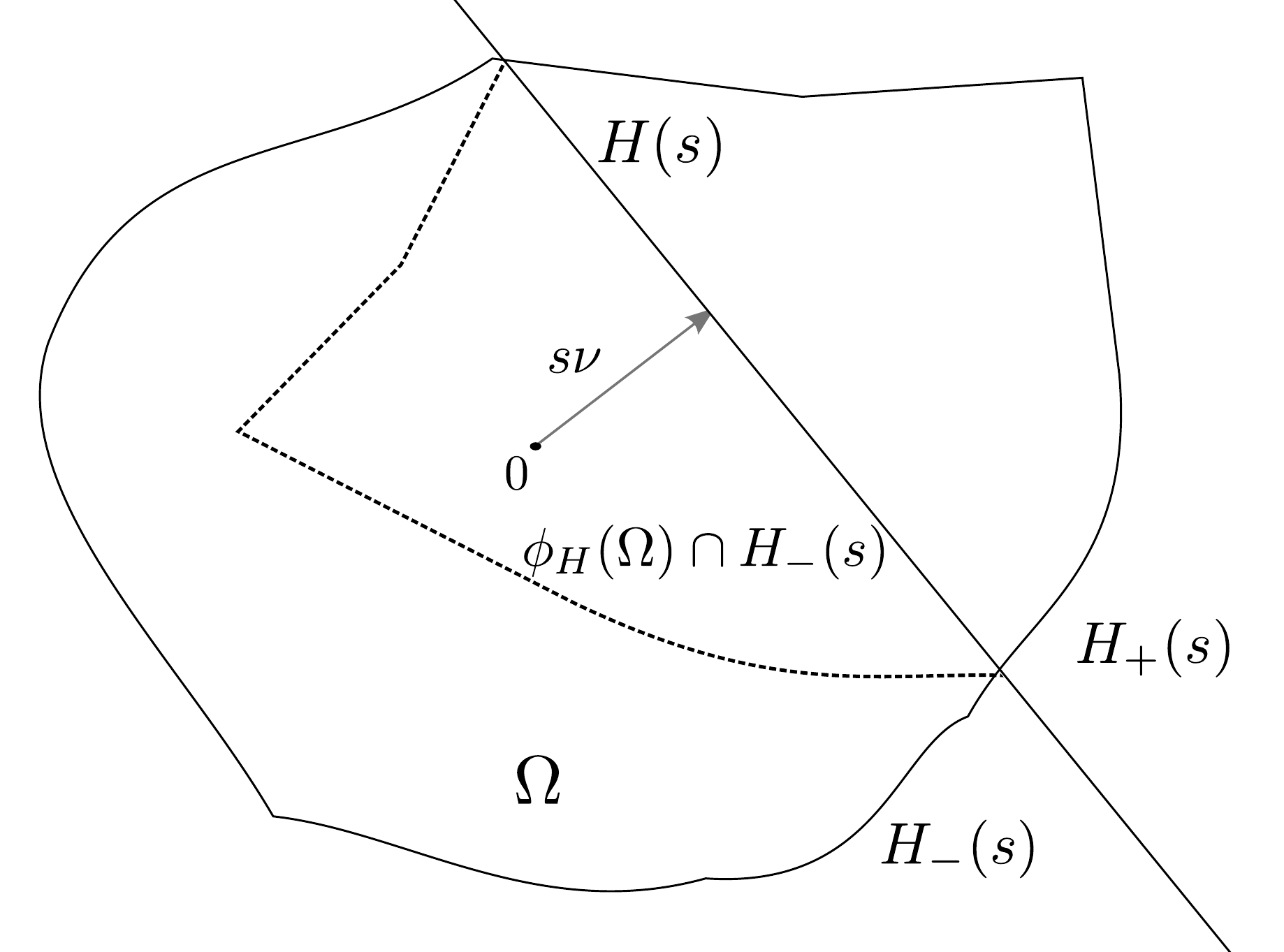}
\end{center}
\caption{Reflection comparison will hold.}
\label{fig: reflections}
 \end{figure}

\begin{DEF}\label{rho_reflection}
We say that a bounded, open set $\Omega$ has \textit{$\rho$-reflection} if $B_\rho(0) \subseteq \Omega$ and
\begin{equation}\label{eqn: condition}
 \sup_{\nu \in S^{n-1}} s_{\text{min}}(\nu) \leq \rho. 
 \end{equation}
\end{DEF}
\begin{rem}
As with strong star-shapedness, given some initial data $\Omega_0$, we will in general assume that 
$$ \sup_{\nu \in S^{N-1}} s_{\text{min}}(\nu,\Omega_0) = \inf_{z \in \real^N} \sup_{\nu \in S^{N-1}} s_{\text{min}}(\nu,\Omega_0+z).$$
Of course if this is not the case it can be corrected by a spatial translation.  
\end{rem}
\begin{rem}
One can show that (see Lemma \ref{lem: lip norm closeness}), upon fixing a maximal diameter $R$, a set has $\rho$-reflection if its boundary normals deviate from the radial direction by $O(\rho/R)$ and additionally the oscillation of the boundary is $O(\rho^2/R)$. Although we would like to know whether the flow preserves convexity, the $\rho$-reflection property does not seem to be very useful in this direction.  In particular $\rho$-reflection does not imply convexity, and neither does the converse hold. 
\end{rem}

 Note that the property \eqref{eqn: condition} is preserved  over time due to the reflection comparison, as long as the positivity set $\Omega_t(u) \in \mathcal{S}_{r}$ for any $r>0$.  We list some basic facts about sets which have $\rho$-reflection.  The proofs are postponed to the appendix.  The first says that the condition $\rho$-reflection imposes a condition on the spatial location of the boundary of the set.

\begin{lem}\label{lem: in an annulus}
Suppose $\Omega$ has $\rho$-reflection, then 
\begin{equation}\label{condition}
 \sup_{x \in \partial\Omega} |x|-\inf_{x\in \partial\Omega}|x| \leq 4\rho. 
 \end{equation}
\end{lem}

We mention that \eqref{eqn: condition} imposes a bound on what kind of normal vectors $\partial\Omega$ can have.  In fact $\Omega$ will be strongly star-shaped with radius depending only on $\text{dist}(\partial\Omega,B_\rho(0))$.   

\begin{lem}\label{lem: almost radiality}
Suppose $\Omega$ has $\rho$-reflection. Then $\Omega$ satisfies the following:

\begin{itemize}
\item[(a)] for all $x\in \partial\Omega$ there is an exterior cone to $\Omega$ at $x$,
\begin{equation}
\begin{array}{lll}
 x+C\left(x,\phi_x\right) \subset \real^N\setminus\Omega & \text{ where } & \cos\phi_x = \frac{\rho}{|x|}, \ \phi_x \in (0,\pi/2),
 \end{array}
 \end{equation}
 and $C(x,\phi_x)$ is the cone in direction $x$ of opening angle $\phi_x$ as defined in \eqref{eqn: cone def};
\item[(b)] $\Omega\in \mathcal{S}_{r}$ where
$$ r = r(\rho,\inf_{x\in \partial\Omega}|x|) = (\inf_{x\in \partial\Omega}|x|^2-\rho^2)^{1/2}.$$
\end{itemize}
\end{lem}
 Now, combining the strong star-shapedness from Lemma \ref{lem: almost radiality} with the reflection comparison principle Proposition \ref{prop: reflection comparison}, we obtain that $\rho$-reflection is preserved as long as the $B_\rho(0)$ is contained in evolving positive phase.
\begin{lem}\label{lem: preservation}
Suppose $u:\mathbb{R}^N\times[0,+\infty)\to[0,+\infty)$ is a solution of \eqref{eqn: CLML} and $\Omega_0$ has $\rho$-reflection for some $\rho>0$. Let $I = [0,T)$ be the maximal time interval containing $0$ on which $\overline{B_\rho(0)}\subseteq \Omega_t(u)$.  Then $\Omega_t(u)$ has $\rho$-reflection and in particular $\Omega_t(u) \in \mathcal{S}_{r}$ for some $r>0$ for every $t \in I$.
\end{lem}
\begin{proof}
Note that by comparing with a radial subsolution placed below $u(\cdot,0)$ one can see that $T>0$.  Suppose towards contradiction that the lemma fails, that is
$$ t_* := \inf \{ t \in [0,T): \Omega_t(u) \text{ does not have $\rho$-reflection }\}<T. $$
  Since $\overline{B_\rho(0)}\subseteq \Omega_{t_*}(u)$ we have that 
$$ h:=\text{dist}(\Omega_{t_*}(u)^C, B_\rho(0))>0$$
so that by Lemma \ref{lem: almost radiality}, $\Omega_{t_*}(u)$ is star-shaped with respect to $B_{(h^2+2\rho h)^{1/2}}(0)$. Then by Lemma \ref{lem: GK short time ss} 
\begin{equation}
\begin{array}{lll}
\Omega_t(u) \in \mathcal{S}_{r} & \text{ for } & r = \frac{1}{2}(h^2+2\rho h)^{1/2}
\end{array}
\end{equation}
on some slightly larger interval $[t_*,t_*+r/c)$. Now let $H$ be a hyperplane which does not intersect $B_\rho(0)$.  We can apply the reflection comparison Proposition \ref{prop: reflection comparison} to $u(x,t)$ on $[t_*,t_*+r/c)$ to see that,
$$ u(\phi_H(x),t) \leq u(x,t) \text{ for } (x,t) \in H_+ \cap [t_*,t_*+r/c)$$
and therefore
$$\phi_H(\Omega_t(u))\cap H_+ \subseteq  \Omega_t(u)\cap H_+. $$
This holds for all admissible $H$, so $\Omega_t(u)$ has $\rho$-reflection on $[t_*,t_*+r/c)$, contradicting the definition of $t_*$.
\end{proof}
Notice that an immediate consequence of Lemma~\ref{lem: preservation} is that for any initial data $\Omega_0$ which has $\rho$-reflection the evolution $\Omega_t$ has $\rho$-reflection for at least as long as,
$$  t \leq \frac{\min_{x\in\Gamma_0}|x|-\rho}{\min F}. $$
The time on the right hand side above is independent of $c$ from Assumption \ref{hyp: F cond} and of $\lambda(t)$.

\medskip

Now we present an application of this idea to the volume preserving problem, showing that for certain initial data any solution of the volume preserving problem must have $\rho$-reflection for all time.  One can think of this as an a priori estimate for solutions of \eqref{eqn: CLMV}. If the initial data $\Omega_0$ has $\rho$-reflection then Lemma \ref{lem: preservation} says that the solution $\Omega_t$ of the volume preserving flow also has $\rho$-reflection until such a time as $\Omega_t$ touches $B_\rho(0)$ from the exterior.  For a domain $\Omega \subset \real^N$ let us define the associated Lagrange multiplier,
$$ \lambda[\Omega] = \inf \{ \int |Dv|^2 : v \in H^1_0(\Omega) \ \hbox{ and } \ \int v = V\}.$$
Then from Lemma~\ref{lem: in an annulus} at the touching time $t_*$ we know that $\Omega_{t_*} \subset B_{5\rho}(0)$.  This allows us to get a lower bound on the Lagrange multiplier $\lambda[\Omega_{t_*}]$ and show that $B_{\rho}(0)$ is a strict subsolution near the touching time in the sense that
$$ \frac{\lambda[\Omega(t)]}{2N}(\rho^2-|x|^2) $$
is a strict classical subsolution for $|t-t_*|$ small.  Of course, this will be a contradiction of the supersolution property of $\Omega_t$.  We make this precise in the following Lemma. 

\begin{lem}\label{lem: reflection for all time}
Suppose $u:\mathbb{R}^N\times[0,T]\to[0,+\infty)$ is a solution of \eqref{eqn: CLMV}  with initial data $\Omega_0$ that has $\rho$-reflection with $B_{\rho} \subset\subset \Omega_0$. Then there is a dimension constant $C_N$  such that if
\begin{equation}\label{eqn: rho small}
 0<\rho < C_NV^{\frac{1}{N+1}}
 \end{equation}
then there exists $a>0$ such that $B_{(1+a)\rho}(0)\subseteq\Omega_t(u)$ for all $t>0$.  In particular, we establish that
$$ C_N \geq \frac{1}{5^{\frac{N+2}{N+1}}}\left(\frac{N(N+2)}{|S^{N-1}|}\right)^{\frac{1}{N+1}}. $$
\end{lem}

\begin{rem}
The scaling $ \rho \lesssim V^{\frac{1}{N+1}}$ is the natural one for such an inequality.  In particular, if $\rho$ is larger than the unique radial stationary solution,
$$\rho >r_* = \left(\frac{N}{N+2}\right)^{\frac{1}{N+1}}V^{\frac{1}{N+1}} $$
then we can construct a counter-example using radially symmetric solutions of \eqref{eqn: CLMV}.  The solution of \eqref{eqn: CLMV} with initial data $B_{r}(0)$ for any $r>\rho$ has $\rho$-reflection and converges as $t \to \infty$ to $B_{r_*}(0)$ so, in particular, there is some time after which $B_\rho(0)$ is no longer contained in the evolving positivity set.  Of course in this case the initial data has $\rho$-reflection for every $\rho>0$.
\end{rem}

\begin{proof}
From the compact containment $B_\rho \subset\subset \Omega_0$ along with the strict inequality \eqref{eqn: rho small} let $a>0$ small so that:
\begin{equation*}
\begin{array}{lll}
B_{(1+a)\rho}(0) \subset\subset \Omega_0 & \text{ and } & \rho < \left(\frac{N(N+2)}{(1+a+4)^{N+2}|S^{N-1}|}\right)^{\frac{1}{N+1}}.
\end{array}
\end{equation*}
 Towards a contradiction assume that $B_{(1+a)\rho}(0)$ touches $\Omega_t(u)$ from the inside for the first time at $0<t_*<T$.  More precisely let
$$t_* = \inf \{T>t >0: B_{(1+a)\rho(0)}\cap \Omega_t(u)^C \neq \emptyset\}$$
where we have assumed that the set being infimized over is non-empty.  Note that $t_*>0$ by comparing $u$ with a radial subsolution starting on some ball slightly larger than $B_{(1+a)\rho}$ but still contained in $\Omega_0$ and by continuity of the free boundary (Corollary \ref{cor: equicontinuity} for example applies)
$$ B_{(1+a)\rho}(0) \subseteq  \Omega_{t_*}(u)  \text{ and } \exists \ x_*\in \partial B_{(1+a)\rho}(0) \cap \Gamma_{t_*}(u).$$
Now by Lemma \ref{lem: preservation} and 
$$ B_\rho(0) \subset\subset \Omega_t(u) \text{ for } t\in [0,t_*+]$$
we have that $\Omega_t(u)$ has $\rho$-reflection and therefore by Lemma \ref{lem: almost radiality}
$$ B_{(1+a)\rho}(0) \subseteq \Omega_t(u) \subseteq B_{(1+a+4)\rho}(0).$$
Therefore we have 
$$ \lambda(t_*) > \lambda[B_{(1+a+4)\rho}] = \frac{N^2(N+2)V}{|S^{N-1}|((1+a+4)\rho)^{N+2}}.$$
Now, define
$$h(x) := \frac{N(N+2)V}{2|S^{N-1}|((1+a+4)\rho)^{N+2}}(\rho^2-|x|^2), $$
which satisfies $-\Delta h(x) < \lambda(t_*)$ where $h>0$ and $h$ touches $u(x,t)$ from below at $(x_0,t_*)$.  On its free boundary, due to the assumption on $\rho$ and $a$, $h$ satisfies
$$ F(|Dh|(x)) = F\left(\frac{N(N+2)V(1+a)\rho}{|S^{N-1}|((1+a+4)\rho)^{N+2}}\right)> 0 .
$$ 
The barrier $h$ is a strict classical subsolution which touches $u(x,t)$ from below at $(x_*,t_*)$, a contradiction of $u$ being a viscosity supersolution. 
\end{proof}

The following holds now due to Lemma~\ref{lem: preservation} and Lemma~\ref{lem: reflection for all time}.

\begin{cor}\label{reflectiontime}
Let $u$ and $\Omega_0$ be given as above. Then $\Omega_t(u)\in\mathcal{S}_{r,R}$ for some $r,R>0$ for all $t>0$.
\end{cor}

\section{Gradient Flow}\label{sec: Gradient Flow}

Now we consider the capillary droplet problem, contact line motion with volume preservation.  We recall the PDE here for convenience,
\begin{equation}\tag{\ref{eqn: CLMV}}
\left \{ \begin{array}{lll}
         -\Delta u(x,t)  = \lambda(t) & \text{ in }&  \Omega(u)\\ \\
         u_t = F(|Du|)|Du| & \text{ on } & \Gamma(u).
         \end{array}\right.
\end{equation}
We will consider the problem when
\begin{equation}\label{eqn: FBV unbdd}
F(|Du|) = |Du|^2-1.
\end{equation}
We restrict ourselves to the above choice of $F$ because it gives the problem a simple gradient flow structure, although we expect similar results to hold for general $F$ which satisfy Assumption~\ref{hyp: F incr cond}.  The viscosity solution theory developed in the previous section does not apply directly to this choice of free boundary velocity.  Therefore we will instead deal with
\begin{equation}\label{eqn: FBV bdd}
F(|Du|) = \max\{|Du|^2-1,M\}.
\end{equation}
for $M>0$ large and send $M \to \infty$ to get results for \eqref{eqn: FBV bdd}.  As mentioned in the introduction, the problem \eqref{eqn: CLMV} with free boundary velocity \eqref{eqn: FBV unbdd} is formally a gradient flow of the energy,
\begin{equation}\tag{\ref{eq: energy}}
\mathcal{J}(\Omega) := \int_\Omega |Du|^2 +|\Omega|.
\end{equation}
in the space of compact subsets of $\real^N$ (we leave the metric unspecified for now).  Above we call
$$ u = u[\Omega] : = \argmin \{\int |Dv|^2: v \in H^1_0(\Omega), \ \ \int_{\mathbb{R}^N} v = V \}, $$
Recall that $u_0$ solves,
\begin{equation}
\left \{ \begin{array}{lll}
         -\Delta u(x)  = \lambda[\Omega] & \text{ in }&  \Omega\\ \\
         u(x)=0 & \text{ on } & \partial\Omega,
         \end{array}\right.
\end{equation}
where the Lagrange multiplier $\lambda[\Omega]$ is given by
$$\lambda[\Omega]:= \min \{\int |Dv|^2: v \in H^1_0(\Omega), \ \ \int_{\mathbb{R}^N} v = V \}.$$

\medskip

Now we give the formal derivation of the gradient flow structure as motivation.  One can think of the Cacciopoli subsets of $\real^N$ as an infinite dimensional manifold such that the tangent space at $\Omega$ is
$$T_\Omega =  L^{\infty}(\partial\Omega, |D\chi_\Omega|), $$
and the metric is specified as
$$g_{\Omega} (f,g) := \int_{\partial\Omega} fg |D\chi_\Omega|. $$
Then suppose that $\Omega_t$ is a smooth set valued solution of \eqref{eqn: CLMV}, i.e. $u[\Omega_t]$ is a smooth viscosity solution.  In the geometric framework $\Omega_t$ is a path with velocity
$$ T_{\Omega_t} \ni v(t) = F(|Du|(t))=|Du|(t)^2-1 .$$
Then we calculate,
\begin{align*}
\frac{d}{dt}\mathcal{J}(\Omega_t) &= \int_{\Omega_t} 2Du\cdot Du_t+ \int_{\Gamma_t} (|Du|^2+1)F(|Du|) \\
&= -2\lambda(t)\int_{\Omega_t} u_t +\int_{\Gamma_t} 2u_t Du\cdot n+(|Du|^2+1)F(|Du|) \\
&= -2\lambda(t) \frac{d}{dt}\left(\int_{\Omega_t} u\right)+\int_{\Gamma_t} -2 F(|Du|)|Du|^2+(|Du|^2+1)F(|Du|) \\
&= \int_{\Gamma_t} (1-|Du|^2)F(|Du|) = -g_{\Omega_t}(v(t),v(t)).
\end{align*}
Of course, viscosity solutions of \eqref{eqn: CLMV} lack the required smoothness to make the above framework anything but formal.  Additionally, as far as we are aware, the distance associated with the metric $g$ has no nice form even in the smooth case.  Instead we use the pseudo-distance
\begin{equation}\label{eqn: dist tilde}
 \widetilde{\text{dist}}^2(\Omega_0,\Omega_1) = \int_{\Omega_0\Delta\Omega_1}d(x,\partial\Omega_0)dx.
 \end{equation}
 This pseudo-distance was used previously for the same problem in \cite{GrunewaldKim11} and for the mean curvature flow in \cite{Chambolle04} and is motivated by the $L^2$-based  Riemannian structure formally given to the space of sets with finite perimeters (see \cite{GrunewaldKim11}).  Unfortunately, $\widetilde{\text{dist}}$ is not a distance. For example, it does not satisfy the triangle inequality. We will address this deficiency in Lemma \ref{lem: SQTineq}.  
 
 \medskip
 
 Our plan is to show the existence of viscosity solutions for \eqref{eqn: CLMV} by constructing a gradient flow of the energy $\mathcal{J}$.   There is a standard approach to defining curves of maximal slope for energies in metric spaces through a discrete time approximation.  Some early works on this subject are \cite{JKO98 , LuckhausSturzenhecker95} and the recent book \cite{AGS08} contains a quite general treatment of the method.  As noted previously our problem does not live in a nice metric space, however the essential idea of the construction is the same.  In \cite{GrunewaldKim11} it is shown that the curves of maximal slope $\mathcal{J}$ with $\widetilde{\text{dist}}$ satisfy a barrier property with respect to strict smooth sub and supersolutions of \eqref{eqn: CLML}.  From the a priori estimate Lemma \eqref{lem: preservation} any solution with sufficiently round initial data of \eqref{eqn: CLMV} will be in $\mathcal{S}_r$ for all time.  With this in mind, it is reasonable to restrict the class of admissible sets in the gradient flow to only include sets which are uniformly strongly star-shaped with respect to the origin.  With this additional regularity it becomes possible to show that the solutions of the discrete gradient flow scheme will converge to a viscosity solution of \eqref{eqn: CLMV}.

\subsection{Discrete Gradient Flow Scheme}\label{sec: Discrete Gradient Flow Scheme}
We will now define the discrete JKO type scheme to construct approximate solutions of the gradient flow \eqref{eqn: CLMV}.  We make a notational choice to denote these solutions with $\omega$. This is to distinguish between solutions of the gradient flow and $\Omega(\cdot)$ which we use to denote the positive phase of a height profile.  Let $\omega_0$ be the initial positive phase and let $u_0 = u[\omega_0]$ be the corresponding initial height profile.  Under the assumption that $\omega_0\in \mathcal{S}_{r_0}$ for some $r_0>0$,  let us define a discrete scheme approximating the gradient flow of the droplet energy $\mathcal{J}$ as follows:  
\begin{DEF}\label{def: JKO scheme}
 Let $M,r,h>0$ we define $\omega_{0,h}=\omega_0$, and define iteratively, 
\begin{equation}\label{eqn: JKOscheme}
 \omega_{k+1,h} = \argmin\left\{\mathcal{J}(\omega)+h^{-1}\widetilde{\text{dist}}^2(\omega_{k,h},\omega): \omega \in \mathcal{A}^{h}(\omega_{k,h})\right\}.
 \end{equation}
 Here $\mathcal{A}^h(\cdot)$ is the class of admissible sets which we will take to be
\begin{equation}\label{eqn: admissible sets}
\mathcal{A}^{h}(\omega)  := \{U\in \mathcal{S}_{r_0} \text{ and } d_H(U,\omega)\leq Mh\}.
\end{equation}
\end{DEF}

We suppress the dependence $M$, $r_0$ and $h$ in the above notations when there will be no confusion.  In order to show that to show that the iteration defined in \eqref{eqn: JKOscheme} is well defined Now we apply the Compactness Lemma \ref{lem: compactness}:
\begin{lem}\label{lem: JKOdef}
The iteration \eqref{eqn: JKOscheme} is well defined.
\end{lem}
\begin{proof}
We show that for any $\omega_0 \in \mathcal{S}_{r,R}$ the infimum from \eqref{eqn: JKOscheme},
\begin{equation}\label{eqn: JKO inf}
 \inf \left\{\mathcal{J}(\omega)+h^{-1}\widetilde{\text{dist}}^2(\omega_{0},\omega): \omega \in \mathcal{A}^{h}(\omega_{0})\right\},
 \end{equation}
is achieved by some $\omega \in \mathcal{A}^{h}(\omega_{0})$.  Let $\omega_j\in \mathcal{A}^{h}(\omega_0)$ be an infimizing sequence for \eqref{eqn: JKO inf}.  In particular from the admissibility conditions \eqref{eqn: admissible sets},
$$ \omega_j \in \mathcal{S}_{r_0,R+Mh}. $$
Applying Lemma ~\ref{lem: compactness} up to a subsequence the $\omega_j$ converge in $d_H$ metric to some $\omega \in \mathcal{S}_{r_0,R+Mh}$ and we easily derive,
$$ d_H(\omega,\omega_0) \leq Mh, $$
so that $\omega \in \mathcal{A}^h(\omega_0)$.  Then, from Lemma~\ref{lem: hausdorff estimates}, $\mathcal{J}(\cdot)$ and $\widetilde{\text{dist}}^2(\omega_{0},\cdot)$ are continuous with respect to convergence in $(\mathcal{S}_{r_0,R+Mh},d_H)$ so $\omega$ achieves the infimum in \eqref{eqn: JKO inf}.

\end{proof}
Note that since $\omega_k$ is admissible at each stage we automatically get the discrete energy decay estimate,
\begin{equation}\label{eqn: discrete energy decay}
 \mathcal{J}(\omega_{k+1})+\frac{1}{h}\widetilde{\text{dist}}^2(\omega_{k+1},\omega_{k})\leq \mathcal{J}(\omega_{k}). 
 \end{equation}
We would like to use the triangle inequality (which we don't have) to get the energy decay inequality over time of the form:
$$\frac{1}{nh}\widetilde{\text{dist}}^2(\omega_{k+n},\omega_{k})\leq \mathcal{J}(\omega_{k})-\mathcal{J}(\omega_{n+k}).$$
It turns out that this is possible as long as we are restricting to strongly star-shaped sets.  

\begin{lem}\label{lem: SQTineq}
Let $\omega_0,...,\omega_n \in \mathcal{S}_{r,R}$, then we have the inequality,
\begin{equation}\label{eqn: triineq}
  \frac{C(r,R,N)}{n}\widetilde{\textup{dist}}^2(\omega_{n},\omega_1)
\leq \sum_{j=1}^n \widetilde{\textup{dist}}^2(\omega_{j+1},\omega_j)
\end{equation}
\end{lem}
In the following it will be useful to introduce some notation.  We say that $A \lesssim B$ if there exists a constant $C>0$ such that $A \leq CB$.  We say $A \sim B$ or $A$ is equivalent up to constants to $B$ if $A \lesssim B \lesssim A$.  The dependence of the constant $C$ on various parameters $q_i$ is expressed by writing $\lesssim_{q_1,...,q_n}$ or $\sim_{q_1,...,q_n}$.  If no dependence is expressed, then the constant is assumed to depend at most on the dimension $N$.

\begin{proof}
Let $\omega_a, \omega_b \in \mathcal{S}_{r,R}$, then we define the distance between the sets $\omega_a$ and $\omega_b$ in the direction $\theta \in S^{N-1}$,
\begin{equation*}
\begin{array}{lll}
d_\theta(\omega_a,\omega_b) = |X_a(\theta)-X_b(\theta)| & \text{ where } & \{X_i(\theta)\} = \partial\omega_i\cap \{x: \frac{x}{|x|} = \theta\}, \ i \in \{a,b\}.
\end{array}
\end{equation*}
Notice, due to the strict star-shapedness of the $\omega_i$, there is only one point $X_i(\theta)$ along the direction $\theta$ in $\partial\omega_i$.  It is easy to check that $d_\theta$ has the triangle inequality:
$$ d_\theta(\omega_a,\omega_b)  \leq d_\theta(\omega_a,\omega_c)+d_\theta(\omega_c,\omega_b),$$
for any bounded strictly star-shaped sets $\omega_i$, $i \in \{a,b,c\}$.  
\begin{Claim}
The $\widetilde{\text{dist}}$ is equivalent up to constants to a true distance:
\begin{equation}\label{eqn: equiv distances}
\widetilde{\text{dist}}(\omega_a,\omega_b) \sim_{r,R,N} \left(\int_{S^{N-1}}d_\theta(\omega_a,\omega_b)^2d\theta\right)^{1/2}.
\end{equation}
\end{Claim}
Using the claim along with Cauchy-Schwarz and the triangle inequality for $d_\theta$ we will get,
\begin{align*}
\sum_{j=1}^n \widetilde{\text{dist}}^2(\omega_{j+1},\omega_j)&\gtrsim_{r,R,N} \sum_{j=1}^n \int_{S^{N-1}}d_\theta(\omega_{j+1},\omega_{j})^2d\theta 
\geq \frac{1}{n} \int_{S^{N-1}}d_\theta(\omega_{n},\omega_{1})^2d\theta 
\gtrsim_{r,R,N} \widetilde{\text{dist}}^2(\omega_{n},\omega_1).
\end{align*}
So it suffices to show the claimed equivalence \eqref{eqn: equiv distances}.  We calculate:
\begin{align*}
\widetilde{\text{dist}}^2(\omega_a,\omega_b) &= \int_{\omega_a\Delta\omega_b}d(x,\partial\omega_a)dx 
=\int_{S^{N-1}}\int_0^R d(\rho\theta,\partial\omega_b){\bf1}_{\omega_a\Delta\omega_b}(\rho\theta)\rho^{N-1} \ d\rho d\theta \\
&\sim_{r,R,N} \int_{S^{N-1}}\int_0^R d(\rho\theta,\partial\omega_b){\bf1}_{\omega_a\Delta\omega_b}(\rho\theta) \ d\rho d\theta \\
& = \int_{S^{N-1}}\int_{\min\{|X_a(\theta)|,|X_b(\theta)|\}}^{\max\{|X_a(\theta)|,|X_b(\theta)|\}}d(\rho\theta,\partial\omega_b) \ d\rho d\theta \\
&\sim_{r,R} \int_{S^{N-1}}\int_{\min\{|X_a(\theta)|,|X_b(\theta)|\}}^{\max\{|X_a(\theta)|,|X_b(\theta)|\}}\left|\rho -|X_a(\theta)|\right| \  d\rho d\theta \\
& = \int_{S^{N-1}}\int_0^{d_\theta(\omega_a,\omega_b)}\rho \ d\rho  d\theta = \frac{1}{2}\int_{S^{N-1}}d_\theta(\omega_a,\omega_b)^2d\theta.
\end{align*}
Here we have used an easy consequence of Lemma \ref{lem: sstolip}, for $x \in \mathbb{R}^N$ the radial distance to $\partial\omega_i$ (i.e. $|X_i(\hat{x})-x|$) is equivalent to $d(x,\partial\omega)$ up to constants depending on $r$ and $R$.   In particular for $x \in \real^N\setminus\omega_i$ from Figure~\ref{fig: strong star shaped},
$$ B_{\tfrac{r}{R}|x-X_i(\hat{x})|} \subset \real^N\setminus\Omega $$
so that
$$ |x-X_i(\hat{x})|\geq d(x,\partial\omega_i) \geq \frac{r}{R}|x-X_i(\hat{x})|. $$
\end{proof}

As in the proof of Lemma~\ref{lem: JKOdef}, the restrictions on the admissible sets will allow us to take a continuum limit of the discrete time scheme.    In particular, the assumption of uniform strong star-shapedness gives equicontinuity of the contact lines $\partial\omega_k$ in space, and the Hausdorff distance movement restriction gives equicontinuity in time.  

\begin{lem}\label{lem: htozero} \textup{ (Continuum Limit)}
Let $\omega_{k,h}$ be given as in \eqref{eqn: JKOscheme}.  Then the following holds for the interpolants
\begin{equation}\label{eqn: interpolant}
\omega_h(t) = \omega_{k,h} \text{ for } t \in [kh,(k+1)h).
\end{equation}
\begin{enumerate}[(i)]
\item There exists $R=R(r_0,N)>0$ such that $\omega_h(t) \subset B_R(0)$ for all $h,t >0$ and therefore 
$$\text{Lip}(\partial\omega_h(t)) \leq C(r_0,N).$$
  \item  $\omega_h(t)$ (up to a subsequence) locally uniformly converges in $t \in [0,+\infty)$ to $\omega(t)$, which satisfies
$$
 d_H(\omega(t_1),\omega(t_2))\leq M|t_1-t_2| \hbox{ for all } t_1>t_2\geq 0.
$$
\item  For $C=C(r_0,R, N)$ given as in Lemma \ref{lem: SQTineq}, $\omega_h(t)$ satisfies the energy decay estimate
\begin{equation}\label{eqn: cts energy decay}
\frac{C}{t_2-t_1}\widetilde{\text{dist}}^2(\omega_h(t_1),\omega_h(t_2))\leq \mathcal{J}(\omega_h(t_1))-\mathcal{J}(\omega_h(t_2)).
\end{equation}

\end{enumerate}
\end{lem}
\begin{proof}
We give a sketch of the proof for (i).  Let $x \in \partial\omega_h(t)$ and recall that
$$|\omega_h(t)| \leq \mathcal{J}(\omega_h(t)) \leq \mathcal{J}(\omega_0), $$
and $\omega_h(t)$ has the interior cone at $x$,
$$
\{y: y\cdot x\geq 0 \text{ and } (x-y)\cdot x \geq |x-y|\sqrt{|x|^2-r_0^2} \}\cup B_r(0) \subset \omega_h(t).
$$
This cone contains $\sim |x|/r_0$ balls of radius $r_0/2$ centered along it's axis, so that
$$ |x| \lesssim R(r_0,N):= r_0^{N-1}\mathcal{J}(\omega_0). $$

Now we prove (ii) and (iii).  For simplicity we will take $h= 2^{-j}$ and then call the interpolant \eqref{eqn: interpolant} $\omega_{j}(t)$.  From (i) and Lemma \ref{lem: compactness} the sequence $\omega_j(t)$ for $t\in h\integer^+$ has a subsequence in $j$ which converges in Hausdorff distance.  Diagonalizing, we get a subsequence which we continue to call $j$ along which $\omega_j(t)$ converges in Hausdorff distance to a limit we call $\omega(t)$. From \eqref{eqn: admissible sets} and the triangle inequality we get:
$$d_H(\omega_j(t_1),\omega_j(t_2))\leq M|t_1-t_2| \hbox{ for } t_1, t_2\in h\integer^+ \hbox{ and for sufficiently large } j ,$$
and $\omega(t)$ inherits the same estimate. $\omega(t)$ is densely defined and uniformly continuous and thus has a unique $M$-Lipschitz extension to $[0,+\infty)$.  The local uniform convergence then directly follows.

 Estimate \eqref{eqn: cts energy decay} follows from the continuity of $\mathcal{J}(\cdot)$ and $\widetilde{\text{dist}}^2(\cdot,\cdot)$ with respect to convergence in Hausdorff distance: see Lemma \ref{lem: hausdorff estimates}.  
\end{proof}

\subsection{Barrier Properties for the Gradient Flow}
In \cite{GrunewaldKim11} a discrete gradient flow is constructed for \eqref{eqn: CLMV} without restrictions on the star-shapedness of the domain or the maximum speed of the free boundary.  There it is proven that the discrete solutions satisfy a barrier property with respect to smooth strict sub and super-solutions of \eqref{eqn: CLML}.  An analogous result can be proven for the solutions of the restricted iteration scheme defined in \eqref{eqn: JKOscheme}. The cost of the restriction to nicer sets in \eqref{eqn: JKOscheme} is that the class of barriers for which the sub and super-solution properties hold is reduced. The necessary conditions on admissible barriers can be deduced by inspecting the proofs in Section 3 of \cite{GrunewaldKim11}.  We give modified statement of the super-solution barrier property applicable to our situation below.
\begin{thm}\label{lem: barriers for restricted JKO} \textup{(Grunewald, Kim) (Super-solution barrier property) } Let $u_h = u[\omega_h]$ be the profile after one iteration of the scheme with time step size $h>0$,  and let 
$$ \lambda = \min\{\lambda[\omega_0],\lambda[\omega_h]\}.$$
  Let $\rho>0$, $x_0 \in \real^N$ and suppose there exists $\phi$ smooth with $|D\phi|\neq 0$ in $B_\rho(x_0) \times [0,h]$ with a $\delta>0$ such that
\begin{equation*}
\begin{array}{lll}
-\Delta \phi(\cdot,t) < \lambda-\delta & \textup{ in } & B_\rho(x_0) \times [0,h] \\ \\
\dfrac{\phi_t}{|D\phi|}-(|D\phi|^2-1)< -\delta & \textup{ on } & \Gamma(\phi)\cap(B_\rho(x_0) \times [0,h]).
\end{array}
\end{equation*}
Moreover, we require that there exists $C>0$ which does not depend on $h$ such that the sets
\begin{align}\label{eqn: barrierrestrict}
\omega_h\cup(\{\phi(\cdot,h)>\eta\}\cap B_\rho(x_0)), \ -h^2\leq\eta\leq Ch^{1/2}
\end{align}
are in the admissible class $\mathcal{A}^h_{\omega_0}$. Then, for $0<h<h_0 = h_0(\rho, \phi,\delta)$,  the following holds:
$$
\hbox{ If }\phi(\cdot,0)\leq u_0(\cdot), \hbox{ then } 
\phi(\cdot,h) \leq u_h(\cdot).
$$
\end{thm}

Essentially, \eqref{eqn: barrierrestrict} requires that the perturbed sets  of $\omega_h$ by the barrier $\phi$ which crosses $u_h$ from below stays in the admissible class.  An analogous barrier property with respect to supersolutions holds as well.  

\medskip

Now we show that the barrier property carries over to the continuum limit $\Omega(t)$.

\begin{lem}\label{lem: barrier property for grad flow}
\textup{(Restricted super-solution property for the gradient flow)} Let $\omega(t)$ be a continuum limit of the discrete gradient flow described in Lemma \ref{lem: htozero}.  Let $\phi \in C^{2,1}(\overline{\Omega(\phi)}\cap\{0\leq t\leq T\})$ such that $|D\phi|\neq 0$ on $\Gamma(\phi)\cap\{0\leq t\leq T\}$, the initial strict separation $\Omega_0(\phi) \subset\subset \omega(0)$ holds and $\phi$ is a strict subsolution,
\begin{equation}\label{eqn: strict subsoln}
\begin{array}{lll}
-\Delta \phi(\cdot,t) < \lambda[\omega(t)] & \textup{ in } & \Omega(\phi)\cap\{0\leq t\leq T\} \\ \\
\phi_t-\max\{|D\phi|^2-1,M\}|D\phi|< 0 & \textup{ on } & \Gamma(\phi)\cap\{0\leq t\leq T\}.
\end{array}
\end{equation}
Then $\phi$ cannot touch $u[\Omega(t)]$ from below at any free boundary point $(x_0,t_0) \in \Gamma(u)$ such that
\begin{equation} \label{eqn: normal condition}
\begin{array}{lll}
D\phi (x_0,t_0) \cdot x_0 < -r_0 |D\phi| (x_0,t_0). 
\end{array}
\end{equation}
\end{lem}
\begin{rem}
Test functions $\phi \in C^{2,1}(\overline{\Omega(\phi)}\cap\{0\leq t\leq T\})$ with $|D\phi|\neq 0$ on $\Gamma(\phi)$ can be extended to be $C^{2,1}(\real^N\times[0,T])$.
\end{rem}
\begin{proof}
Suppose not. Then there exists $\phi\in C^{2,1}(\real^N\times[0,T])$ which touches $u = u[\omega(t)]$ from below in $\mathbb{R}^N\times(0,T)$ without satisfying  \eqref{eqn: normal condition}.  From the continuity of the free boundaries of $\phi$ and $u$, this occurs for the first time at some $(x_0,t_0) \in \Gamma(u)\cap\Gamma(\phi)$, with $0<t_0\leq T$.  That is, there exists $\rho>0$ such that
\begin{equation*}
\begin{array}{lll}
\phi(x,t) \leq u(x,t) & \text{ for } & (x,t) \in Q:=B_{\rho}(x_0)\times[t_0-\rho,t_0].
\end{array}
\end{equation*}
We may assume that $\phi$ touches $u$ strictly from below by making the replacement,
 $$\phi \to (\phi -\epsilon|x-x_0|^2-\epsilon(t_0-t))_+. $$
  In particular, we will have,
   \begin{equation*}
\begin{array}{lll}
\phi(x,t) < u(x,t) & \text{ for } & (x,t) \in Q\cap \overline{\Omega(\phi)}\setminus \{(x_0,t_0)\},
\end{array}
\end{equation*}
 and for $\epsilon>0$ sufficiently small all the strict subsolution conditions for $\phi$ along with \eqref{eqn: normal condition} still hold.  Furthermore, we can make a subsolution which crosses $u$ from below in $Q$ by making a small translation in the normal direction to $\Gamma_{t_0}(\phi)$ at $x_0$,
 \begin{equation*}
\begin{array}{lll}
\varphi(x,t):= \phi(x+\epsilon \nu,t) & \text{ where } & \nu = \frac{D\phi}{|D\phi|}(x_0,t_0).
\end{array}
\end{equation*}
 Here $\epsilon>0$ is chosen sufficiently small based on the $C^2$ (in space) norm of $\phi$ so that \eqref{eqn: normal condition 2} holds for $\varphi$, $\varphi(x_0,t_0)>0 = u(x_0,t_0)$ and also
    \begin{equation*}
\begin{array}{lll}
\varphi(x,t) \prec u(x,t) & \text{ for } & (x,t) \in \partial_pQ.
\end{array}
\end{equation*}
Then we take $\rho$ small depending on the modulus of continuity of $D\phi$ such that there exists $\delta>0$ with: 
\begin{equation*}
\begin{array}{lll}
|D\varphi|>\delta,& \varphi_t <(M-\delta)|D\varphi|,& u-\varphi >\delta \text{ on } \partial_pQ\cap \Omega(\varphi),
\end{array}
\end{equation*}
 and \eqref{eqn: normal condition} hold on all of $Q$, that is
\begin{equation}\label{eqn: normal condition 2}
\begin{array}{lll}
D\varphi (x,t) \cdot x < -r |D\varphi| (x,t) & \text{ for } & (x,t) \in Q\cap \overline{\Omega(\varphi)}.
\end{array}
\end{equation}
  In order to make use of Lemma \ref{lem: barriers for restricted JKO} we need to take this information back to the approximating sequence of discrete solutions. Let $h_j \to 0$ be the sequence along which the discrete gradient flow scheme
$$ d_H(\omega_{h_j}(t), \omega(t))\to 0$$
uniformly for $t_0-\rho\leq t \leq t_0$.  From Lemma \ref{lem: hausdorff estimates} $u_j(\cdot,t) := u[\omega_{h_j}(t)]$ converge uniformly in $(x,t)$ to $u$.  Let $j$ be sufficiently large ($h_j$ small) so that Lemma \ref{lem: barriers for restricted JKO} will apply and also $h_j \leq \delta^2$, $\varphi(x_0,t_0)> u_j(x_0,t_0)$ and
\begin{equation}\label{eqn: strict separation for discrete}
\begin{array}{lll}
\varphi(x,t) \prec u_j(x,t) & \text{ for } & (x,t) \in \partial_pQ.
\end{array}
\end{equation}
Now there exists $t_0-\rho-h_j< t_k:=kh_j <t_0$ such that
\begin{equation*}
\begin{array}{lll}
\varphi(x,t_k) \leq u_j(x,t_k) & \text{ for all } & x \in B_\rho(x_0) \\ \\
\varphi(x_1,t_{k+1}) > 0=u_j(x_1,t_{k+1}) & \text{ for some } & x_1 \in B_\rho(x_0).
\end{array}
\end{equation*}
To apply Lemma \ref{lem: barriers for restricted JKO} and get a contradiction it is sufficient to show the following:
\begin{Claim}
Let $\eta\geq -h_j^2$ then, calling $\omega_{h_j,k} = \omega_k$:
$$
 U_\eta := (\{\phi(\cdot,t_{k+1})>\eta\}\cap B_{\rho}(x_0))\cup \omega_{k+1} \in \mathcal{A}^{h_j}({\omega_{k}}).
 $$
\end{Claim}
  Let $x \in U_\eta$, we claim that:
$$ d(x,\omega_k) \leq Mh_j. $$
 If $x \in \omega_{k+1}$ then $d(x,\omega_{k})\leq Mh_j$ from the admissibility condition \eqref{eqn: admissible sets} for $\omega_{k+1}$.  If $x \in \Omega_{t_{k+1}}(\varphi-\eta) \cap B_\rho(x_0)$ then in the case $\eta\geq0$ from the containment $\Omega_{t_k}(\varphi)\cap B_\rho(x_0) \subset \omega_{k} $ and $\phi_t <M|D\phi|$ we get,
 $$ d(x,\omega_{k}) \leq d(x,\Omega_{t_k}(\varphi)) < Mh_j. $$
 In the case $-\delta^2h_j\leq-h_j^2\leq\eta<0$:
\begin{align*}
 d(x,\omega_{k}) \leq d(x,\Omega_{t_k}(\varphi)\cap B_\rho(x_0)) &\leq d_H(\Omega_{t_{k+1}}(\varphi-\eta),\Omega_{t_{k+1}}(\varphi))+(M-\delta)h_j \\
 & \leq \frac{-\eta}{\delta}+(M-\delta)h_j \\
 & \leq  Mh_j.
 \end{align*}
Let $x \in  \omega_{k}$ then $d(x,U_\eta)\leq d(x,\omega_{k+1})\leq Mh_j$. Combining with the above arguments,
$$ d_H(U_\eta,\omega_{k+1}) \leq Mh_j.$$

\medskip

That $U_\eta\in \mathcal{S}_{r_0}$ is more subtle and this is where the condition \eqref{eqn: normal condition} comes in.  For a given $x \in \partial U_\eta$, we need to show that for every $y \in B_{r_0}(0)$, 
\begin{equation}\label{eqn: containment}
 L_{x,y}:=\{z \in \mathbb{R}^N: z = y + s(x-y), \ s \in [0,1]\}\subset \overline{U_\eta}.
 \end{equation}
This is trivially true for all $x \in \partial U_\eta \cap \partial \omega_{k+1}$ from the strong star-shapedness of $\omega_{k+1}$.  So we consider 
$$
x \in  \partial U_\eta \cap \partial (\Omega_{t_{k+1}}(\varphi-\eta)\cap B_\rho(x_0)) \subset B_{\rho}(x_0)\cap\partial \Omega_{t_{k+1}}(\varphi-\eta).
$$  
 Let $y \in B_{r_0}(0)$ and we consider the line segment from $x$ to $y$ parametrized by
$$ z(s) = x + s(y-x).$$
First we show that -- for $s>0$ small depending on the $C^2$ norm of $\phi$ -- $z(s) \in U_\eta$.  Using $|y|<r$ and \eqref{eqn: normal condition 2}, 
$$ \phi(z(s)) = \eta+sD\varphi(x) \cdot (y-x)+o(s) > s(-|y|+r_0)|D\varphi|(x)+o(s)>0. $$
Suppose towards a contradiction that $z(s)$ exits $\overline{U_\eta}$ before $s=1$, then by continuity it must pass through $\partial U_\eta$ for the first time at $0<s_0<1$ defined by,
$$ 1>s_0 = \inf \{s>0:  z(s) \in \partial U_\eta\}>0. $$
 If $z(s_0) \in \partial \omega_{k+1}$ then we are done due to $\omega_{k+1} \in \mathcal{S}_r$: 
$$L_{z(s_0),y} \subset \overline{\omega_{k+1}} \subset \overline{U_\eta}.$$
So it remains to consider the case $z(s_0) \in B_{\rho}(x_0)\cap\partial \Omega_{t_{k+1}}(\varphi)$ and in this case:
$$ 0 \geq \left.\frac{d}{ds}\varphi(z(s))\right|_{s=s_0^{-}} = D\varphi(z(s_0)) \cdot (y-x)  = (1-s_0)D\varphi(z(s_0)) \cdot (y-z(s_0)),$$
rearranging this and using $|y| < r_0$ we get
$$D\varphi(z(s_0))\cdot z(s_0) \geq D\varphi(z(s_0))\cdot y > -r_0|D\varphi(z(s_0))|$$
which contradicts \eqref{eqn: normal condition 2}.  
\end{proof}
An analogous statement about touching from above by strict classical super-solutions holds as well.  Lemma \ref{lem: barrier property for grad flow} almost says that $\omega(t)$ is a viscosity super-solution of \eqref{eqn: CLMV} except for the restriction \eqref{eqn: normal condition} on the barriers.  In fact we can make the following statement:
\begin{prop}\label{prop: condition to be vs}
If there exists $r'>r_0$ such that $\omega(t)\in\mathcal{S}_{r'}$ for all $t \in [0,T]$ then $u[\omega(t)]$ cannot be touched from below (above) by any strict classical sub-solution (super-solution) of (P-$\lambda[\omega(t)]$) and is therefore a viscosity solution of \eqref{eqn: CLMV}.
\end{prop}
  \begin{proof}
  This is simply because any $\phi\in C^{2,1}(\real^N\times[0,T])$ which touches $u[\omega(t)]$ from below (above) at a free boundary point must necessarily satisfy \eqref{eqn: normal condition}.  To see this fact, let $x \in \partial\omega(t)$, now from the strong star-shapedness of $\omega(t)$,
$$\{x+s(x-y): s\geq 0 \text{ and } y \in B_{r'}(0)\}\subset\real^N\setminus\omega(t) \subset\real^N\setminus\Omega_t(\phi)$$
and so for any $y \in B_{r'}(0)$
$$ 0\geq \left.\frac{d}{ds}\phi(x+s(x-y))\right|_{s=0} = D\phi(x) \cdot (x-y). $$
Rearranging above inequality, we obtain \eqref{eqn: normal condition}:
$$D\phi(x) \cdot x \leq \inf_{y \in B_{r'}(0)}D\phi(x)\cdot y = -r'|D\phi|(x)<-r_0|D\phi|(x).$$
\end{proof}
Unfortunately this statement is not that useful to us because we would still need to show that the gradient flow preserves strong star-shapedness.  That is if we had a Lemma for $\omega(t)$ like Lemma \ref{lem: GK short time ss} then from the above Proposition we would know that $\omega(t)$ was a viscosity solution of \eqref{eqn: CLMV} at least for a short time, and then we could apply a result about preservation of strong star-shapedness for viscosity solutions like Lemma \ref{lem: reflection for all time} to show that $\omega(t)$ is a viscosity solution globally in time.  Instead of attempting to prove Lemma \ref{lem: GK short time ss} for solutions of the gradient flow we will show directly that $\omega(t)$ and the viscosity solution of (P-$\lambda[\omega(t)]$) are the same.  The idea comes from the following corollary of the proof of Proposition \ref{prop: condition to be vs}.  As usual the corresponding result for supersolutions is also true.
\begin{cor}\label{cor: comparison for smooth}
(Comparison with star-shaped classical subsolutions) Suppose $\phi\in C^{2,1}(\real^N\times[0,T])$ is a strict classical sub-solution of (P-$\lambda[\omega(t)]$) and the initial strict separation
$$ \phi(\cdot,0) \prec u[\omega(0)] $$
holds.  Moreover for some $r'>r_0$ the positive phase of $\phi$, $\Omega_t(\phi)$, is in $\mathcal{S}_{r'}$ for all $t \in [0,T]$.   Then 
$$ \phi(\cdot,T) \leq u[\omega(T)]. $$
\end{cor}

\subsection{Comparison}\label{sec: JKO Comparison}
Let $\omega(t)$ be the continuum limit of the discrete gradient flow scheme defined above.   We want to show that $u = u[\omega(t)]$ is a viscosity solution of the problem,
\begin{equation}\label{eqn: fixedlambda}
\left\{\begin{array}{lll}
-\Delta v(x,t) = \lambda[\omega(t)] &\text{ in }& \omega(t) \\ \\
v_t = \min\{|Dv|^2-1,M\}|Dv|& \text{ on }& \partial \omega(t), 
\end{array}\right.
\end{equation}
at least as long as the unique viscosity solution of (P-$\lambda[\omega(t)]$) satisfies the appropriate strong star-shaped condition.

In fact $u$ satisfies a comparison principle with respect to sub and super-solutions of \eqref{eqn: fixedlambda} which are star-shaped with respect to a slightly larger ball than $B_{r_0}(0)$. This is the non-smooth version of the idea of Corollary \ref{cor: comparison for smooth}.  Various iterations of the proof can be found in the literature \cite{CaffarelliVazquez99, Kim03, BrandleVazquez05}, we will only sketch most of the details of the proof the essential idea is, as in Corollary \ref{cor: comparison for smooth}, to demonstrate that the restricted barrier property of the gradient flow solution $\omega(t)$ is sufficient to make the argument work.  
\begin{prop}\label{prop: comparison for grad flow} \textup{(Comparison with sub-solutions)}
Let $u(\cdot,t)=u[\omega(t)]$ as above from the gradient flow and $v$ a continuous sub-solution of 
\begin{equation*}
\left\{\begin{array}{lll}
-\Delta v(x,t) \leq \lambda[\omega(t)] &\text{ in }& \Omega_t(v) \\ \\
v_t \leq \min\{|Dv|^2-1,M\}|Dv|& \text{ on }& \Gamma_t(v),
\end{array}\right.
\end{equation*}
 with initial data ordered,
 $$\Omega_0(v) \subseteq \omega(0).$$ 
 Let $r > 0$ from the restriction on admissible sets in the discrete scheme \eqref{eqn: JKOscheme}. Suppose there exists $r'>r_0$ such that $\Omega_t(v)\in\mathcal{S}_{r'}$ for all $t \in [0,T]$. Then $$ v(\cdot,t)\leq u(\cdot,t) \ \text{ on }\ [0,T].$$
\end{prop}

\begin{proof} 
1. First we show it suffices to treat the case when $u$ and $v$ are initially strictly separated, 
\begin{equation}\label{eqn: strict assumption 1}
\Omega_0(v) \subset\subset \omega(0),
\end{equation}
and there exists $\epsilon>0$ so that $v$ satisfies,
\begin{equation}\label{eqn: strict assumption 2}
\begin{array}{lll}
 -\Delta v(x,t) \leq (1-\epsilon)\lambda(t) & \text{ for } & (x,t) \in \Omega(v).
\end{array}
\end{equation}
Let us assume temporarily that result holds in this case and we show that it also holds under the non-strict conditions given in the statement of the proposition.  This is accomplished by a similar device to the one used in the proof of the strong comparison type results Lemma \ref{lem: strong comparison} and Proposition \ref{prop: reflection comparison}.  Let $\epsilon>0$ small enough that $(1+\epsilon)^{-1}r' >r_0$, $c$ from Assumption \ref{hyp: F cond}, and $0<a<r_0$ then define the perturbation,
\begin{equation*}
\begin{array}{lll}
\widetilde{v}_\epsilon(x,t) = \sup_{|z|\leq a\epsilon-c\epsilon t} (1+\epsilon)^{-3}v((1+\epsilon)(x+z),t) & \text{ for } & 0 \leq t \leq \frac{a}{c}.
\end{array}
\end{equation*}
Now the perturbed subsolution $\widetilde{v}_\epsilon$ satisfies \eqref{eqn: strict assumption 1} and \eqref{eqn: strict assumption 2} and so we get,
\begin{equation*}
\begin{array}{lll}
 v(x,t) = \sup _{\epsilon>0} \widetilde{v}_\epsilon(x,t) \leq u(x,t) & \text{ for } & 0 \leq t \leq a/c
\end{array}
\end{equation*}
Iterating this $\floor{Tc/a}$ times we get the desired result.

\medskip

2. Now we work in the case when \eqref{eqn: strict assumption 1} and \eqref{eqn: strict assumption 2} hold.  Suppose towards a contradiction that $v$ crosses $u$ from below on $\real^N\times[0,T]$.  In order to have some regularity at the touching point we use the space-time convolutions described in Appendix \ref{sec: convolutions}.  Let $\epsilon, \delta >0$ with $ \delta < \epsilon/T$ define the sup-convolution of $v$
\begin{equation}
V(x,t) = \sup_{|(y,s)-(x,t)|\leq \epsilon} v(y,s)
\end{equation}
and the inf-convolution of $u$
\begin{equation}
U(x,t) = \inf_{|(y,s)-(x,t)|\leq{\epsilon-\delta t}}Z(y,s),   \quad Z(y,s):=\inf_{|(w,\tau)-(y,s)| \leq \e} u(w,\tau). 
\end{equation}
defined in the domain
\begin{equation*}
\Sigma_\epsilon :=\mathbb{R}^N\times [\epsilon,T-\epsilon].
\end{equation*}
  We will show that $V$ still has the following properties:
\begin{enumerate}[(a)]
\item $\Omega_t(V)$ strongly star-shaped with respect to a ball of radius larger than $r$,
\item $V$ is a subsolution of (P-$\lambda[\omega(t)]$) and is initially strictly separated from $U$,
$$ \Omega_\epsilon(V) \subset\subset \Omega_\epsilon(U).$$
\end{enumerate}

First we show (a), let $\epsilon \leq (r'-r)/2$, then 
\begin{equation} 
 \Omega_t(V) = \bigcup_{|(y,s)|\leq \epsilon} y+\Omega_{t+s}(v).
\end{equation}
Noting that all of the sets in the above union are star-shaped with respect to $B_{(r+r')/2}(0)$ implies that $\Omega_t(V)$ is as well.  This proves (a).

Now we show (b).  Choosing $\epsilon$ smaller if necessary based on $\eta$ and the modulus of continuity of $\lambda(t)$ the subsolution property for $V$ is from Lemma \ref{lem: spacetimeconvs}.  Then simply from the initial strict separation of $u$ and $v$ along with the Hausdorff distance continuity of the free boundaries of $u$ and $v$
 we also have the strict separation of $U$ and $V$ at time $\epsilon$
\begin{equation*}
\Omega_\epsilon(V) \subset \subset \Omega_\epsilon(U).
\end{equation*}
    
    Now by our assumption we know $U$ crosses $V$ from below for the first time at some point $P_0 = (x_0,t_0) \in\Sigma_\epsilon$.  By the strong maximum principle for subharmonic functions 
    $$x_0\in \Gamma_{t_0}(U)\cap \Gamma_{t_0}(V).$$
     At the touching point $P_0$ the positivity set of $V$ has an interior space-time ball of radius $\epsilon$ centered at some point $P_1 = (x_1,t_1) \in \Gamma(v)$, i.e. $|P_1-P_0| = \epsilon$ and
$$ B_1 :=\{(y,s) \in \mathbb{R}^N\times[0,T]: |(y,s)-P_1| <\epsilon\}\subset \Omega_+(V).$$
This ball has the tangent hyperplane through $P_0$
$$ H_1 = \{(y,s) \in \mathbb{R}^N\times[0,T]: ((y,s)-P_0)\cdot(P_1-P_0)=0\}.$$
  Let $(\nu,m)$ be normal to $H_1$ in the direction inward to $\Omega(V)$ with $|\nu| = 1$, $m$ is the advance speed of the free boundary $\Gamma(V)$ at $P_0$. We will show below that $m$ is finite, but at this point let us include the possibilities that $m=\infty$ or $m=-\infty$.

 Similarly, there exists $P_2=(x_2,t_2) \in \Gamma(Z)$ with $|P_2-P_1| = \epsilon - \delta t_0$ and $\Omega(U)$ has an exterior space-time ellipse at $P_0$ of the form
$$ \widetilde{B}_2 := \{(y,s) \in \mathbb{R}^N\times[0,T]: |(y,s)-P_2| \leq \epsilon-\delta s\}$$
while there exists $(x_3,t_2)\in\Gamma(u)$ with $|x_2-x_3|=\e$ and $\Omega(u)$ contains the following set:
\begin{equation}\label{set}
 E = \{(y,s) \in \mathbb{R}^N\times[0,T]: |y-w| \leq \e, |(w,s)-P_2| <\epsilon-\delta t_0\} \subset \Omega(u), \quad (x_3,t_2) \in\bar{E}\cap \Gamma(u) .
 \end{equation} 
Let $\widetilde{H}_2$ be the tangent hyperplane to $\widetilde{B}_2$ at $P_0$ and $H_2$ be the tangent hyperplane to $B_2$ at $P_2$. 

\begin{lem}
$H_1$ is not horizontal.
\end{lem}
\begin{proof}
 We refer to the proof of Lemma 2.5 in \cite{Kim03}, which rules out the possibility that $m=+\infty$.  
 
 \medskip
  
  Next let us rule out the possibility that $m=-\infty$.  Let $x_3$ and $E$ be the set given in \eqref{set}. Due to the star-shapedness and H\"{o}lder continuity of $\Omega(u)$, it follows that $(x_3,t_2)$ lies in the spatial boundary of $E\cap \{t=t_2\}$. We consider the classical subsolution 
 
   Take $\tau>0$ sufficiently small that $D_0:=E\cap \{t=t_2-\tau\}$ is nonempty. Let $D_1$ be the top portion of the closure of $E$, i.e., $D_1:= \bar{E} \cap \{t=t_2\}$, and let us define the interpolation of $D_0$ and $D_1$, i.e., 
   $$
   D_t := (1-s(t)) D_0 +  s(t)D_1, \hbox{ where } s(t)\hbox{ is linear and satisfies } s(t_2-\tau)=0, s(t_2)=1.
   $$  
   We choose $\tau$ small that $s'(t) <-1$. 
   Now let us consider the space-time domain  
  $$
  \Sigma:= \bigcup_{\tau \leq t-t_2 \leq 0} [D(t) - \frac{1}{2}D(t)]
 $$ 
  
 Then, due to the fact that $\Sigma\subset E$, the positive set $\Omega(u)$ crosses $\Sigma$ for the first time at $(x_3,t_2)$. Let us  consider the classical subsolution $\psi$ in the domain 
  satisfying
\begin{equation}
\left\{\begin{array}{lll}
-\Delta_x \phi(x,t) = \lambda[w(t)] & \text{for} & (x,t) \in \Sigma \\ \\
\phi =0 & \text{for} & (x,t) \in \partial D(t) \\ \\
 \phi = \min_{\frac{1}{2}D(t)} u(x,t)>0 & \text{for} & (x,t)\in \frac{1}{2}\partial D(t) \\ \\
\end{array}\right.
\end{equation}
Then $\psi$ crosses $u$ from below at $(x_3,t_2)$. One can check from the definition of $u$ and the strong star-shapedness of $\Omega(V)$ that the outer normal $\nu$  of $D_1$ at $x_3$ satisfies $\nu\cdot x_2 < -r_0$ and thus it contradicts Lemma~\ref{lem: barrier property for grad flow}.
 \end{proof}

 Let $(\nu',m')$ be the normal to the hyperplane $H_2$ rescaled as before.  Then as in Lemma \ref{lem: spacetimeconvs} the normal to $\widetilde{H}_2$ is $(\nu',m'+\delta)$.  As a consequence of the ordering $V(x,t) \leq U(x,t)$ for all $ t\leq t_0$ we also get the ordering of the advance speeds,
$$ m'+\delta \leq m.$$
Since $U\leq V$ for $t \leq t_0$ we get the following inclusions
\begin{equation*}
B_1 \cap \{t \leq t_0 \} \subseteq \{U>0\}\cap\{V>0\} \quad \text{ and } \quad \widetilde{B}_2 \cap \{t \leq t_0 \} \subseteq \{U=0\}\cap\{V=0\}.
\end{equation*}
In particular $\Omega_{t_0}(U)$ and $\Omega_{t_0}(V)$ both have interior and exterior spatial balls at $x_0$,
$$ B_1 \cap \{t = t_0 \}  \ \ \text{and} \ \ \widetilde{B}_2 \cap \{t = t_0 \} $$
which must both have centers lying along the same axis, or equivalently both free boundaries have spatial inward normal $\nu$ at $x_0$.  From the strong star-shapedness of $\Omega_{t_0}(V)$, and this is a key point as we have seen in the proof of above lemma: we get that
\begin{equation}\label{eqn: key thing}
\nu \cdot x_0 \leq - r''<-r_0.
\end{equation}
Now the idea is to construct a smooth strict sub-solution $\phi$ which touches $\Gamma(V)$ from below at $(x_0,t_0)$.  Then a translation of $\phi$ will touch $u$ from below at $P_2$ from the inadmissible direction,
$$\nu = \frac{\phi_t}{|D\phi|}(x_0,t_0)$$
leading to a contradiction of Lemma~\ref{lem: barrier property for grad flow}.

 \begin{lem}\label{lem: advance speed bounds}
$$ -1< m \leq M. $$
\end{lem}
\begin{proof}
 The fact that $m\leq M$ follows from a relatively simple barrier argument, based on the condition $v_t\leq M|Dv|$ on $\Gamma_t(v)$ as well as  the fact that $\Omega(v)$ has an exterior space-time ball at $(x_1,t_1)$. 
 
 \medskip
 
   It remains to show that $ m\leq-1$. If it is, then 
  \begin{equation}\label{fact}
  m' \leq -1-\delta. 
  \end{equation}
  Take $\tau>0$ sufficiently small that $D_0:=B_2\cap \{t=t_2-\tau\}$ is nonempty.
    We can now  construct a classical subsolution $\psi$ in the domain 
  $$\tilde{\Sigma}:=\cup_{t_2-\tau \leq t \leq s_0} (B_2-\frac{1}{2}B_2), $$ 
 similarly as in the proof of $m >-\infty$ in the lemma above, and use \eqref{fact} to derive a contradiction.
\end{proof}

\begin{lem}\label{lem: nontangential estimate}
Near the point $P_0$ we have the nontangential estimate
\begin{equation}
\liminf_{d\to 0^+}\frac{V(x_0-d\nu,t_0)}{d} \geq \sqrt{m+1}.
\end{equation}
\end{lem}
\begin{proof} 
 
The proof is based on construction of the barrier for $v$ to yield a contradiction, in the event that the lemma holds false, and it is parallel to the proof of Lemma 2.6 in \cite{Kim03}.
 
 
\end{proof}
Now as in \cite{Kim03} for any $\eta>0$ we construct a smooth test function $\phi(x,t)=\phi(|x-x_0|,t)$ with the following properties:
\begin{equation}
\left\{\begin{array}{lll}
-\Delta_x \phi(x,t) < 0 & \text{for} & (x,t) \not\in \tfrac{1}{4}B_2 \\ \\
\phi > 0 & \text{for} & (x,t) \in B_2 \\ \\
 \phi <0 & \text{for} & (x,t)\not\in \overline{B_2} \\ \\
|D\phi|(x,t) = \sqrt{m+1}(1-\eta) & \text{ on } & \partial B_2 \cap \{ t = t_2 \}.
\end{array}\right.
\end{equation}
Now from the definition of $m'$ we have for $x \in \partial B_2 \cap \{t=t_2\}$,
$$ \tfrac{\phi_t}{|D\phi|}(x,t_2) = m' \leq m-\delta$$
and choosing $\tau$ sufficiently small depending on the $C^2$ norm of $\phi$ and $\eta$ small depending on $\delta$ and $m$ we get for $x \in \partial B_2 \cap \{t_2-\tau\leq t\leq t_2\}$,
$$ \tfrac{\phi_t}{|D\phi|}(x,t_2)<m-\delta/2 < |D\phi|^2-1.$$
Therefore $\phi$ is a strict subsolution in the region
$$ (B_2 \setminus \tfrac{1}{4}B_2 )\cap \{t_2-\tau \leq t \leq t_2\}.$$
Next we show that $\phi$ touches $u$ from below at $P_2 = (x_2,t_2)$.  For $d>0$ sufficiently small depending on the $C^2$ norm of $\phi$ we have
\begin{equation}\label{eqn: growth estimate I}
\begin{array}{lll}
\phi(x,t) \leq \sqrt{m+1}(1-2\eta)d & \text{ on } & \partial (1-d)B_2 \cap \{t_2-\tau \leq t \leq t_2\} .
\end{array}
\end{equation}
Let $(x,s) \in B_2$ and let 
$$d:= d(x,\partial B_2\cap\{t=s\})$$
 then since $P_0$ is the center of $B_2$ we have that $|P_0-(x,s)| = \epsilon-\delta t_0-d$ and
$$ |(x_0+d\nu,t_0)-(x,s)|\leq |P_0-(x,s)|+d = \epsilon-\delta t_0.$$
 Therefore from the definition of $U$ as an infimum and from Lemma~\ref{lem: nontangential estimate} for any $\eta>0$ there exists $d_0$ such that $d<d_0$ implies
\begin{equation}\label{eqn: growth estimate II}
 \sqrt{m+1}(1-2\eta)d\leq U(x_0+d\nu,t_0) \leq u(x,s).
 \end{equation}
 Now combining \eqref{eqn: growth estimate I} and \eqref{eqn: growth estimate II} with the fact that $B_2 \subset\Omega(u)$ and $\phi=0$ on $\partial B_2$ we get that
 \begin{equation}
 \begin{array}{lll}
  \phi(x,t) \leq u(x,t) & \text{ for }  & (x,t) \in (\partial (1-d)B_2 \cup \partial B_2) \cap \{t_2-\tau \leq t \leq t_2\}.
  \end{array}
  \end{equation}
  Then since $u-\phi$ is superharmonic by the strong minimum principle we get,
   \begin{equation}
 \begin{array}{lll}
  \phi(x,t) < u(x,t) & \text{ for }  & (x,t) \in B_2\setminus(1-d)B_2 \cap \{t_2-\tau \leq t \leq t_2\},
  \end{array}
  \end{equation}
 and $u(P_2) = \phi(P_2)=0$ so $\phi$ touches $u$ from below at $P_2$.  This is a contradiction of Lemma~\ref{lem: barrier property for grad flow} since from \eqref{eqn: key thing}
 $$ \frac{D\phi}{|D\phi|}(x_2,t_2)\cdot x_2 = \nu \cdot x_2 <\nu \cdot x_0<-r_0. $$
\end{proof}

\section{Convergence for Solutions with Global in Time Star-Shapedness}

Now we can combine the results of the previous sections to get our main result.  Under the assumption of sufficient roundness of the initial data $\Omega_0$ phrased in terms of $\rho$-reflection any continuum limit of the restricted gradient flow of the functional $\mathcal{J}$ is also a global in time viscosity solution of the problem \eqref{eqn: CLMV}.   To make this precise let us define the class of weak solutions arising from the gradient flow scheme described in section \ref{sec: Gradient Flow}.
\begin{DEF}
An evolution $\omega(t): [0,\infty) \to \mathcal{S}_r$ is an {\it energy solution} of \eqref{eqn: CLMV} if there exist $\omega_{M_k}(t)$, $M_k \to \infty$ as $k \to \infty$, which are minimizing movements of the restricted gradient flow scheme from Definition \ref{def: JKO scheme} with initial data $\omega(0)$ such that,
$$ d_H(\omega_{M_k}(t),  \omega(t) ) \to 0 \ \hbox{ as } \ k \to \infty \ \hbox{ locally uniformly in } [0,\infty).$$
A droplet profile $u : \real^N \times [0,\infty) \to [0,\infty)$ is an energy solution if $u = u[\omega(t)]$ for an energy solution $\omega(t)$.
\end{DEF}
Now we can say the following about any energy solution arising from an initial data with $\rho$-reflection.

\begin{thm}\label{thm: convergence thm}
Let $V>0$ and $\Omega_0$ a domain in $\real^N$ such that $\Omega_0$ has $\rho$-reflection with $\rho$ satisfying,
$$ \rho < C_N V^{\frac{1}{N+1}}, $$
where $C_N$ is a dimensional constant. Then there exists an energy solution $u: \real^N\times[0,+\infty)\to[0,+\infty)$. Moreover any energy solution $u$, the following holds:

\begin{itemize}
\item[(a)] $u$  is also a viscosity solution of the free boundary problem,
\begin{equation}\tag{P-V}
\left \{ \begin{array}{lll}
         -\Delta u(x,t)  = \lambda(t) & \text{ in }& \Omega_t(u)\\ \\
         u_t = (|Du|^2-1)|Du| & \text{ on } & \Gamma_t(u) \\ \\
         \{u(\cdot,0)>0\} = \Omega_0,
         \end{array}\right.
\end{equation}
where $\lambda(t)$ is chosen to enforce the volume constraint for all $t>0$,
$$ \int u(x,t)dx = V. $$
\item[(b)]The positivity set $\Omega_t(u)$ has $\rho$-reflection for all $t>0$ with $\inf_{x \in \Gamma_t(u)}(|x|-\rho)$ bounded below.  The energy $\mathcal{J}$ is decreasing along the flow and additionally $u$ satisfies the energy decay estimate for all $t>s\geq0$,
$$ \frac{1}{t-s}\widetilde{\textup{dist}}(\Omega_s(u),\Omega_t(u)) \lesssim_{\rho,N} \mathcal{J}(\Omega_s(u))-\mathcal{J}(\Omega_t(u)).$$
\item[(c)]The flow of the sets $\Omega_t(u)$ converges uniformly modulo translation to the radially symmetric equilibrium solution,
$$ \frac{\lambda_* }{2N}(r_*^2-|x|^2)_+,$$
where $r_*$ is given in \eqref{eqn: equilibrium radius} and $\lambda_*$ can be calculated from the volume constraint.  More precisely we mean that,
$$ \inf \{ d_H(\Omega_t(u),B_{r*}(x)) : x \in \overline{B_{\rho}(0)}\} \to 0 \hbox{ as } t \to \infty. $$
\end{itemize}

\end{thm}

\begin{proof}
1. The existence proof as well as the energy estimate follow from a fairly straightforward combination of the results we have already proven.  We give an outline of the proof. Given the assumption on $\Omega_0$ having $\rho$-reflection we know due to the apriori estimate Lemma \ref{lem: reflection for all time} that any solution of \eqref{eqn: CLMV} will be in $\mathcal{S}_r$ for some $r>0$ as long as it exists. Then we construct the restricted gradient flow solution with initial data $\Omega_0$ as in Section \ref{sec: Gradient Flow} where the restriction on the radius of the strong star-shapedness is strictly weaker than that from the apriori estimate.  Given the Lagrange multplier $\lambda(t)$ associated with the restricted gradient flow we then solve \eqref{eqn: CLML}.  Then the idea is that the viscosity solution of \eqref{eqn: CLML} is strongly star-shaped with a larger radius than the restriction on the gradient flow and so we will be able to use Proposition \ref{prop: comparison for grad flow} to show that the two solutions agree for all time.

\medskip

Let $b>0$ so that 
\begin{equation*}
\begin{array}{lll}
B_{(1+b)\rho}(0)\subset \Omega& \hbox{ and } &(1+b)\rho <C_NV^{\frac{1}{N+1}}.
\end{array}
\end{equation*}
  Then we expect thanks to the apriori estimate Lemma \ref{lem: reflection for all time} that $(1+b)\rho$ will be contained in $\Omega_t$ for all time. In particular we expect $\Omega_t$ will be in $\mathcal{S}_{r}$ for $r = b\rho$.  Let $R=R(r)>0$ be so large that any set in $\mathcal{S}_{r}$ which touches $B_R$ from the inside must have larger energy than $\Omega_0$ and then define $\gamma = \tfrac{r}{2R}$.  That it is possible to choose $R$ in this way is described in Lemma \ref{lem: htozero}.  Let $M>0$ and let $\omega_M(t)=\omega_{M, \gamma r}(t)$ be the restricted gradient flow solution with initial data $\Omega_0$ as constructed in the last section.  The notation indicates that $\omega_{M,\gamma r}(t)$ is restricted to remain in $\mathcal{S}_{r}$ and its free boundary can move no faster than $M$.  We suppress the dependence on $\gamma r$.  Define the putative Lagrange multiplier $\lambda_M$ to be
$$ \lambda_M(t) := \lambda[\omega_{M}(t)],$$
and let $u_M$ be the possibly discontinuous viscosity solution of $(P-\lambda_M)$ constructed by Perron's method in Theorem \ref{thm: short time existence}.  We will show that $u_M$ and $u[\omega_{M}]$ are the same.

\medskip

Let $I$ be the largest interval containing the origin on which $u_M$ and $u[\omega_M]$ agree. We will show that $I$ is open and closed in $[0,+\infty)$. Since $u_M$ agrees with $u[\omega_M]$ on $I$ it is continuous and has constant volume and thus it is a volume preserving viscosity solution.  Therefore Lemma \ref{lem: reflection for all time} implies that:
$$\Omega_t(u_M) \hbox{ has $\rho$-reflection and }  B_{ (1+b)\rho} \subset \Omega_t(u_M) \hbox{ for $t \in I$}.$$  
In particular $\Omega_t(u_M)$ is in $\mathcal{S}_{r}$ on $I$. Suppose $I= [0,T)$ for some $T>0$, then the short time existence theorem implies that, for some small $t_0(r,R)$, $u_M$ is continuous on $[0,T+t_0)$.  Since the set where two continuous functions agree is closed $I=[0,T]$.  Suppose $I=[0,T]$ where $T$ may be equal to zero.  Then thanks to Lemma \ref{lem: GK short time ss} there exists $t_M>0$ such that: 
$$\Omega_t(u_M)\in\mathcal{S}_{r'} \hbox{ for some }r'>\gamma r \hbox{ on } [0,T+t_M).$$  Then from Proposition \ref{prop: comparison for grad flow} $u[\omega_M]$ satisfies a strong comparison principle with respect to viscosity solutions which are strongly star-shaped with a larger radius than $\gamma r$ so $u_M = u[\omega_M]$ on $[0,T+t_M)$.  Therefore $I= [0,+\infty)$ and there exists a global in time continuous viscosity solution of \eqref{eqn: CLMV}$_M$ which has $\rho$-reflection and is in $\mathcal{S}_{r,R}$ for all $t>0$.

\medskip

2. Now we show the existence for \eqref{eqn: CLMV} without the restriction on a maximum speed.  The key point in this case is the equicontinuity afforded by the fact that the $\Omega_t(u_M) \in \mathcal{S}_{r,R}$ with $r,R$ independent of $M$.  In particular from Corollary \ref{cor: equicontinuity} we get for some $C>0$ and $ \alpha\in(0,1)$ independent of $M$,
\begin{equation}\label{eqn: holder for uM}
 d_H(\Gamma_t(u_M),\Gamma_s(u_M)) \leq C |t-s|^\alpha. 
 \end{equation}
Then we also derive thanks to Lemma \ref{lem: hausdorff estimates} the equicontinuity of the Lagrange multipliers $\lambda_M(t)$.  Taking a subsequence such that the $\lambda_M(t)$ converge uniformly on compact subsets of $[0,+\infty)$ to some $\lambda(t)$ we can apply Lemma~\ref{lem: sending M to infty}.  We get that along this subsequence the $u_M$ converge uniformly to a viscosity solution $u$ of \eqref{eqn: CLMV} with free boundary velocity $F(|Du|) = |Du|^2-1$.  Due to the uniform convergence of $\Omega_t(u_M)$ to $\Omega_t(u_M)$ in hausdorff distance sense the energy estimates for the $u_M$ and the H\"{o}lder regularity in time \eqref{eqn: holder for uM} carry over to $u$.  Then $u$ is an energy solution by the definition. Unfortunately by the compactness method we do not know whether there is any uniqueness of the limiting Lagrange multiplier $\lambda(t)$.

\medskip

3. Now we show that any subsequential limit of the $u(\cdot,t_n)$ must be a viscosity solution of the equilibrium problem \eqref{eqn: EQ}.  Note that the same result is true for all the $u_M$.

\begin{Claim}
Let $t_n\to \infty$ such that $\Omega_{t_n}(u)$ converges in Hausdorff topology to $\Omega_\infty$.  Then $u[\Omega_\infty]$ is a stationary solution of \eqref{eqn: CLMV}, that is it solves the equilibrium problem \eqref{eqn: EQ} in the viscosity sense.
\end{Claim}
Define $U_n: [0,+\infty) \to \mathcal{S}_{r,R}$ by
\begin{equation}
U_n(t) := \Omega_{t+t_n},
\end{equation}
 and we also consider the viscosity solutions of \eqref{eqn: CLMV} which lie above the $U_n(t)$,
\begin{equation}
\begin{array}{lll}
 v_n(x,t) := u[U_n(t)](x) & \text{ with Lagrange Multipliers } & \lambda_n(t) := \lambda[U_n(t)].
 \end{array}
 \end{equation}
 
 \medskip

 First we show that, uniformly in $t>0$,
 $$\mathcal{J}(U_n(t)) \to \mathcal{J}(\Omega_\infty).$$ 
  Since is $\mathcal{J}(U_n(t))$ is monotone decreasing for all $n>0$ we have for all $t>0$
 $$ \mathcal{J}(\Omega_\infty) = \inf_{s>0}  \mathcal{J}(\Omega_s) \leq \mathcal{J}(U_n(t)) \leq \mathcal{J}(U_n(0))= \mathcal{J}(\Omega_{t_n}),$$
 but due to Lemma \ref{lem: hausdorff estimates} along with the convergence $\Omega_n\to \Omega_\infty$ in Hausdorff distance,
 $$\mathcal{J}(\Omega_{t_n}) \searrow \mathcal{J}(\Omega_\infty).$$
 
 \medskip
 
  Now we will show that up to a subseqeunce the $v_n$ converge uniformly on compact time intervals.  Recalling the uniform H\"{o}lder estimates from \eqref{eqn: holder for uM} for $\Omega_t(u_M)$ which carry over in the limit to $\Omega_t(u)$,
\begin{equation}
\begin{array}{lll}
 d_H(U_n(t),U_n(s)) \leq C|t-s|^\alpha, 
 \end{array} 
 \end{equation}
 and from Lemma \ref{lem: htozero} the energy estimates,
 \begin{equation}\label{eqn: energy est for Un}
 \frac{C}{t-s}\widetilde{\text{dist}}(U_n(s),U_n(t))\leq \mathcal{J}(U_n(s))-\mathcal{J}(U_n(t)).
 \end{equation}
The paths $U_n$ are a sequence of equicontinuous maps into $( \mathcal{S}_{r,R},d_h)$ and so from the Compactness Lemma \ref{lem: compactness} there exists $U_\infty: [0,+\infty) \to \mathcal{S}_{r,R}$ with $U_\infty(0) = \Omega_\infty$ such that up to taking a subsequence, 
$$ U_n \to U_\infty \text{ uniformly on compact subsets of } [0,+\infty).$$

\medskip

  Now we show that $v_\infty$ is a stationary viscosity solution of \eqref{eqn: CLMV}.  From Lemma \ref{lem: hausdorff estimates} we get the following:
\begin{enumerate}[(i)]
\item $v_n \to v_\infty$ uniformly in $(x,t)$ on compact time intervals and therefore from the stability of viscosity solutions under uniform convergence -- Lemma \ref{lem: uniform stability} -- $v$ is a viscosity solution of \eqref{eqn: CLMV}.
\item $\widetilde{\text{dist}}(U_n(s),U_n(t)) \to \widetilde{\text{dist}}(U_\infty(s),U_\infty(t)) $ uniformly on compact subsets of $[0,+\infty)\times[0,+\infty)$.
\end{enumerate}
Combining (ii) with \eqref{eqn: energy est for Un} we derive the energy estimate for $U_\infty(\cdot)$ for all $t>s>0$:
 \begin{equation}
 \frac{C}{t-s}\widetilde{\text{dist}}(U_\infty(s),U_\infty(t))\leq \mathcal{J}(U_\infty(s))-\mathcal{J}(U_\infty(t)) = 0.
 \end{equation}
So $v_\infty = u[\Omega_\infty]$ is a stationary viscosity solution of \eqref{eqn: CLMV}.  Then due to Theorem~\ref{thm: serrin}, Theorem~\ref{thm: de silva} and Corollary~\ref{smoothness}  it follows that $\Omega_{\infty} = B_{r*}(x_0)$ for some point $x_0 \in \real^N$.  Actually $x_0$ is not completely arbitrary since we know that $\Omega_\infty$ must have $\rho$-reflection.  One can easily check that this implies $x_0 \in \overline{B_\rho(0)}$.

\medskip

4. Finally we show that the convergence is uniform modulo translation.  Suppose that there exists a sequence of times $t_n \to \infty$ and a $\delta>0$ such that
$$ \inf_{x \ \in \overline{B_\rho(0)}} d_H(\Omega_{t_n}(u),B_{r*}(x)) > \delta. $$
By taking a subsequence of the $t_n$ we may assume that $\Omega_{t_n}(u)$ converges in Hausdorff distance to some $\Omega_\infty$.  By part 3 of the argument $u[\Omega_\infty]$ must be equal to $B_{r*}(x_0)$ for some $x_0 \in B_\rho(0)$.  Choosing $n$ sufficiently large so that
$$ d(\Omega_{t_n}(u),B_{r*}(x_0)) < \delta $$
we derive a contradiction.

\end{proof}

Note that above theorem leaves the possibility that the drop oscillates between a family of round drops with its speed going to zero but not fast enough to converge to a single profile. Below we show that, if the drops are sufficiently regular at large times, then this does not happen. Such regularity, when the time is sufficiently large so that the profile of $u$ is sufficiently close to a round one is suspected to be true in the light of existing results introduced by Caffarelli et. al. (see e.g. the book \cite{CS}), but verifying this for our setting would be highly nontrivial and thus we do not pursue this question here.

\begin{prop}\label{conditional}
\textup{(Conditional uniqueness of the limit)} Suppose additionally that $u(\cdot,t)$ are uniformly $C^{1,\alpha}$, then $\Omega_t \to B_{r*}(x_0)$ for some $x_0$.
\end{prop}
\begin{proof}
Let $t_n \to \infty$ be a sequence of times along which $\Gamma_{t_n}(u)$ converges in Hausdorff distance.  The limit is a ball $\partial B_{r*}(x_0)$ for some $x_0$ which is compatible with the reflection symmetry of $\Omega$. By translating we may assume that $x_0 = 0$.  Note under this translation $\Omega_t(u)$ no longer need have $\rho$-reflection or be strongly star shaped with respect to the origin, this will not affect the proof.  From Lemma \ref{lem: hausdorff estimates} we get that $u(\cdot,t_n)$ converge uniformly to 
$$ u_{EQ} = \frac{\lambda_*}{2N}(r_*^2-|x|^2)_+. $$
Due to the assumption, we have 
$$ \sup_n||u(\cdot,t_n)||_{C^{1,\alpha}}<C_1<+\infty. $$
In particular the $Du(\cdot,t_n)$ are uniformly bounded and equicontinuous so they must converge uniformly to $Du_{EQ}$.  Let $\delta>0$ small enough that $C_1 \delta ^\alpha \leq 1/2$ and choose $n$ sufficiently large that 
\begin{equation*}
\begin{array}{lll}
d_H(\Gamma_{t_n}(u),\partial B_{r*}) \leq \delta &\hbox{ and }&  ||Du(\cdot,t_n)-Du_{EQ}||_\infty \leq \delta.
\end{array}
\end{equation*}
Now let $x \in \Gamma_{t_n}(u)$ and let $y=y(x) \in \partial B_{r_*}$ such that $|x-y| = d(x,\partial B_{r_*})$.  We calculate,
$$ |Du(x,t_n)-Du_{EQ}(y)| \leq ||Du(\cdot,t_n)-Du_{EQ}||_\infty+C_1d(x,\partial B_{r_*})^\alpha \leq C\delta^\alpha $$
so since $|Du_{EQ}|=1$ on $\partial B_{r*}$ we have that $|Du(x,t_n)|\geq 1/2$ and
  $$ |\tfrac{Du}{|Du|}(x,t_n)-\tfrac{y}{|y|}| \leq C\delta^\alpha. $$
Rephrasing the above in terms of the interior normal field $\nu_n$ to $\Gamma_{t_n}(u)$ we get that
\begin{equation}
  \langle \nu_n(x),x\rangle \geq 1-o(1).
  \end{equation}
  In particular this means, thanks to Lemma \ref{lem: lip norm closeness}, that for any $\rho>0$ there exists $n$ sufficiently large so that $\Omega_{t_n}(u)$ has $\rho$-reflection.  Since $\rho$-reflection is preserved under the flow for $\rho$ small this means that for every $\rho>0$ there is a $T$ such that for $t\geq T$, $\Omega_t(u)$ has $\rho$-reflection.  Thus any subsequential limit of the $\Omega_t(u)$ has $\rho$-reflection for every $\rho>0$ and must be a ball centered at the origin.

\end{proof}
There is one nontrivial case where we can say that the limiting equilibrium solution is unique.  When the initial data is symmetric with respect to $N$ orthogonal hyperplanes through the origin in addition to all the conditions given in the statement of Theorem \ref{thm: convergence thm} then the limit is unique.  In this case the reflection symmetries are preserved by the equation and so any Hausdorff distance limit of the $\Omega_t(u)$ will share these symmetries.  Then it is basic to check that the only ball of radius $r^*$ which is symmetric with respect to $N$ orthogonal axes through the origin is in fact centered at the origin.  We record this fact in the following corollary:

\begin{cor}
If $\Omega_0$ satisfies all the conditions of Theorem \ref{thm: convergence thm} and additionally is symmetric with respect to $N$ mutually orthogonal hyperplanes through the origin then any solution $u$ of \eqref{eqn: CLMV} with initial data $\Omega_0$ constructed from the discrete gradient flow scheme of Section \ref{sec: Gradient Flow} satisfies,
$$ d_H(\Omega_t(u),B_{r*}(0)) \to 0 \ \hbox{ as } \ t \to \infty. $$

\end{cor}

\appendix

\section{Geometric Properties}\label{sec: geometry appendix}

  Let $\Omega$ be an open bounded domain in $\mathbb{R}^N$ which is strictly star-shaped with respect to $0$.  Let $\rho>0$ such that $B_\rho(0) \subseteq \Omega$ and let $H$ be a hyperplane in $\mathbb{R}^N$ such that $H\cap B_\rho(0) = \emptyset$.  Let $H_+$ be the open half-space of $H$ which contains $B_\rho(0)$ and $H_{-}$ be the interior of its complement.  Then define 
$$ \Omega_+ = \Omega \cap H_+ \text{ and } \Omega_- = \Omega \cap H_{-}$$
and define the reflection through $H$ by
$$ \phi_H(x) = x-2\langle x-y, \nu_H\rangle\nu_H$$
where $\nu_H$ is the unit normal to $H$ (pointing inward towards $H_-$ for concreteness) and $y \in H$.  
\begin{replem}{lem: in an annulus}
Suppose $\Omega$ has $\rho$-reflection, then 
$$ \sup_{x \in \partial\Omega} |x|-\inf_{x\in \partial\Omega}|x| \leq 4\rho. $$
\end{replem}
\begin{proof}
 Take $x_0 \in \partial\Omega$ such that
 $$ |x_0| = \inf_{x\in \partial\Omega}|x| . $$
 Let $H$ be any hyperplane tangent to $B_\rho(0)$ such that $x_0 \in H_+$.
\\
1. Claim:  The reflections $\phi_H(x_0)$ with respect to the hyperplanes described above cover all directions in the sphere, more precisely
 $$ \{\frac{\phi_H(x_0)}{|\phi_H(x_0)|}: H \text{ tangent to $B_\rho(0)$ and }x_0 \in \overline{H}_+\}= S^{N-1}. $$
 We index $H$ by its normal vector, so for $\nu \in S^{N-1}$ let $H_\nu$ be the hyperplane through $0$ normal to $\nu$ and then define:
 $$\phi_\nu(x) = \phi_{H_\nu+\rho\nu}(x).$$
 Let $\omega \in S^{N-1}$.  Without loss, by changing coordinate names, we may assume that $\widehat{x}_0= e_1$ and $\omega\in \text{span}\{e_1,e_2\}$.  We restrict ourselves to hyperplanes with normal direction $\nu \in \text{span}\{e_1,e_2\}$ and thereby reduce to the case $N=2$.   Let $\varphi$ such that $\cos\varphi = \rho/|x_0|$.  Then for any $\theta\geq\varphi$ the plane through 
 $$ \rho((\cos\theta) e_1+(\sin\theta) e_2) $$
 has $x_0$ in the same half-space as $B_\rho$.  Let us consider the continuous map $f : [0,\pi]\to[0,\pi]$ defined by,
 $$ f(\theta) = \cos^{-1}\left\langle\tfrac{\phi_{\nu(\theta)}(x_0)}{|\phi_{\nu(\theta)}(x_0)|}, e_1\right\rangle \hbox{ where } \nu(\theta) = \cos\left(\varphi+\left(\tfrac{\pi-\varphi}{\pi}\right)\theta\right) e_1+\sin\left(\varphi+\left(\tfrac{\pi-\varphi}{\pi}\right)\theta\right) e_2.$$
 We show that $f(0)=0$ and $f(\pi)=\pi$ and thus $f$ maps $[0,\pi]$ onto itself.  We chose $\varphi$ so that $x_0 \in H_{\nu(0)}+\rho\nu(0)$,
 $$ \langle x_0-\rho\nu(0),  \nu(0)\rangle = \langle x_0, \nu(0)\rangle-\rho = |x_0| \cos{\varphi} -\rho =0.$$
 Therefore $\phi_{\nu(0)}(x_0) = x_0$ and $f(0)=0$.  Meanwhile since $\nu(\pi)=\pi$ we compute directly that
  $$\phi_{\nu(\pi)}(x_0) = (\rho-|x_0|)e_1$$
   and thus $f(\pi)=\pi$.  Now in order to show that $\omega$ is hit we just choose $\theta$ so that: 
 $$f(\theta) = \cos^{-1}\langle \omega , e_1 \rangle$$
 or in other words,
 $$ \frac{\phi_{\nu(\theta)}(x_0)}{|\phi_{\nu(\theta)}(x_0)|} = \omega. $$
 
 \medskip
 
 2. Let $\omega \in S^{N-1}$, by $\Omega$ strictly star-shaped with respect to $0$ there exists $t>0$ such that $t\omega \in \partial\Omega$ and $s\leq t$ implies $s \omega \in \Omega$, $s>t$ implies $s \omega \in \Omega^C$.  We want to show that 
 $$ t \leq |x_0|+4\rho. $$
 Let $H$ be such that $x_0 \in H_+$ and $\frac{\phi_H(x_0)}{|\phi_H(x_0)|}= \omega$, which is possible by part 1 of the proof above.  Then by star-shapedness $\lambda x_0 \in \overline{\Omega}^C$ for $\lambda \in (1,+\infty)$ and the analagous statement for the reflected domain, 
 $$\lambda \phi_H(x_0) \in \overline{\phi_H(\Omega_+)}^C \text{ for } \lambda \in (1,+\infty).$$
 Moreover since $\Omega$ has $\rho$-reflection $\Omega_{-} \subseteq \phi_H(\Omega_+)$ so in particular
 $$ \frac{t}{|\phi_H(x_0)|}\phi_H(x_0) = t\omega \ \in \ \overline{\Omega_-} \subseteq \overline{\phi_H(\Omega_+)}$$
 so $ t \leq |\phi_H(x_0)|$.  Now, noting that $-\rho\nu_H \in H$ due to $H$ being tangent to $B_\rho(0)$,
 \begin{align*}
 \phi_H(x_0) = x - 2\langle x_0+\rho\nu_H, \nu_H\rangle\nu_H = (x_0-\langle x_0, \nu_H\rangle\nu_H)-(2r+\langle x_0, \nu_H\rangle)\nu_H
 \end{align*} 
 so by pythagoras
 \begin{align*}
 |\phi_H(x_0)|^2 &=  |x_0-\langle x_0, \nu_H\rangle\nu_H|^2+(2r+\langle x_0, \nu_H\rangle)^2 \\
 & = |x_0-\langle x_0, \nu_H\rangle\nu_H|^2+|\langle x_0, \nu_H\rangle|^2+4\rho^2+4\rho\langle x_0, \nu_H\rangle \\
 & = |x_0|^2 +4\rho^2+4\rho\langle x_0, \nu_H\rangle
 \end{align*} 
 rearranging (and noticing that $\langle x_0, \nu_H\rangle >- \rho$ due to our assumption that $x_0 \in H_+$),
 \begin{align*} 
 t\leq |\phi_H(x_0)| &\leq |x_0|+4\frac{\rho^2+\rho\langle x_0, \nu_H\rangle}{|\phi_H(x_0)|+|x_0|}
 \\
 &\leq |x_0|+4\frac{\rho^2+\rho|x_0|}{|\phi_H(x_0)|+|x_0|}
 \\
 &\leq |x_0|+4\frac{2r}{1+|\phi_H(x_0)|/|x_0|}
 \\
 & \leq |x_0|+4\rho,
 \end{align*}
 where we have used $|x_0|\geq \rho$ and $|\phi_H(x_0)|\geq |x_0|$ in the last two lines.  This completes the proof.
 \end{proof}

\begin{replem}{lem: almost radiality}
Suppose $\Omega$ has $\rho$-reflection. Then $\Omega$ satisfies the following:

\begin{itemize}
\item[(a)] for all $x\in \partial\Omega$ there is an exterior cone to $\Omega$ at $x$,
\begin{equation}
\begin{array}{lll}
 x+C\left(x,\phi_x\right) \subset \real^N\setminus\Omega & \text{ where } & \cos\phi_x = \frac{\rho}{|x|}, \ \phi_x \in (0,\pi/2),
 \end{array}
 \end{equation}
 and $C(x,\phi_x)$ is the cone in direction $x$ of opening angle $\phi_x$ as defined in \eqref{eqn: cone def};
\item[(b)] $\Omega\in \mathcal{S}_{r}$ where
$$ r = r(\rho,\inf_{x\in \partial\Omega}|x|) = (\inf_{x\in \partial\Omega}|x|^2-\rho^2)^{1/2}.$$
\end{itemize}
\end{replem}

\begin{proof}
 Let $x_0\in \partial\Omega$ and let $\mathcal{R}$ be the collection of planes which pass through $x_0$ and are admissible for reflection i.e.
$$ \mathcal{R} = \{H \text{ hyperplane in $\mathbb{R}^N$}: \ x_0 \in H \text{ and } \ H\cap \overline{B_{\rho}(0)} = \emptyset\}.  $$
Notice that if $H \in S$ then for any $a>0$ the plane $H+a\nu_H \subset H_-$ so $H+a\nu_H$ is also admissible for reflection. We can also characterize $\mathcal{R}$ as hyperplanes $H \subset \mathbb{R}^N$ such that $\inf_{y\in H} |y| \geq \rho $ and $x_0 \in H$,
$$ \rho \leq \inf_{y\in H} |y| = |\langle x_0,\nu_H\rangle|.$$
So if we think of these planes as being indexed by their normal vectors,
$$ \mathcal{R} = \{H: \langle \nu_H,x_0 \rangle \geq \rho \text{ and } x_0 \in H\}. $$
 Now let $y \in x_0+C(x_0,\phi_{x_0})$ as above, then from the definition of $\phi_{x_0}$,
$$ \langle \frac{y-x_0}{|y-x_0|},x_0 \rangle > \rho$$
so the plane $H_y$ with normal $\nu_{H_y} = \frac{y-x_0}{|y-x_0|}$ through $x_0$ is in $\mathcal{R}$.  Now we will use the reflection comparison with respect to the plane
$$ \widehat{H} = H_y+\frac{1}{2}(y-x_0)$$
which is admissible for reflection by the remark above that $\widehat{H} \subset H_{y-}$. Then from the $\rho$-reflection property of $\Omega$,
$$ y = \phi_{\widehat{H}}(x_0) \in \mathbb{R}^N\setminus\phi_{\widehat{H}}(\Omega \cap \widehat{H}_+) \subseteq \mathbb{R}^N\setminus(\Omega \cap \widehat{H}_-)$$
and therefore $y \in \real^N\setminus\Omega$.  Since $y$ was arbitrary,
$$ C\left(x_0,\phi_{x_0}\right) \subset \mathbb{R}^N\setminus\Omega.$$
This proves (a),  (b) follows from (a) and  Lemma \ref{lem: sstolip}.
\end{proof}
Conversely we can show that given a domain $\Omega$ which is uniformly close to a ball around zero and has some uniform condition on the directions of its normal vectors (i.e. star-shapedness with respect to a large ball) has $\rho$-reflection.  So $\rho$-reflection is a natural condition to describe closeness to being round.

\begin{lem}\label{lem: lip norm closeness}
Let $\Omega \in \mathcal{S}_{0}$, $n(x)$ be normal to $\partial\Omega$, and let $\rho >0$ such that $B_\rho(0) \subset \Omega$.  Suppose that for $\mathcal{H}^{N-1}$ almost every $x \in \partial\Omega$
\begin{equation}\label{eqn: rho ref normal condition}
 |\langle n(x), x\rangle|^2 \geq |x|^2-\rho^2/5 
 \end{equation}
and also,
\begin{equation}\label{eqn: rho ref osc condition}
 \sup _{x,y \in \partial\Omega} |x|^2-|y|^2 \leq \rho^2 
 \end{equation}
then $\Omega$ has $\rho$-reflection.
\end{lem}
\begin{rem}
The conditions \eqref{eqn: rho ref normal condition} and \eqref{eqn: rho ref osc condition} should be interpreted together as a smallness requirement on the Lipschitz norm distance between $\partial\Omega$ and the nearest sphere or alternatively that $\Omega \in \mathcal{S}_{r,R}$ for $r$ sufficiently large depending on $R$ and $\rho$.
\end{rem}
\begin{proof}
 We first prove the result when $\partial\Omega$ is $C^1$.  Then derive the result for nonsmooth $\Omega$ by density.  Let $n(x) \in C(\real^N;\real^N)$ be a vector field normal to $\partial\Omega$. Let $\nu \in S^{N-1}$ and define
$$H = \{x: x\cdot \nu = 0\}, \quad H_+ = \{x: x\cdot \nu > 0\}, \quad H_- = \{x: x\cdot \nu < 0\}$$ 
the hyperplane normal to $\nu$ through the origin and corresponding half spaces.  For $s>0$ define the translates of $H$,
$$ H(s) = H+s\nu, \quad H_+(s) = H_++s\nu, \quad H_-(s) = H_-+s\nu.$$
We want to show that for $s \geq \rho$
\begin{equation}\label{eqn: C1 bdry containment}
 \phi_{H(s)}(\Omega) \cap H_-(s) \subseteq \Omega\cap H_-(s).
\end{equation}
We make the following notations for simplicity,
\begin{equation*}
\begin{array}{lll}
 \Omega_s := \Omega\cap H_-(s), &\text{and}& \widetilde{\Omega}_s :=  \phi_{H(s)}(\Omega) \cap H_-(s).
 \end{array}
 \end{equation*}
From $\Omega \subset B_R(0)$ we know that $H(R)$ does not intersect $\Omega$ and therefore \eqref{eqn: C1 bdry containment} holds for $s \geq R$.  Now move the plane $H$ inward towards the origin until it touchs $\Omega$ for the first time at
$$ s_0 := \inf \{ s>0: H(s) \cap \Omega = \emptyset\}.$$
Then the containment \eqref{eqn: C1 bdry containment} holds until $s_{\text{min}}(\nu)$ which we call $t$ for convenience.  We want to show that $t \leq \rho$. Note that $t \leq s_0$ since for $s\geq s_0$ the intersection $\Omega\cap H_-(s)$ is empty and so \eqref{eqn: C1 bdry containment} holds trivially.  Moreover $t<s_0$ because $\partial\Omega$ is $C^1$.  At $t$ there are two possibilities. The first is that $\widetilde{\Omega}_{t}$ touches $\Omega_{t}$ from the inside at some point off of $H(t)$, that is there exists $x \in \partial \Omega \cap H_+(t)$ such that:
$$\phi_{H(t)}(x) \in \partial\Omega_{t}\cap  \partial\widetilde{\Omega}_t\setminus H(t).$$
The second is that $H(t)$ intersects $\partial\Omega$ at a point where $\nu$ is tangent to $\partial\Omega$.
 
\medskip

In the first case, call $y = \phi_{H(t)}(x)$ to be the point in $\partial\Omega_{t}\cap  \partial\widetilde{\Omega}_t\setminus H(t)$. Note that $\phi_H(n(x))$ is the normal to $\partial\Omega$ at $y$.  Initially we suppose that
$$ \langle y , \nu \rangle \leq \frac{t}{4}. $$
Then a simple calculation, or some geometry, shows that:
$$  t = \tfrac{1}{2}(\langle x,\nu \rangle+\langle y , \nu \rangle). $$
Now we derive a bound for $t$ from above in this situation,
\begin{align*}
 |x|^2 &= |x-y|^2+|y|^2+2\langle x-y,y\rangle   = \langle x-y , \nu \rangle^2+|y|^2+2|x-y|\langle \nu , y \rangle \\
 & = (2t-2\langle y,\nu \rangle)^2+|y|^2-4(t-\langle y,\nu \rangle)\langle \nu , y \rangle  \\
 &=4t^2-12 t\langle \nu , y \rangle+8\langle \nu , y \rangle^2+|y|^2 \\
 & \geq t^2+|y|^2. 
  \end{align*}
Keeping this in mind we now work in the case when 
$$ \langle y , \nu \rangle > \frac{t}{4}. $$
Because $x$ lies to the in $H_+(\nu)$ and we have assumed a bound on the angle between $x$ and $n(x)$ we also get a bound on the angle between $n(x)$ and $\nu$,
$$ \langle n(x) , \nu \rangle  = \langle n(x) -\widehat{x},\nu\rangle + \frac{t}{|x|} \geq -2\left(1-\sqrt{1-\frac{\rho^2}{5|x|^2}}\right)+ \frac{t}{|x|} \geq  \frac{t}{|x|}-\frac{\rho^2}{5|x|^2}.$$
Meanwhile, $\phi_H(n(x))$ inherits the opposite bound,
$$\langle \phi_H(n(x)),\nu \rangle = -\langle n(x) , \nu \rangle \leq -\frac{t}{|x|}+\frac{\rho^2}{5|x|^2}.$$
Due to our assumption on the angle between $y$ and $\nu$ we also get a bound in the other direction,
\begin{align*}
\langle \phi_H(n(x)),\nu \rangle  &= \langle \phi_H(n(x)) -\widehat{y},\nu\rangle+\langle\widehat{y},\nu\rangle  > -\frac{\rho^2}{5|y|^2}+\frac{t}{4|y|}.
\end{align*}
Combining the above two bounds we get that,
$$ t \leq \frac{\rho ^2 |x|}{5|y|^2}+\frac{\rho^2}{5|x|}-\frac{t|x|}{4|y|},$$
then calling $\gamma = |x|/|y|$ and using $\rho \leq \min\{|x|,|y|\}$ we rearrange to get,
$$ t \leq \frac{1}{5}\left(\rho \frac{\rho}{|y|} \frac{\gamma}{1+\tfrac{1}{4}\gamma}+\rho \frac{\rho}{|x|}\right) \leq \rho.$$

\medskip

In the second case, let $x \in \partial\Omega \cap H(t)$ such that,
$$\langle n(x) , \nu \rangle= 0,  \ \ \hbox{ and therefore } \ \  |\langle x , \nu \rangle|^2+|\langle x , n(x) \rangle|^2 \leq |x|^2,$$
Rearranging and noting that $x \in H(t)$ implies $|\langle x, \nu \rangle| = t$,
$$ t \leq \left(|x|^2-(|x|^2-\rho^2))\right)^{1/2}= \rho. $$

\medskip

Finally in the case of general $\Omega \in \mathcal{S}_0$ one can approximate $\partial\Omega$ by boundaries of $C^1$ domains which converge in Lipschitz norm (in the appropriately interpreted sense) and use the fact the $\rho$-reflection is preserved by convergence in Hausdorff distance.

\end{proof}

\begin{lem}\label{lem: hausdorff estimates}
Let $\Omega_j$ for $j\in \{1,2\}$ and $U$ be in $\mathcal{S}_{r,R}$, and let $\alpha=\alpha(r,R) \in (0,1)$ from the H\"{o}lder estimates for harmonic functions in Lipschitz domains Lemma~\ref{lem: holder in lip domain}. Then the following estimates hold:
\begin{align} 
 d_H(\Omega_1,\Omega_2) &\leq d_H(\partial\Omega_1,\partial\Omega_2) \label{eqn: hausdorff bdry est}
\\  \label{hausdorff to lambda estimate} |\lambda[\Omega_1]-\lambda[\Omega_2]|&\lesssim_{r} d_H(\Omega_1,\Omega_2)
\\ \label{hausdorff to u estimate}
 ||u[\Omega_1]-u[\Omega_2]||_\infty &\lesssim_{r,R} d_H(\Omega_1,\Omega_2)^{\alpha}
\\ \label{hausdorff to L1 estimate}
 |\Omega_1 \Delta \Omega_2|&\lesssim_{r,R} d_H(\Omega_1,\Omega_2)  
\\ \label{hausdorff to E estimate}
 |\mathcal{J}(\Omega_1)-\mathcal{J}(\Omega_2)|&\lesssim_{r,R} d_H(\Omega_1,\Omega_2)^{\alpha}
\\ \label{hausdorff to tildedist estimate}
 |\widetilde{\textup{dist}}^2(\Omega_1,U)- \widetilde{\textup{dist}}^2(\Omega_2,U)| \ \ \text{ and }  &\ \  |\widetilde{\textup{dist}}^2(U,\Omega_1)- \widetilde{\textup{dist}}^2(U,\Omega_2)| \ \ \lesssim_{r,R}d_H(\Omega_1,\Omega_2) 
  \end{align}
 \end{lem} 

\begin{proof}
We will start by showing \eqref{hausdorff to L1 estimate}, the only necessary assumption on the $\Omega_j$ is that one of their boundaries be rectifiable with $\mathcal{H}^{N-1}$ Hausdorff measure bounded in terms of $r$ and $R$,
$$ |\Omega_1\Delta\Omega_2|\leq d_H(\Omega_1,\Omega_2)\mathcal{H}^{N-1}(\partial\Omega_1)\lesssim_{r,R}d_H(\Omega_1,\Omega_2).$$
Then \eqref{hausdorff to tildedist estimate} follows easily from \eqref{hausdorff to L1 estimate},
\begin{align*}
 |\widetilde{\text{dist}}^2(U,\Omega_1)-\widetilde{\text{dist}}^2(U,\Omega_2)|  &=\left| \int_{\Omega_1\Delta U}d(x,\partial U)dx-\int_{\Omega_2\Delta U}d(x,\partial U)dx\right| 
 \\
 & \lesssim_R |\Omega_1 \Delta \Omega_2| 
 \end{align*}
 and because the $\widetilde{\text{dist}}^2$ is asymmetric we also check
 \begin{align*}
 |\widetilde{\text{dist}}^2(\Omega_1,U)-\widetilde{\text{dist}}^2(\Omega_2,U)|  &=\left| \int_{\Omega_1\Delta U}d(x,\partial \Omega_1)dx-\int_{\Omega_2\Delta U}d(x,\partial\Omega_2)dx\right| 
 \\
 & \lesssim_R |\Omega_1 \Delta \Omega_2| +\int_{(\Omega_1\cap\Omega_2)\Delta U }|d(x,\partial\Omega_1)-d(x,\partial\Omega_2)|dx 
 \\
 &\lesssim_R d_H(\Omega_1,\Omega_2)+\int_{(\Omega_1\cap\Omega_2)\Delta U } d_H(\Omega_1,\Omega_2)dx \\
 & \lesssim_R d_H(\Omega_1,\Omega_2)
 \end{align*}
 where in the third line we have used that if $y_1\in \partial\Omega_1$ there exists $y_2 \in \partial\Omega_2$ with $|y_1-y_2| \leq d_H(\Omega_1,\Omega_2)$ and therefore $d(x,\partial\Omega_2) \leq d(x,\partial\Omega_1)+d_H(\Omega_1,\Omega_2)$ and vice versa.  Note that \eqref{hausdorff to E estimate} is an easy consequence of \eqref{hausdorff to u estimate} and \eqref{hausdorff to lambda estimate},
 $$|\mathcal{J}(\Omega_1)-\mathcal{J}(\Omega_2)|\leq ||u[\Omega_1]-u[\Omega_2]||_\infty\max\{|\Omega_1|\,,\,|\Omega_2|\}+|\Omega_1\Delta\Omega_2|,$$
 so we are left to prove the first three estimates.
 
 \medskip
 
 First recall the behavior of $\lambda$ under spatial dilations,  
 \begin{equation}\label{dilations}
 \lambda[a\Omega] = a^{-(N+2)}\lambda[\Omega].
 \end{equation}
    We will show that taking,
    \begin{equation}\label{dilation factor}
   a:= \frac{r}{d_H(\partial\Omega_1,\partial\Omega_2)+r}
   \end{equation}
    one gets the following containments, 
   \begin{equation}\label{dilation containments}
   a\Omega_1\subseteq\Omega_2\subseteq \frac{1}{a}\Omega_1.
   \end{equation}
    This will imply by \eqref{dilations} that
    $$ a^{-(N+2)}\lambda[\Omega_1] \geq \lambda[\Omega_2] \geq  a^{N+2}\lambda[\Omega_1]$$
    and by rearranging,
    $$ |\lambda[\Omega_1]-\lambda[\Omega_2]| \leq \max\{\lambda[\Omega_1],\lambda[\Omega_2]\}(1-a^{N+2}).$$
    Now to get an estimate of the form \eqref{hausdorff to lambda estimate}, for $t\geq0$ note that 
    $$ 1-(1+t)^{-(N+2)}\leq (N+2)t$$ 
    to get
    $$|\lambda[\Omega_1]-\lambda[\Omega_2]| \leq r^{-1}(N+2)\lambda[B_r]d_H(\partial\Omega_1,\partial\Omega_2).$$
   
   \medskip
  
    Let us prove \eqref{dilation containments} for $a$ chosen as in \eqref{dilation factor}. If $x \in \Omega_1\setminus\Omega_2$ let $t \in (0,1)$ such that $tx \in \partial\Omega_2$.   Then from the star-shaped property there is an exterior cone to $\Omega_2$ at $tx$,
    \begin{equation*}
    \begin{array}{lll}
     tx+C\left(x,\theta_{tx}\right) \subset \mathbb{R}^N\setminus\Omega_2 & \text{ where } & \sin\theta_{tx} = \frac{r}{t|x|}, \ \theta_{tx} \in (0,\pi/2)
     \end{array}
     \end{equation*}
    so that 
    $$B_{\frac{(1-t)}{t}r}(x) \subset tx+C\left(x,\theta_{tx}\right) \subset \mathbb{R}^N\setminus\Omega_2.$$
    Then $d(x,\partial\Omega_2)\geq \left(\frac{1}{t}-1\right)r$ and moreover,
    $$ \left(\frac{1}{t}-1\right)r \leq d_H(\partial\Omega_1,\partial\Omega_2).$$
    or by rearranging,
    $$  a\leq t.$$
    Then by making the same argument for $x \in \Omega_2\setminus\Omega_1$ we get
    $$a \Omega_1 \subseteq \Omega_2 \subseteq \frac{1}{a}\Omega_1. $$
    
    \medskip
    
Now we show \eqref{hausdorff to u estimate}. Without loss of generality suppose that $\lambda[\Omega_1]\geq\lambda[\Omega_2]$.  Consider the difference of the droplet profiles on the two domains,
$$w(x) = u[\Omega_1](x)-u[\Omega_2](x). $$
First we estimate the size of $w(x)$ on $\Omega_1\Delta\Omega_2$.  If $x \in \Omega_1\Delta\Omega_2$ let $h = d(x,\partial(\Omega_1\cup\Omega_2))$ then
$$ h \leq d_H(\Omega_1,\Omega_2) \,\,\text{ and } \,\, B_h(x) \subset \Omega_1\cup\Omega_2 \,\,\text{ with } \,\, \partial B_h(x) \cap \partial(\Omega_1\cup\Omega_2) \neq \emptyset.$$
By the standard construction of barriers for domains with the exterior cone property there exist $0<\alpha<1$ and $C>0$ depending on the uniform lower bound on the opening angle of the exterior cones such that
$$ |w(x)| \leq Ch^{\alpha}+ \frac{\lambda[\Omega_1]}{2N}h^2.$$
So on $\Omega_1\Delta\Omega_2$ and in particular on $\partial (\Omega_1 \cap \Omega_2)$ we have that
$$ |w(x)| \lesssim_{r,R} \max\{d_H(\Omega_1,\Omega_2)^{\alpha},d_H(\Omega_1,\Omega_2)^2\}\lesssim_{r,R}d_H(\Omega_1,\Omega_2)^{\alpha}, $$
where the last inequality is due to the fact that $d_H(\Omega_1,\Omega_2) \leq 2R$.  Finally on intersection $x \in \Omega_1\cap \Omega_2$,
$$ |w(x)| \lesssim_{r,R}d_H(\Omega_1,\Omega_2)^{\alpha}+|\lambda[\Omega_1]-\lambda[\Omega_2]|\lesssim_{r,R}d_H(\Omega_1,\Omega_2)^{\alpha}.$$

  Finally we show \eqref{eqn: hausdorff bdry est}.  Let $x_0 \in \Omega_1\setminus\Omega_2$, then $tx \in \real^N\setminus\Omega_2$ for all $t>1$.  In particular let $s>1$ such that $x_1=sx_0 \in \partial\Omega_1$.  We claim that:
\begin{equation}\label{eqn: dist ineq}
h:=d(x_0,\Omega_2) \leq d(x_1,\Omega_2). 
\end{equation}
Let $x \in B_h(x_1)$, then consider the point $s^{-1}x$,
$$|s^{-1}x-x_0| = s^{-1}|x-x_1| \leq |x-x_1| = h. $$
Therefore $s^{-1}x \in B_h(x_0) \subset \real^N\setminus\Omega_2$ and so from the star-shapedness of $\Omega_2$ the point $x = ss^{-1}x$ is as well, this proves the claimed inequality \eqref{eqn: dist ineq}.  In particular:
$$ \sup_{x \in \Omega_1} d(x,\Omega_2) = \sup_{x \in \Omega_1\setminus\Omega_2}d(x,\partial\Omega_2) \leq \sup_{x\in\partial\Omega_1} d(x,\partial\Omega_2), $$
and by noting that $\Omega_1$ and $\Omega_2$ play symmetric roles:
$$ d_H(\Omega_1,\Omega_2) \leq d_H(\partial\Omega_1,\partial\Omega_2). $$
\end{proof}

Define the metric space of boundaries of strongly star-shaped sets,
$$\partial\mathcal{S}_{r,R} := \{\partial \Omega : \Omega \in \mathcal{S}_{r,R}\} $$
with metric $d_H$.  This space embeds continuously into
$$ \{ f \in C^{0,1}(S^{N-1}) : r \leq f \leq R \text{ and } ||Df||_{L^{\infty}} \leq C(r,R)\} $$
with the $L^{\infty}$ distance.  As a direct consequence of this we get the following compactness for $\partial\mathcal{S}_{r,R}$:
\begin{lem}\label{lem: compactness}
\textup{(Compactness)} The metric space $(\partial\mathcal{S}_{r,R},d_H)$ is compact:
\begin{enumerate}[(i)] 
\item Suppose that $\Gamma_j \in (\partial\mathcal{S}_{r,R},d_H)$ for some $r,R>0$ and all $j \in \mathbb{N}$. Then $\{\Gamma_j\}_{j\in \mathbb{N}}$ has a subsequence that converges and any subsequential limit is also in $\partial\mathcal{S}_{r,R}$.
\item Let $I$ be a compact interval in $\real$ and $\Gamma_j : I \to (\partial\mathcal{S}_{r,R},d_h)$ for $j \in \mathbb{N}$ are an equicontinuous sequence of paths in $(\partial\mathcal{S}_{r,R},d_h)$. Then there is a subsequence of the $\Gamma_j(\cdot)$ that converges uniformly on $I$ to a path $\Gamma: I \to (\partial\mathcal{S}_{r,R},d_h)$.
\end{enumerate}
\end{lem}
\begin{proof}
Arzela-Ascoli.  The only interesting point is that $\partial\mathcal{S}_{r,R}$ is closed in $(\mathcal{K},d_H)$ where $\mathcal{K}$ is the class of compact subsets of $\real^N$.  It suffices to establish the estimate:
\begin{equation}
 d_H(\partial\Omega_1,\partial\Omega_2) \lesssim_{r,R} d_H(\Omega_1,\Omega_2) 
 \end{equation}
where $\Omega_j \in \mathcal{S}_{r,R}$.  Let $x \in \partial\Omega_1$ and $y \in \partial\Omega_2$ such that:
$$ |x-y|=d(x,\partial\Omega_2) = d_H(\partial\Omega_1,\partial\Omega_2). $$
If $x \in \real^N\setminus \Omega_2$ then,
$$d_H(\partial\Omega_1,\partial\Omega_2)=|x-y| = d(x,\Omega_2) \leq d_H(\Omega_1,\Omega_2). $$
Otherwise $x  \in \Omega_2$, then let $t>0$ such that $(1+t)x \in \partial\Omega_2$. From the strong star-shapedness of $\Omega_1$ we know that
$$d((1+t)x,\Omega_1) \geq t|x| \frac{r}{|x|} \geq |(1+t)x-x|\frac{r}{R}\geq |x-y|\frac{r}{R}.$$
Rearranging we get that
$$ d_H(\partial\Omega_1,\partial\Omega_2) =|x-y| \leq \frac{R}{r}d((1+t)x,\Omega_1) \leq \frac{R}{r}d_H(\Omega_1,\Omega_2). $$
\end{proof}

\section{Sup and Inf Convolutions}\label{sec: convolutions}

For the comparison principle the convolutions in space only give spatial regularity of the free boundary at a point where a sup-convolved subsolution touches an inf-convolved supersolutions from below.  To fix this convolutions over space-time ellipsoids are used.  These were introduced in \cite{CaffarelliVazquez99} for viscosity solutions of the porous medium equation.  Let $u$ be a subsolution of \eqref{eqn: CLML} and $v$ a supersolutionon $\mathbb{R}^N\times[0,+\infty)$, let $r>0$ and $c \geq 0$.  Define the sup-convolution of $u$
\begin{equation}\label{spacetimesupconv}
u^{r,c}(x,t)  = \sup_{|(y,s)-(x,t)|\leq r-c t} u(y,s),
\end{equation}
and the inf-convolution of $v$
\begin{equation}\label{spacetimeinfconv}
v_{r,c}(x,t)  = \inf_{|(y,s)-(x,t)|\leq r -c t} v(y,s).
\end{equation}
For simplicity of presentation and since we want to emphasize the $c = 0 $ case separately we define
$$ \overline{u}(x,t) = u^{r,0}(x,t) \quad \text{ and } \quad \underline{v}(x,t) = v_{r,c}(x,t).$$
As for the convolutions in space only, $c$ will boost the speed of the free boundary of $\underline{v}$.
\begin{lem}\label{lem: spacetimeconvs}
(Space-time convolutions)  
\begin{enumerate}[(a)]
\item Let $\overline{u}$ as defined above the sup-convolution of $u$.
\begin{enumerate}[(i)]
\item
  $\overline{u}$ is a subsolution of \eqref{eqn: CLML} in $\mathbb{R}^N\times [r,+\infty)$ with $\widetilde{\lambda}(t) = \sup_{|s-t|\leq r} \lambda(s)$.  
  \item
  The positive phase $\Omega(\overline{u})$ has the interior ball property. More specifically, let $P_0 = (x_0,t_0) \in \partial\Omega(\overline{u})$ then the positivity set of $u$ has an exterior ball of radius $r$ centered at $P_0$,
$$ \mathcal{B}_r(P_0) := \{(y,s) \in \mathbb{R}^N\times[0,+\infty): |(y,s)-(x_0,t_0)| <r\}\subset \Omega(\overline{u})^C$$
 and there exists $P_1 \in \partial\mathcal{B}_r(P_0)\cap\partial\Omega(\overline{u})$.  Then $ \mathcal{B}_r(P_1)$ touches $ \Omega_+(\overline{u})$ from the inside at $P_0$. 
 \end{enumerate} 
 \item Let $\overline{v}$ as defined above the inf-convolution of $v$.
\begin{enumerate}[(i)]
\item
  $\overline{v}$ is a supersolution of \eqref{eqn: CLML} in $\mathbb{R}^N\times [r,r/c)$ with $\widetilde{\lambda}(t) = \inf_{|s-t|\leq r-ct} \lambda(s)$.  
  \item
  The positive phase $\Omega(\underline{v})$ has an exterior ellipsoid at every point of its free boundary. More specifically, let $P_0 \in \partial\Omega_+(\underline{v})$ then $\mathcal{B}_{r-c t_0}(P_0)$ touches the free boundary of $v$ from the inside i.e. 
  $$\mathcal{B}_{r-c t_0}(P_0) \subseteq \Omega(v) \quad \text{and there exists} \quad P_1 \in \partial\mathcal{B}_{r-c t_0}(P_0)\cap\partial\Omega(v).$$
  The ellipsoid 
  $$ E_{r,\delta}(P_1) = \{(y,s): |(y,s)-(x_1,t_1)|<r-c s\}$$
  touches $\Omega(\underline{v})$ from the outside at $P_0$.
  \item Let $H_1$ be the hyperplane orthogonal to $P_1-P_0$ passing through $P_1$, i.e. the tangent hyperplane to $ \mathcal{B}_{r-c t_0}(P_0)$ at $P_1$ and $(\nu,m)\in \mathbb{R}^n\times \mathbb{R}_+$ be the normal vector to $H_1$ rescaled so that $|\nu| = 1$.  We call $m$ the advance speed of the free boundary of $v$ at $P_1$.  This terminology is justified when $v \in C^{2,1}(\mathbb{R}^n\times\mathbb{R}_+)$.   Let $H_0$ be the tangent hyperplane to $E_{r,\delta}(P_1)$ at the point $P_0$ and $(\nu ',m')$ it's normal vector rescaled as before.  Then $m'=m+\delta$, the inf-convolution boosts the advance speed of the free boundary by $\delta$.
    \end{enumerate} 
\end{enumerate}
\end{lem}

\section{Higher Regularity for the Equilibrium Problem}
In this appendix we describe the method in\cite{DeSilva09} and \cite{KN} to show that if $u$ is a solution of the equilibrium problem
\begin{equation}\tag{\ref{eqn: EQ}}
\left \{ \begin{array}{lll}
         -\Delta u(x) = \lambda &\text{ in }& \Omega= \Omega(u), \\ \\
|Du|=1 &\text{ on } &\Gamma=\Gamma(u). 
         \end{array}\right.
\end{equation}
with $\Omega \in \mathcal{S}_0$ then $\Gamma$ is $C^{2,\alpha}$.  The first step is Theorem \ref{thm: de silva} from \cite{DeSilva09} which says that $\Gamma$ is $C^{1,\alpha}$.  As a consequence of this $u \in C^{1,\alpha}(\overline{\Omega})$.  Then one uses a transformation of the problem first used for this purpose by \cite{KN} that has the effect of flattening the free boundary while making the PDE and the boundary condition nonlinear. Below we describe the transformation in detail.

\medskip

We may assume that $0 \in \Gamma$ and $1=|Du|(0) = \partial_n u(0)$.  Then near $0$ we make the change of variables:
\begin{equation*}
\begin{array}{lll}
y_j =x_j  \ \ \text{ for } \ \ 1\leq j \leq n-1, &
y_n = u, &
v = x_n,
\end{array}
\end{equation*}
Then we can write the $y$ derivatives of $v$ in terms of the $x$ derivatives of $u$ so that $v(y)$ solves the problem:
\begin{align}
\begin{cases}
\sum_{j=1}^n \partial_j(F_j(Dv))= \lambda \ \  \text{ for } \ y_n>0 \\
(\partial_nv)^2-\sum_{j=1}^{n-1}(\partial_jv)^2 = 1 \ \ \text{ on } \ \ y_n =0
\end{cases}
\end{align}
where 
\begin{align}
\begin{cases}
F_j(p) = \dfrac{p_j}{p_n} \ \ \text{ for } \ \ 1\leq j\leq n-1, \\
F_n(p) = -\frac{1}{2}\dfrac{1+\sum_{j=1}^{n-1}(p_j)^2}{(p_n)^2}.
\end{cases}
\end{align}

It is straightforward to check that this operator is uniformly elliptic when $\partial_nv$ is bounded from below. This fact as well as the nondegeneracy of $\partial_n v$ from the boundary condition enables the argument in \cite{KN} to go through for our problem and yield the desired regularity result in Corollary~\ref{smoothness}.


\medskip

Let us show that $(C.1)$ is the correct equation.  Based on the definitions of $y_j$ and $v$, we have
\begin{equation}\label{identity}
 u(y_1,...,y_{n-1},v(y_1,...,y_n)) = y_n.
 \end{equation}
   Taking the $n$-th partials in \eqref{identity} we get
 $$ \partial _nu=\dfrac{1}{\partial_nv}, \quad  -\partial_{nn}u = \dfrac{\partial_{nn}v}{(\partial_n v)^3} \ \ \text{ and } \ \ \partial_{nj}u =  \dfrac{\partial_{nn}v \, \partial_j v}{(\partial_n v)^3}-\dfrac{\partial_{nj}v}{(\partial_n v)^2}$$
  where any derivative of $u$ is assumed to be evaluated at $(y_1,...,y_{n-1},v(y_1,...,y_n))$ and any derivative of $v$ at $(y_1,...,y_n)$.  Similarly taking derivatives in \eqref{identity} for $1\leq j \leq n-1$, we have $ \partial_ju+\partial_nu\,\partial_jv = 0$ and thus 
 \begin{align*} 
 -\partial_{jj}u&=2\partial_{nj}u\,\partial_jv+\partial_{nn}u(\partial_jv)^2+\partial_nu\,\partial_{jj}v  
 \\
 &=  \dfrac{\partial_{nn}v \, (\partial_j v)^2}{(\partial_n v)^3}-\dfrac{\partial_{nj}v\,\partial_j v}{(\partial_n v)^2}+\dfrac{\partial_{jj}v}{\partial_n{v}}.
 \end{align*}
 Then we substitute into (EQ) to get
 \begin{align*} 
 \lambda &= -\sum_{j=1}^{n-1} \partial_{jj} u -\partial_{nn}u = \sum_{j=1}^{n-1} \left[\dfrac{\partial_{nn}v \, (\partial_j v)^2}{(\partial_n v)^3}-2\dfrac{\partial_{nj}v\,\partial_j v}{(\partial_n v)^2}+\dfrac{\partial_{jj}v}{\partial_n{v}}\right]+\dfrac{\partial_{nn}v}{(\partial_n v)^3}
 \\
 & = \sum_{j=1}^{n-1}\dfrac{\partial_{jj}v}{\partial_n{v}}+\left(\dfrac{1+\sum_{j=1}^{n-1}(\partial_jv)^2}{(\partial_n v)^3}\right)\partial_{nn}v-2\sum_{j=1}^{n-1}\dfrac{\partial_{nj}v\,\partial_j v}{(\partial_n v)^2}
 \\
 & = \sum_{j=1}^{n-1}\partial_j\left(\frac{\partial_jv}{\partial_nv}\right)-\partial_n\left(\frac{1}{2}\dfrac{1+\sum_{j=1}^{n-1}(\partial_jv)^2}{(\partial_nv)^2}\right).
 \end{align*}
 The calculation of the boundary condition is easier
 $$ 1= |Du|^2 = \sum_{j=1}^{n-1} \left(\frac{\partial_jv}{\partial_nv}\right)^2+\frac{1}{(\partial_n v)^2}.$$

\bibliographystyle{plain}
\bibliography{DropletStabilityArticles}
\end{document}